\documentclass[aap,preprint]{imsart}

\RequirePackage{amsthm,amsmath,amsfonts,amssymb,mathrsfs,stmaryrd,stackrel,physics}
\RequirePackage[numbers]{natbib}
\RequirePackage[colorlinks,citecolor=blue,urlcolor=blue]{hyperref}
\RequirePackage{graphicx,color,epsfig,subfigure}
\RequirePackage{bm,bbm}

\startlocaldefs

\newtheorem{theorem}{Theorem}[section]
\newtheorem{proposition}{Proposition}[section]

\newtheorem{lemma}[theorem]{Lemma}
\newtheorem{corollary}[theorem]{Corollary}
\theoremstyle{remark}
\newtheorem{remark}{\emph{ Remark}}
\newtheorem{definition}[theorem]{Definition}

\newcommand{\Prob}[1]{\mathbb{P}\left[#1\right]}

\newcommand{\Exp}[1]{\mathbb{E}\left[#1\right]}
\newcommand{\Var}[1]{\mathbb{V}\left[#1\right]}

\newcommand{\mdd}{\mathrm{d}}

\newcommand{\ee}{\mathrm{e}}
\newcommand{\ii}{\mathrm{i}}
\newcommand{\cadlag}{\emph{c\`adl\`ag} }
\newcommand{\tF}{\widetilde{F}}
\newcommand{\tf}{\widetilde{f}}
\newcommand{\tg}{\widetilde{g}}

\newcommand{\tT}{\widetilde{\Theta}}
\newcommand{\tG}{\widetilde{G}}
\newcommand{\tGamma}{\widetilde{\Gamma}}
\newcommand{\mId}{\mathcal{I}_\delta}
\newcommand{\mIz}{\mathcal{I}_0}
\newcommand{\mIn}{\mathcal{I}_{\delta_n}}

\endlocaldefs

\begin{document}

\begin{frontmatter}
\title{Physical blowups via buffered time change in a mean-field neural network}
\runtitle{Buffered blowup solutions}

\begin{aug}

\author[A,B]{\fnms{Nikolaos} \snm{Papadopoulos}\ead[label=e1]{nikospap@austin.utexas.edu}}
\and
\author[A,B]{\fnms{Thibaud} \snm{Taillefumier}\ead[label=e2]{ttaillef@austin.utexas.edu}}

\address[A]{Department of Mathematics,
University of Texas, Austin,
}

\address[B]{Department of Neuroscience,
University of Texas, Austin,
}
\end{aug}

\begin{abstract}
Idealized networks of integrate-and-fire neurons with impulse-like interactions obey McKean-Vlasov diffusion equations in the mean-field limit. 
These equations are prone to blowups: for a strong enough interaction coupling, the mean-field rate of interaction diverges in finite time with 
a finite fraction of neurons spiking simultaneously, thereby marking a macroscopic synchronous event.
Characterizing these blowup singularities analytically is the key to understanding the emergence and persistence of spiking synchrony in mean-field neural models. 
Such a treatment is possible via time change techniques for a Poissonian variation of the classically considered integrate-and-fire dynamics.
However, just as for shock solutions for  nonlinear conservation laws, there are several admissible blowup solutions to the corresponding McKean-Vlasov equations.
Because of this ambiguity, it is unclear which notion of blowup solutions shall be adopted. 
Here, we unambiguously define physical blowup dynamics as solutions to a fixed-point problem bearing on the time change associated to the McKean-Vlasov equation.
To justify the physicality of this formulation, we introduce a buffering  mechanism that regularizes blowup dynamics, while satisfying the same conservation principles
as the finite-dimensional particle-system dynamics.
We then show that  physical blowup dynamics are recovered from these regularized solutions in the limit of vanishing buffering.
Our approach also shows that these physical solutions are unique, globally defined, and avoid the so-called eternal blowup phenomenon, as long 
as the neurons exhibit nonzero, distributed refractory periods, a reasonable modeling assumption. 
\end{abstract}



\end{frontmatter}


\section{Introduction}


\subsection{Background} 

This work considers neural network models for which the emergence of synchrony can be studied analytically in the idealized, mean-field limit of infinite-size networks.
By synchrony, we refer to the possibility that a finite fraction of the network's neurons simultaneously spikes.
Dynamics exhibiting such synchrony can serve as models to study the maintenance of precise temporal information in neural networks \cite{Panzeri:2010,Kasabov:2010,Brette:2015}.
The maintenance of precise temporal information in the face of neural noise remains a debated issue from an experimental and computational perspective.
Mathematical approaches to understand synchrony involve making simplifying assumptions about the individual neuronal processing as well as about the network supporting their interactions.

Integrate-and-fire neurons \cite{Lapicque:1907,Knight:1972} constitute perhaps the simplest class of models susceptible to displaying synchrony \cite{Taillefumier:2013uq}.
In integrate-and-fire models, the internal state of a neuron $i$ is modeled as a continuous-time diffusive process $X_{i,t}$, whose dynamics stochastically integrates past neural interactions.
Spiking times are then defined as first-passage times of this diffusive process to a spiking boundary $L$. 
Upon spiking, the process $X_{i,t}$ resets to a base real value $\Lambda$.
In other words, $X_{i,t}=\Lambda$ whenever neuron $i$ spikes at time $t$.
There is no loss of generality in assuming that  $L=0$ and $\Lambda>0$, so that $X_{i,t}$ has nonnegative state space.
We consider integrate-and-fire models for which $X_{i,t}$ follows a Wiener diffusive dynamics with negative drift $-\nu<0$ \cite{Carrillo:2015}.
Drifted Wiener processes are the simplest diffusive dynamics for which spikes occur in finite time with probability one, even in the absence of interactions.

A key feature of integrate-and-fire models is that they allow for the occurrence of synchronous spiking events.
This is most conveniently seen by considering a finite neural network with instantaneous, homogeneous, impulse-like excitatory interactions.
For such interactions, if neuron $i$ spikes at time $t$, downstream neurons $j \neq i$ instantaneously update their internal states according to $X_{j,t} = X_{j,t^-} - w$, where $ w>0$ is the size of the impulse-like interaction.
Thus, the spiking of neuron $i$ in $t$ causes the states of all other neurons to move toward the zero spiking boundary, leading to two possible outcomes for downstream neuron $j$: 
either $X_{j,t}>w$ and the update merely hastens the next spiking time of neuron $j$, or  $X_{j,t} \leq w$ and the interaction causes neuron $j$ to spike in synchrony with $i$ \cite{Taillefumier:2012uq}.
The latter synchronous spiking events occur with finite probability, as we generically have $\Prob{X_{j,t^-} \in (0,w]}>0$ for regular diffusion processes.
In turn, the synchronous spiking of downstream neuron $j$ can trigger additional synchronous spiking events in the network, via branching processes referred to as spiking avalanches.
Spiking avalanches are well-defined under the modeling assumptions that neurons transiently enter a post-spiking refractory state and always exit this refractory state by reseting to $\Lambda$  \cite{Taillefumier:2014fk}.
Under such assumptions, neurons can spike at most once within an avalanche and synchronously spiking neurons can be distinguished according to their generation number \cite{Delarue:2015}.

Tellingly, the finite probability to observe a spiking avalanche is maintained in certain simplifying limits, such as the thermodynamic mean-field limit \cite{Amari:1975aa, Faugeras:2009,Touboul:2014,Renart:2004}.
For a homogeneous, excitatory, integrate-and-fire network, the thermodynamic mean-field limit considers a network of $N$ exchangeable neurons in the infinite-size limit, $N\to \infty$, with vanishingly small impulse size $w_N = \lambda/N \to 0$, where $\lambda$ is a parameter quantifying the interaction coupling.
In this mean-field limit, individual neurons only interact with one another via a deterministic population-averaged firing rate $f(t)$ \cite{Caceres:2011,Carrillo:2013,Carrillo:2015}.
Specifically, the dynamics of a representative process $X_t$ obeys a nonlinear partial differential equation (PDE) of the McKean-Vlasov type
\begin{eqnarray}\label{eq:classicalMV}
\partial_t p = \big(\nu + \lambda f(t) \big)  \partial_x p + \partial_x^2 p/2  + f(t^-) \delta_\Lambda   \quad  \mathrm{with} \quad p(t,x) \, \mdd x =\Prob{X_t \in \mdd x} \, ,
\end{eqnarray}
where the effective drift features the firing rate $f(t)$.
The nonlinearity of the above equation stems from the conservation of probability, which equates $f(t)$ with a boundary flux of probability:
\begin{eqnarray}\label{eq:classicalFlux}
 f(t) = \partial_x p(t,0)/2 \, .
\end{eqnarray}
In the following, we refer to \eqref{eq:classicalMV} and \eqref{eq:classicalFlux} as the classical McKean-Vlasov (cMV) equations and to the underlying dynamics supporting these equations as the classical mean-field (cMF) dynamics.
Within the setting of cMF dynamics, a blowup occurs at time $T_0$ if the spiking rate diverges when $t \to T_0^-$ and a synchronous event happens if a  fraction of the neurons $0<J_0 \leq 1$ synchronously spikes in $T_0$.


\subsection{Motivation} 

Following on seminal computational work in \cite{Brunel:1999aa,Brunel:2000aa}, the cMF dynamics was first investigated in a PDE setting by C\'aceres {\it et al.} \cite{Caceres:2011}, who established the occurrence of blowups.
The existence and regularity of solutions to \eqref{eq:classicalMV} and \eqref{eq:classicalFlux} have been considered by several authors \cite{Delarue:2015,Delarue:2015b,Hambly:2019,Nadtochiy:2019,Nadtochiy:2020}.
These authors combined results from the theory of interacting-particle systems \cite{Liggett:1985,Sznitman:1989} and of the convergence of probability measures \cite{Billingsley:2013} to establish criteria for the existence of global solutions \cite{Delarue:2015,Delarue:2015b} and to classify the type of singularities displayed by these solutions \cite{Hambly:2019,Nadtochiy:2019,Nadtochiy:2020}.
However, the analytical characterization of blowup singularities have proven rather challenging.
Here, we consider a modified interacting-particle system with Poisson-like attributes that is also prone to blowup, the so-called delayed Poissonian mean-field (dPMF) model.
By contrast with \cite{Delarue:2015b} and in line with \cite{Caceres:2011,Carrillo:2013}, we only conjecture that the propagation of chaos holds to 
motivate the form of the corresponding mean-field PDE problem.
This conjecture is numerically supported in the weak interaction regime  $\lambda < \Lambda$ and for the strong interaction regime $\lambda > \Lambda$.

The motivation for our rigorous treatment of synchrony via dPMF dynamics is two-fold:

First, our approach allows for the consideration of more realistic neural dynamics than those classically considered for analysis.
In practice, realistic neural dynamics are heuristically derived as mean-field limits of Poissonian network models in the diffusion approximation~\cite{Brunel:1999aa,Brunel:2000aa}. 
Perhaps the most studied examples of these diffusion-based models are the so-called balanced networks~\cite{Vreeswijk:1996aa,Doiron:2014}.
These models obey systems of PDEs of the form
\begin{eqnarray*}\label{eq:neuroMV}
\tau \partial_t p_\ee &=&  \partial_x (x + \mu_\ee(t) )p_\ee + \sigma_\ee^2 \partial_x^2 p_\ee/2  + f_\ee(t^-) \delta_\Lambda  \quad  \mathrm{with} \quad  f_\ee(t) = \partial_x p_\ee(t,0)/2  \, , \\
\tau  \partial_t p_\ii &=&  \partial_x ( x + \mu_\ii(t)  )p_\ii +  \sigma_\ii^2\partial_x^2 p_\ii/2  + f_\ii(t^-) \delta_\Lambda    \quad  \mathrm{with} \quad  f_\ii(t) = \partial_x p_\ii(t,0)/2 \,  ,
\end{eqnarray*}
where the subscripts $\ee$ and $\ii$ refer to the neurons being excitatory or inhibitory, respectively.
In the above systems, biophysical modeling specifies the drift and diffusion coefficients for  a neuron of type $\alpha$ in $\{\ee,\ii\}$ as
\begin{eqnarray*}\label{eq:neuroMV}
\mu_\alpha(t) = w_{\alpha} \nu_\alpha + \sum_{\beta \in \{ \ee, \ii\}} w_{\alpha\beta} c_\beta f_\beta(t) 
\quad
\mathrm{and}
\quad
\sigma_\alpha(t) = w_{\alpha}^2  \nu_\alpha + \sum_{\beta \in \{ \ee, \ii\}} w^2_{\alpha\beta} c_\beta f_\beta(t) \, ,
\end{eqnarray*}
where $w_{\alpha\beta}$ denotes effective coupling weights with $w_{\ee\ee}, w_{\ii\ee} \geq 0$ and $w_{\ii\ii}, w_{\ee\ii} \leq 0$ and where the coefficients $c_\ee, c_\ii \geq 0$ measures the relative density of recurrent excitatory and inhibitory inputs, respectively.
Thus, in balanced models, both the drift and diffusion coefficients $\mu_\alpha$ and $\sigma_\alpha$ depend linearly on the firing rates $f_\ee$ and $f_\ii$.
This dependence indicates that neurons experience interactions as if delivered via Poisson arrival processes, at least in the conjectured mean-field dynamics.
For this reason, we refer to the joint linear dependence of $\mu_\alpha$ and $\sigma_\alpha$ on $f_\ee$ and $f_\ii$ as Poissonian attributes of the dynamics.
Owing to these Poissonian attributes, the rigorous study of balanced models requires one to consider the framework of dPMF dynamics rather than that of cMF dynamics.

Second, our approach aims to solve the problem of defining physical solutions with blow-ups.
Just as for shock solutions in nonlinear conservation PDEs, there is in principle many formally admissible solutions with blowup for dPMF dynamics.
We recently proposed one such class of admissible solutions in~\cite{TTPW} for dPMF dynamics with full blowup, whereby a finite fraction of neurons fires synchronously.  
Almost at the same time, Dou \emph{et al.} leveraged results from the theory of free-boundary problems to consider a larger class of equations, including dPMF dynamics, and proposed a conflicting method of resolution for blowups~\cite{dou2022dilating}.
However, the latter method can lead to unphysical behavior such as eternal blowups, whereby one cannot terminate---or rather integrate---a blowup episode.
This is by contrast with the expected behavior of physical particle systems for which blowups must represent integrable singularities, as the fraction of synchronously firing neurons cannot exceed one.
Our goal here is to confirm the validity of the blowup resolution presented for dPMF dynamics in~\cite{TTPW} in the sense that the thus-obtained solutions are the only ones that satisfy the same conservation principles as those obeyed by the finite-size particle systems.

%
%
%

%
%


\subsection{Results}

Our main contribution is to specify physical solutions for a class of mean-field models for neural network dynamics,
which we refer to as delayed Poissonian mean-field (dPMF) models.
The PDEs governing these models have been considered by prior works
~\cite{Carrillo:2013,Carrillo:2015,dou2022dilating,dou2024noisy} and conjectured 
to hold under the assumption of propagation of chaos. 
Such an assumption is supported by numerical simulations 
and related results obtained for classical mean-field (cMF) models~\cite{Delarue:2015b}.
The specific class of dPMF dynamics considered here are  as follows:

\begin{definition}\label{def:mainProb1} 
\begin{subequations}\label{eq:mainProb1}
Given parameters $\lambda_1, \lambda_2, \nu_1, \nu_2, \Lambda, \epsilon>0$, the McKean-Vlasov equation 
governing the density functions $\mathbbm{R}^+ \times \mathbbm{R}^+ \ni (t,x) \mapsto p(t,x)=  \Prob{X_t \in \mdd x} / \mdd x$ of the dPMF dynamics $X_t$
is defined as the probability-conserving boundary value problem
\begin{eqnarray}
&\partial_t p  = \big(\nu_1 + \lambda_1 f(t) \big) \partial_x p + \big(\nu_2 + \lambda_2 f(t)  \big) \partial^2_{x} p /2  +  f_\epsilon(t) \delta_{\Lambda}   \, , &  \label{eq:mainProb1a}\\
&p(t,0) = 0 \, , &\\
&\partial_t  \left( \int_0^\infty p(t,x) \, \mdd x \right) = f_\epsilon(t) -  f(t) \, , & \label{eq:mainProb1c}
\end{eqnarray}
with firing rate function $f: \mathbbm{R}^+ \to  \mathbbm{R}^+$ and with delay reset rate given as the convolution $f_\epsilon = f * p_\epsilon$, where $p_\epsilon$ is a smooth probability density with support in $[\epsilon,\infty)$.
\end{subequations}
\end{definition}

The dynamics defined above differ from mean-field dynamics 
considered in other works by the inclusion of distributed refractory delays. 
This inclusion represents a reasonable modeling assumption
as real neurons engage in refractory periods after spiking.
At the same time, including i.i.d. refractory delays with a smooth 
density $p_\epsilon$ (as opposed to considering, e.g., a single delay $\epsilon$ with $p_\epsilon=\delta_\epsilon$) 
regularizes the reset rate of dPMF dynamics, 
which  offers several mathematical advantages as we will see.
Specifically, such a regularization will allow us to use classical fixed-point arguments
to show the existence and uniqueness of globally defined blowup solutions 
and that these solutions cannot exhibit unphysical eternal blowups.
Additional constraints may be imposed on the delay distribution $p_\epsilon$,
which is always assumed smooth with support in $[\epsilon, \infty)$, $\epsilon>0$,
 to reflect that \eqref{eq:mainProb1} emerges as a mean-field limit.
For instance, we expect that propagation of chaos only holds 
if $p_\epsilon$ has finite first-order moment, and we may
assume that $p_\epsilon$ has finite moments of all orders. 
That said, the analysis  presented in this manuscript 
will only assume that $p_\epsilon$  has finite first and second moments.

Definition~\ref{def:mainProb1} does not cover leaky neural dynamics, 
for which the drift term includes a linear component in $x$.
Although including such a term in the analysis is possible~\cite{Carrillo:2013,dou2022dilating},
this inclusion does not impact the emergence of blowups, which is 
the main focus of this work.
In fact, the emergence of blowups in dPMF dynamics is a diffusion phenomenon 
rather than a drift one.
This can be seen by realizing that conservation of probability implies that $f$, the instantaneous
firing rate of a representative neuron governed by \eqref{eq:mainProb1} 
is related to $\partial_x p(t,0)$, the slope of the density at the spiking boundary via:
\begin{eqnarray}\label{eq:fluxBound}
f(t) = \frac{ \nu_2 \partial_x p(t,0)}{2 -  \lambda_2 \partial_x p(t,0)} \, .   
\end{eqnarray}
The above relation reveals that blowups will occur whenever the slope
$\partial_x p(t,0)$ attains the threshold value $2/\lambda_2$, 
so that the emergence of blowups is controlled
by the diffusion terms rather than the drift ones in \eqref{eq:mainProb1}.
Given this relation, we introduce the notion of non-explosive initial conditions
as a matter of convenience  to avoid  immediate blowups:

\begin{definition}\label{def:mainProb2} 
\begin{subequations}\label{eq:mainProb2}
Nonexplosive initial conditions are specified as $p(0,x) =  p_0(x)$, $x>0$, and $f(t) = f_0(t)$, $-\infty< t < 0$, where $(p_0,f_0)$ are nonnegative measures in $\mathcal{M}(\mathbbm{R}^+) \times \mathcal{M}(\mathbbm{R}^-) $
satisfying:
\begin{itemize}
\item the integrability condition:
\begin{eqnarray}\label{eq:mainProb2a}
\vert F_0(t) \vert = O(\vert t\vert)  \, , \quad t \to - \infty \, , \quad \mathrm{where} \quad F_0(t) =- \int_{(t,0]} f_0(\mdd s) \leq 0 \, ,
\end{eqnarray}
\item the no-initial-explosion condition:
\begin{eqnarray}\label{eq:mainProb2b}
\partial_x p(0) =  \lim_{x \downarrow 0} \frac{p_0((0,x])}{x} <\frac{2}{\lambda_2} \, ,
\end{eqnarray}
\item  the normalization condition:
\begin{eqnarray}\label{eq:mainProb2c}
\int_{0}^\infty p_0(x) \, \mdd x + \int_{-\infty}^0 (1-P_\epsilon(t)) \, \mdd  F_0( t) = 1 \, .
\end{eqnarray}
\end{itemize}
\end{subequations}
\end{definition}

Given the above notion of initial conditions, the central result of this manuscript is to show that physical blowup solutions to the dPMF problem stated 
in Definition~\ref{def:mainProb1} are uniquely defined as follows:


\begin{theorem}\label{th:mainRes} 
\begin{subequations}
\label{all:mainRes}
Given nonexplosive initial conditions, the physical solutions of the McKean-Vlasov equation \eqref{eq:mainProb1} are specified as $(t,x) \mapsto p(t,x)=q(\Phi(t),x)$
where $\mathbbm{R}^+\times \mathbbm{R}^+ \ni (\sigma,x) \mapsto q(\sigma,x)$ and $\Psi=\Phi^{-1}: \mathbbm{R}^+ \to \mathbbm{R}^+$ uniquely solve the system of equations
\begin{eqnarray}\label{eq:mainTCdyn}
\partial_\sigma q 
&=&
-\mu(\sigma)  \partial_x q + \frac{1}{2} \partial^2_x q + \partial_\sigma G_\epsilon(\sigma)  \delta_\Lambda \,   ,  \quad  q(\sigma,0) = 0 \, , \label{eq:mainTCdyn1}\\
\mu(\sigma) 
&=& 
 -\left(\left( \nu_1 - \frac{\lambda_1}{\lambda_2} \nu_2   \right)\Psi'(\sigma)  +  \frac{\lambda_1}{\lambda_2} \right) \, , \label{eq:mainTCdyn2}\\
G_{\epsilon}(\sigma)
&=& 
\int_0^\sigma \frac{1}{2}  \left( \int_0^\sigma  \partial_x q(\tau-\eta, 0) \, \mdd \tau \right)   \, \mdd P_\epsilon \big(\Psi(\sigma) - \Psi(\sigma-\eta)\big) \label{eq:mainTCdyn3} \\
&& +
\int_\sigma^\infty   F_0(\Psi_0(\sigma-\eta))  \, \mdd P_\epsilon \big(\Psi(\sigma) - \Psi_0(\sigma-\eta)\big) \nonumber
\, , \\
\Psi(\sigma)
&=&
\frac{1}{\nu_2} \sup_{0 \leq \tau \leq \sigma} \left( \tau-\frac{\lambda_2}{2}  \left( \int_0^\sigma  \partial_x q(\tau, 0) \, \mdd \tau \right) \right)  \, . \label{eq:mainTCdyn4}
\end{eqnarray}
where we have defined $\Psi_0= \Phi_0^{-1}: \mathbbm{R}^- \to  \mathbbm{R}^-$  with
$\Phi_0(t)=\nu_2 t - \lambda_2 \int_{(t,0]} f_0(s) \, \mdd s$, $t \leq 0$.
\end{subequations}
\end{theorem}

The above theorem will directly follow from Proposition~\ref{prop:solG} and Proposition~\ref{prop:physFP}.
Proposition~\ref{prop:solG} states that assuming $\Psi$ known, 
the system of equations \eqref{eq:mainTCdyn1},  \eqref{eq:mainTCdyn2},  \eqref{eq:mainTCdyn3} admits
a unique solution $(\sigma,x) \mapsto q(\sigma,x)$ on $\mathbbm{R}^+\times \mathbbm{R}^+$.
Proposition~\ref{prop:physFP} complements this result by characterizing $\Psi$ self-consistently as the unique
solution to \eqref{eq:mainTCdyn4}, which defines a fixed-point equation for $\Psi$ once one
recognizes that $q=q[\Psi]$ depends on $\Psi$ by Proposition~\ref{prop:solG}.

Rigorously deriving \eqref{eq:mainTCdyn4} and establishing its physicality is the core contribution of the manuscript. 
By contrast with other approaches that treat McKean-Vlasov equations similar 
to \eqref{eq:mainProb1} as a free-boundary analysis problem,
our approach leverages probabilistic arguments to treat \eqref{eq:mainProb1} via the lens of renewal theory.
Taking such an approach is the key to formulate the fixed-point equation \eqref{eq:mainTCdyn4}.
Specifically, in the absence of blowups, one can check that upon performing the change of variable
\begin{eqnarray*}
\sigma=\Phi(t) \Leftrightarrow t = \Psi(\sigma)
\quad \mathrm{with}
\quad
\Phi(t)
=
\nu_2 t + \lambda_2  \int_0^t f(s) \, \mdd s \, , 
\end{eqnarray*}
nonexplosive dPMF dynamics satisfy  \eqref{all:mainRes} but with \eqref{eq:mainTCdyn4} reading
\begin{eqnarray}\label{eq:smoothRel}
\Psi(\sigma)
=
\frac{1}{\nu_2} \left( \sigma-\frac{\lambda_2}{2}  \left( \int_0^\sigma  \partial_x q(\sigma, 0) \, \mdd \tau \right) \right)  \, .
\end{eqnarray}
Simple renewal theory arguments show that the above relation cannot hold in the presence of full blowups, 
when $\Phi$ exhibits jump discontinuities, or equivalently, when $\Psi$ has flat sections.
Furthermore, examining probabilistically solvable cases obtained for $\lambda_1/\lambda_2=\nu_1/\nu_2$
reveals that during full blowups, \eqref{eq:smoothRel}  specifies $\Psi$ as a locally decreasing function,
which is paradoxical for any time change.
Our insight is to solve this paradox by restricting the fixed-point equation 
specifying $\Psi$ on the set of nondecreasing functions
via the introduction of a running supremum operator in \eqref{eq:mainTCdyn4}.
However, \eqref{eq:mainTCdyn4} only represents one possible notion of blowup solutions
and other notions have been proposed (see~\cite{dou2022dilating}).
Theorem~\ref{th:mainRes} states that the fixed-point problem resulting from adopting \eqref{eq:mainTCdyn4}
admits a unique global solution on the whole time half-line
and that this solution is the only physical one in a precise sense that we discussed in the next section.



\subsection{Methodology and innovation}

In order to show that the solution defined in Theorem~\ref{th:mainRes} represents a physical blowup dPMF dynamics, 
our approach is to introduce a regularized version of the PDE problem stated in Definition~\ref{def:mainProb1}.
We obtain this regularized dPMF problem by imposing that the rate of interactions cannot exceed some level $1/\delta$, with $\delta>0$.
A simple way to implement this rate limitation is to distinguish the firing rate $f$ 
from the rate of interaction, now denoted $\tf $, with which a representative particle experiences interactions.
The main innovation of our approach is to define $\tf$ via a buffer mechanism, and for this reason, 
we refer to the thus obtained regularized dynamics as $\delta$-buffered dynamics:

\begin{definition}\label{def:mainBuffer1} 
\begin{subequations}
\label{eq:distribBuffPDE}
Given nonexplosive initial conditions $(p_0,f_0)$ in $\mathcal{M}(\mathbbm{R}^+) \times \mathcal{M}(\mathbbm{R}^-)$ and $\delta>0$, the probability-conserving McKean-Vlasov equation governing the $\delta$-buffered dPMF density function $\mathbbm{R}^+ \times \mathbbm{R}^+ \ni (t,x) \mapsto p(t,x)= \mdd \Prob{0<X_t \leq x} / \mdd x$ is defined as the boundary value problem
\begin{eqnarray}
&\partial_t p  =\big(\nu_1 + \lambda_1 \tf (t) \big) \partial_x p + \big(\nu_2 + \lambda_2 \tf (t)  \big) \partial^2_{x} p /2  + f_\epsilon(t) \delta_{\Lambda}  \, ,  & \\
& \quad p(t,0) = 0 \, , & \\
&\partial_t  \left( \int_0^\infty p(t,x) \, \mdd x \right) = f_\epsilon(t) -  f(t) \, , &
\end{eqnarray}
where $\tf$, the $\delta$-buffered rate of the firing rate $f$ ensuring probability conservation, is specified in Definition~\ref{def:mainBuffer2}. 
\end{subequations}
\end{definition}

Fully specifying the regularized dPMF dynamics given in Definition~\ref{def:mainBuffer1} 
requires to define the buffer mechanism by which $\tf$ is obtained:



\begin{definition}\label{def:mainBuffer2} 
\begin{subequations}
\label{all:mainBuffer}
The the $\delta$-buffered version $\tf$ of the firing rate $f$ is specified as 
\begin{eqnarray}\label{eq:mainBuffer}
\tf (t)= (1-B(t)) f(t) + B(t) /\delta \, ,
\end{eqnarray}
where the blowup indicator function $B$ is given by
\begin{eqnarray}\label{eq:mainIndicator}
B(t)  = \mathbbm{1}_{\{ f(t) > 1/\delta \} \cup \{E(t) > 0\}}    \, .
\end{eqnarray}
and where the interaction excess function $E$ satisfies
\begin{eqnarray}\label{eq:mainExcess}
\partial_t E(t)  &=&  B(t) (f(t) -1/\delta)  \, , \quad  E(0)=0\, .
\end{eqnarray}
\end{subequations}
\end{definition}

The buffer mechanism described in Definition~\ref{def:mainBuffer2}  is designed 
to satisfy a physical conservation principle that is obeyed by the finite-size particle systems.
Indeed, according to Definition~\ref{def:mainBuffer2}, the capped rate $\tf $ is obtained by buffering the firing rate $f$,
 i.e., by diverting the excess of interactions to a fictitious reservoir state with rate $f-1/\delta$.
Buffered blowups are then defined as these periods of time when the reservoir is not empty, i.e., when it contains some excess amount of interactions $E>0$.
The hallmark of our proposed buffer mechanism is that it is conservative:  excess interactions must be purged from the reservoir and restituted to the representative particle for the blowup to terminate when $E=0$.
This is done by imposing that as long as $E>0$, excess interactions must be restituted with fixed rate $1/\delta$.

\begin{figure}[htbp]
\begin{center}
\includegraphics[width=\textwidth]{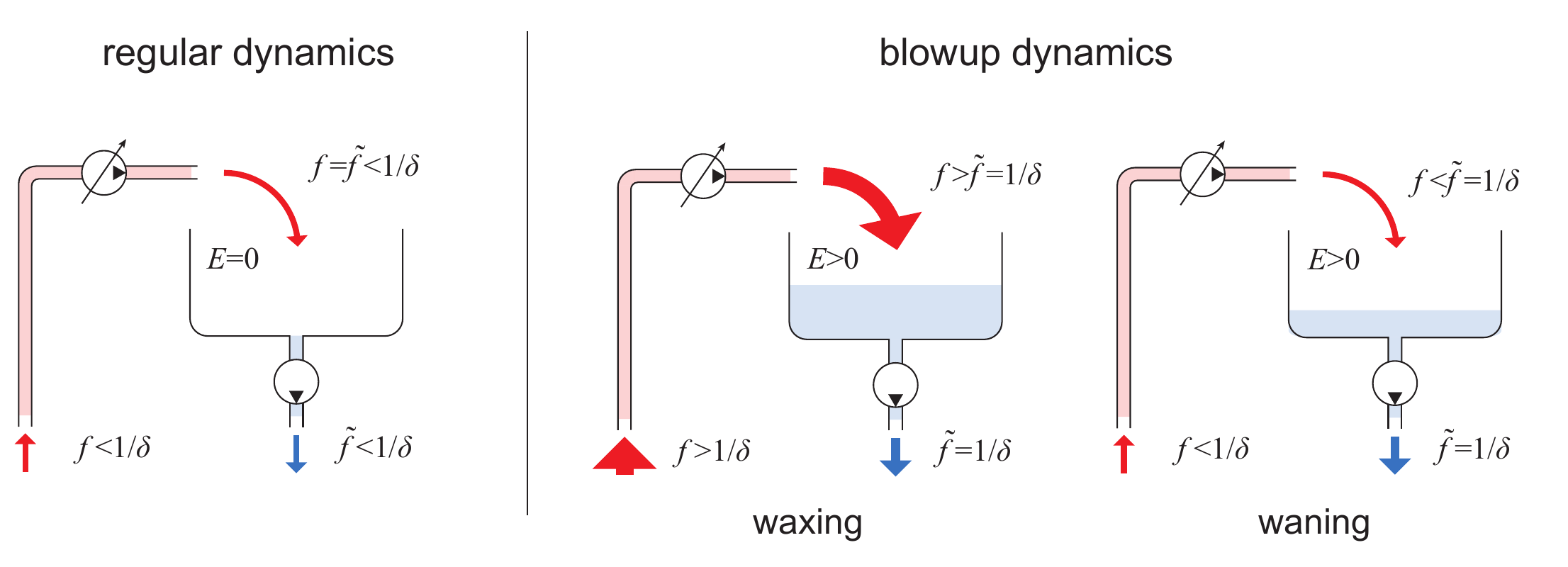}
\caption{{\bf Physical picture of the buffer mechanism.}
It is convenient to think of the buffer mechanism in terms of a fictitious pumping physical system whereby water flows represent interaction rates between particles.
The buffer mechanism is implemented by a reservoir which is fed by a pump that operates at the spiking rate $f$.
At the same time, the reservoir is depleted by a pump that operates at a capped 
interaction rate $\tf$, which cannot exceed the buffer value $1/\delta$.
During regular dynamics, interactions are experienced by particles at the same rate $\tf$ as the rate $f$ at which these particles emit them.
There is no water in the reservoir, indicating the absence of blowups.
Blowups are triggered when the inlet rate $\tf$ transiently exceeds $1/\delta$, leading to water accumulation in the reservoir $E>0$ (waxing period).
The outlet pump operates at capacity, i.e., at rate $1/\delta$, as long as there is water in the reservoir, even though the inlet rate may be below $1/\delta$ (waning period).
The defining property of such a physical system is to conserve interaction rates.}
\label{fig:buffMech}
\end{center}
\end{figure}

The physical picture inspiring  Definition~\ref{def:mainBuffer1} and Definition~\ref{def:mainBuffer2} for $\delta$-buffered 
dPMF dynamics is illustrated in Fig.~\ref{fig:buffMech}.
Although this illustration exploits an analogy with a pumping physical system,
its governing abstract principle is that of rate conservation.
Imposing this principle ensures that every nonspiking neuron must be impacted
by spiking neurons, as is the case in the corresponding finite-dimensional particle systems.
In other word, there is no loss of interactions.
Systems that do not conserve interactions can be subjected to paradoxical effects
such as eternal blowups, which are unphysical~\cite{dou2022dilating}.
Our main result is to show that the dPMF solutions specified in Theorem~\ref{th:mainRes2}
represent physical blowup solutions in the sense that they are limits
of the  rate-conserving, $\delta$-buffered dPMF dynamics obtained when $\delta \downarrow 0$. 
Specifically, we show the following theorem:

\begin{theorem}\label{th:mainRes2} 
The unique inverse time change $\Psi=\Phi^{-1}: \mathbbm{R}^+ \to \mathbbm{R}^+$ parametrizing 
the dPMF dynamics solution from Theorem~\ref{th:mainRes} is specified as $\Psi = \lim_{\delta \downarrow 0} \Psi_\delta$, where $\Psi_\delta$
is the unique inverse time change parametrizing the $\delta$-buffered version of the dPMF dynamics.
\end{theorem}

The above theorem directly follows from Theorem~\ref{th:deltaGlobSol} about the existence and uniqueness 
of $\delta$-buffered dPMF dynamics and from Proposition~\ref{prop:GHEta} about
the limit behavior of these solutions when $\delta \downarrow 0$.
To establish these results, we study the regularizing effect of the buffer mechanism 
proposed  in this manuscript. 
In particular, we give a characterization of the cumulative $\delta$-buffered
rate functions in terms of the $\delta$-tilted running extrema of the cumulative unbuffered rate, 
both in the original and in the time-changed picture.
Given $\delta>0$, this characterization allows us to relate the time change functions $\Phi_\delta$
and $\Psi_\delta=\Phi_\delta^{-1}$ to the excess functions $E_\delta$ and $D_\delta = E_\delta \circ \Psi_\delta$, respectively.
It turns out that while taking $\delta \downarrow 0$ negates buffering in the original picture,
the buffer mechanism, and its regularizing action, persist in the time-changed picture.
This latter point allows us to justify that physical dPMF blowup dynamics are uniquely specified as solutions 
to the PDE problem given in Theorem~\ref{th:mainRes}.
This result follows from deriving the fixed-point equation \eqref{eq:mainTCdyn4}, which involves a running supremum operator,
from the characterization of the buffer mechanism in terms of $\delta$-tilted running extrema.

It is unclear whether the buffered mechanism we originally introduced in this manuscript 
may be used to analyze other singularly behaved PDE problems. 
However, we expect similarly inspired buffer mechanisms to be useful 
for defining blowup solutions to PDE problems that obey rate conservation principles.
At the very least, we expect our approach to generally apply to the coupled McKean-Vlasov equations
describing the mean-field dynamics of neural networks with heterogeneous structure. 
In support of the possible general scope of our approach, 
we note that our buffered blowup solutions are defined by introducing 
a new physical variable, captured by the excess function $E$, 
an approach that is reminiscent of the introduction of entropic conditions
to specify shock solutions to conservation laws~\cite{Evans:2010}.


\subsection{Structure of the manuscript}

In Section~\ref{sec:modelDef}, we introduce the McKean-Vlasov equations governing dPMF dynamics under conjecture of propagation of chaos. 
In Section~\ref{sec:physTC}, we show the existence and uniqueness of our candidate solutions for blowup DPMF dynamics.
In Section~\ref{sec:bufferAnalysis}, we introduce the physical notion of buffered dPMF solutions, which are uniquely defined and prone to regularized blowups.
In Section~\ref{sec:limBuff}, we show that our candidate solutions are physical for being specified as limit buffered dPMF solutions. 
In Section~\ref{sec:compRes}, we provide complementary results to explain the regularizing role of the buffer mechanism and refractory delays. 
In Section~\ref{sec:proofs}, we provide various proofs supporting the results presented in the preceding sections.


\section{Background on delayed Poisson-McKean-Vlasov dynamics}\label{sec:modelDef}

In this section, we justify the statements of Definitions~\ref{def:mainProb1} and~\ref{def:mainProb2},
which are the focus of this work.
In Section~\ref{sec:FSSM}, we introduce the finite-dimensional particle system from which dPMF dynamics are derived.
In Section~\ref{sec:propChaos}, we conjecture propagation of chaos to exhibit the stochastic dynamics governing a representative dPMF particle.
In Section~\ref{sec:MKV}, we justify that the Fokker-Plank equation satisfied by the density of a representative dPMF particle 
specifies the same McKean-Vlasov problem as defined in Definitions~\ref{def:mainProb1} and~\ref{def:mainProb2}.


\subsection{Finite-size stochastic model}\label{sec:FSSM}

We start by defining the finite-size version of the dPMF dynamics in terms of a particle system whose dynamics is unconditionally well-posed.
This particle system consists of a network of $K$ interacting processes $X_{K,i,t}$, $1\leq i \leq K$, whose  dynamics is as follows:
$(i)$ Whenever a process $X_{K,i,t}$ hits the spiking boundary at zero, it spikes and instantaneously enters an inactive refractory state.
$(ii)$ At the same time, all the other active processes $X_{K,j,t}$ (which are not in the inactive refractory state) are 
respectively decreased by amounts $w_{K,j,t}$, which are independently drawn from a normal law with mean $\lambda_1/K$, $\lambda_1>0$, and variance $\lambda_2/K$, $\lambda_2>0$.
$(iii)$ After an inactive (refractory) period of duration $s \geq \epsilon$ where $s$ is sampled i.i.d. according to some distribution $P_\epsilon$ with support in $[\epsilon,\infty)$, $\epsilon>0$, the process $X_{K,i,t}$ restarts its autonomous stochastic dynamics from the reset state $\Lambda>0$.
$(iv)$ In between spiking/interaction times, the autonomous dynamics of active processes follow independent drifted Wiener processes with negative drift $-\nu_1$ and positive diffusion coefficient $\nu_2>0$.
Correspondingly, an initial condition for the network is given by specifying the starting values of the active processes, i.e. $X_{K,i,0}>0$ if $i$ is active, and the last hitting time of the inactive processes, denoted by $\tau_{K,i,0} \leq 0$, if $i$ is inactive.

The coupled dynamics of $X_{K,i,t}$, $1\leq i \leq K$, described above are governed by the stochastic differential equations 
\begin{eqnarray}\label{eq:stochEq}
X_{K,i,t} = X_{K,i,0} - \int_0^t \mathbbm{1}_{ \{ X_{K,i,s^-} >0 \} } \mdd Z_{K,i,s} + \Lambda N_{K,i,t} \, , 
\end{eqnarray}
where $Z_{K,i,t}$, $1\leq i \leq K$, denote continuous-time driving processes with Poisson-like attributes and 
where $N_{K,i,t}$, $1\leq i \leq K$, are increasing processes counting the number of times that $X_{K,i,t}$ has reset before $t$.
In \eqref{eq:stochEq}, the driving processes $Z_{K,i,t}$ are specified in terms of the mean empirical counting functions $F_{K,i}$ via
\begin{eqnarray}
Z_{K,i,t} = \nu_1 t+ \lambda_1 F_{K,i}(t)   +  W_{i,\nu_2 t + \lambda_2 F_{K,i}(t)} \, , \quad F_{K,i}(t) =  \frac{1}{K} \sum_{j\neq i} M_{K,j,t} \nonumber \, ,
\nonumber
\end{eqnarray}
where $W_{i,t}$, $1\leq i \leq N$, are independent Wiener processes and where the individual 
counting processes $M_{K,i,t}$, $1\leq i \leq N$, register each time particle $i$ spikes.
In turn, the newly introduced counting processes $M_{K,i,t}$ and the counting processes $N_{K,i,t}$ featured in \eqref{eq:stochEq} are defined as
\begin{eqnarray}
M_{K,i,t} = \sum_{n>0} \mathbbm{1}_{[0,t]}(\tau_{K,i,n}) \quad \mathrm{and} \quad N_{K,i,t} = \sum_{n>0} \mathbbm{1}_{[0,t]}(r_{K,i,n}) \, , \nonumber 
\end{eqnarray}
where $\tau_{K,i,n}$, $n \geq 0$, denote the successive first-passage times of $X_{K,i,t}$ in zero 
and where $r_{K,i,n}=\tau_{K,i,n}+s_{K,i,n}$, $n \geq 0$,
denote the subsequent reset times.
By definition, we consider refractory delays $s_{K,i,n}$ that are i.i.d. according to $P_\epsilon$,
while the times $\tau_{K,i,n}$ are formally defined for all $n\in \mathbbm{N}$ by 
\begin{eqnarray}
\tau_{K,i,n+1} = \inf \left\{  t>r_{K,i,n} \, \big  \vert \,  X_{K,i,t}  \leq 0 \right\} \, . \nonumber 
\end{eqnarray}
By convention, we set $r_{K,i,0} = 0$ for processes that are active in zero with $X_{K,i,0}>0$.  
We say that the driving processes $Z_{K,i,t}$ exhibit Poisson-like attributes because
both their drift and diffusion coefficients are linearly related to the empirical spiking rate of the network.
By contrast, classical particle-system approaches are restricted to models for which 
only the drift coefficient exhibits such dependencies, i.e., to the case $\lambda_2=0$.
Note that by definition, the paths of $t \mapsto M_{K,i,t}$ and $t \mapsto N_{K,i,t}$ are right-continuous with left limits,
 which is commonly referred to as having the \emph{c\`{a}dl\`{a}g} property in the probabilistic terminology.

Thus defined, the particle-system dynamics is self-exciting: 
every spiking event of a neuron $i$  hastens the spiking of other neurons $j \neq i$ by bringing their states closer to the zero threshold boundary.
Moreover, the particle-system dynamics allows for synchronous spiking as whenever neuron $i$ spikes due to its autonomous dynamics, 
we generically have that $\Prob{X_{K,j,t^-} \in (0,w_{K,j,t}]}>0$. 
If the process $X_{K,i,t}$ first hits zero at time $\tau$,  \eqref{eq:stochEq} implies that $X_{K,i,t}$ remains in zero for all $t$ in $[\tau, r) \supset [\tau, \tau+\epsilon)$, until it receives an instantaneous kick that enforces a reset in $\Lambda$ at time $r=\tau+s$, with $s \sim P_\epsilon$.
Thus, \eqref{eq:stochEq} formally identifies the refractory state with zero.
However, it is better to consider the refractory state as an isolated state, where time-stamped inactivated process can be stored before reset.
The introduction of this isolated state will allow one to consider regular absorbing boundary conditions in zero.

\begin{remark} The inclusion of refractory delays is worth a few observations:

\begin{enumerate}
\item Including inactive periods $s \geq \epsilon$ guarantees the uniqueness of the particle-system dynamics 
during spiking avalanches, thereby ensuring that the overall dynamics is well-posed. 
Spiking avalanches occur when the spiking of a neuron triggers the instantaneous spiking of other neurons.
Neurons that engage in a spiking avalanche can be sorted out according to a generation number $g \in \mathbbm{N}$.
Generation $0$ contains the lone triggering neuron which is driven to the absorbing boundary by its autonomous dynamics.
Generation $1$ comprises all those neurons that spike due to interactions with the triggering neuron alone.
In general, generation $g>1$, comprises all the neurons that spike from interacting with the neurons of generation $g-1$ alone.
In the absence of a post-spiking inactive period, it is ambiguous whether the neurons 
from previous generations are impacted by the spiking of neurons from the following generations.
By contrast, in the presence of inactive periods, $s \geq \epsilon>0$, neurons from previous generations 
are unresponsive to neurons from following generations due to post-spiking transient inactivation.
\item
The need to clearly define spiking avalanche was identified by Delarue \emph{et al.} in \cite{Delarue:2015b},
 which led these authors to introduce the notion of physical solution for cMF dynamics that may exhibit blowups. 
These physical solutions assume that independent of their generation number, every neuron engaging 
in a spiking avalanche at time $t$ spikes only once and resets in $\Lambda$.
Thus, this notion of physical solutions precisely corresponds to considering delayed cMF dynamics
 in the limit of vanishing refractory period $\epsilon \to 0^+$.
Note that Delarue \emph{et al.} also considered delayed dynamics in \cite{Delarue:2015b}.
However, these dynamics include interaction delays rather than reset delay and therefore differ from the dPMF considered in this work.
\item Importantly, considering inactive period with $s \geq \epsilon>0$ implies that neurons 
cannot spike more than once over a duration $\epsilon$, so that we must have $F_{N,i,t} >t/\epsilon+1$ for all time $t>0$.
In Section~\ref{sec:limBuff}, we will see that the existence of a uniform bound on the cumulative counts, only available for $\epsilon>0$, 
allows one to exclude the possibility of nonphysical blowups, such as eternal blowups.
Eternal blowups are nonphysical in the sense that they corresponds to blowups with self-sustaining, infinite-size spiking avalanche, 
a phenomenon which precludes to consider dynamics past the time at which an eternal blowup occurs.
This will not be the case for delayed dynamics as considering $\epsilon>0$ guarantees the existence of physical solutions with blowups for all time $t>0$.

\end{enumerate}

\end{remark}

\subsection{Conjectured mean-field dynamics under propagation of chaos}\label{sec:propChaos}

The particle-system dynamics introduced above primarily differs from the classically considered one by its Poisson-like attributes.
In classically defined particle systems, the jump discontinuities of the driving inputs have fixed size $\lambda_1/K$ instead of being i.i.d according to a normal law with mean $\lambda_1/K$ and variance $\lambda_2/K$.
In \cite{Delarue:2015b}, Delarue {\it et al.} showed that the property of propagation of chaos holds in the infinite-size limit of classically defined particle systems.
This property establishes that in the infinite-size limit, a representative process $X_t=\lim_{K \to \infty} X_{K,i,t}$ follows 
a mean-field dynamics satisfying the PDE problem \eqref{eq:classicalMV} and  \eqref{eq:classicalFlux} originally considered in \cite{Caceres:2011,Carrillo:2013}.
This particle-system-based approach automatically yields the existence of---possibly explosive---solutions to the PDE problem  \eqref{eq:classicalMV} and \eqref{eq:classicalFlux}.
Here, by contrast with \cite{Delarue:2015b} and in line with \cite{Caceres:2011,Carrillo:2013}, we only conjecture 
propagation of chaos to motivate the form of the PDE problem considered in this work.
The form of this PDE problem is supported by numerical simulations even when the particle system asymptotically experiences blowups.

The propagation of chaos states that for exchangeable initial conditions, the processes $X_{K,i,t}$, $1 \leq i \leq K$, 
become i.i.d. in the limit of infinite-size networks $K \to \infty$, conditioning to the mean interaction field they experience.
By the law of large numbers, the mean-field interaction governing the dynamics of a representative process $X_t$
 is mediated by the expectation of the cumulative count
\begin{eqnarray}\label{eq:MFE}
F(t)
=
\lim_{K \to \infty} F_{K,i}(t)
=
\Exp{M_t}\, ,
\end{eqnarray}
where the process $M_t$ counts the successive first-passage times of the representative process $X_t$ to the zero spiking threshold.
This corresponds to the formal definition
\begin{eqnarray*}
M_t = \sum_{n>0} \mathbbm{1}_{[0,t]}(\tau_n) \, , \quad \mathrm{with} \quad \tau_{n+1} = \inf \left\{  t>r_n = \tau_n+s_n  \, \big  \vert \,  X_{t}  \leq 0 \right\} \, ,
\end{eqnarray*}
where $s_n$, $n \in \mathbbm{N}$, are i.i.d. according to $P_\epsilon$.
Observe that by definition, $F$ is a nondecreasing \cadlag function.
This allows one to define the firing rate $f$ in the distribution sense 
as the generalized derivative of $F$, i.e., $f = \mdd F /\mdd t$.
Correspondingly, we will refer to the $F$  as a cumulative rate function. 
Synchronous events whereby a finite fraction of processes spikes simultaneously,
which we refer to as full blowups, are marked by the presence of a jump discontinuity in $F$.

\begin{remark}
For the conjecture of propagation of chaos to hold, we expect certain moment constraints to hold on $P_\epsilon$. 
In particular, we expect that $P_\epsilon$ at least has a finite first moment, so that a representative
neuron spikes an infinite number of times over $\mathbbm{R}^+$ with probability one.
In this work, we only use the fact that $P_\epsilon$ has finite mean and finite variance.
\end{remark}

By contrast with cMF models, the interaction-mediating cumulative flux $F$ impacts neurons with Poissonian attributes in dPMF models, i.e., via a process 
\begin{eqnarray}
Z_t = \nu_1 t+ \lambda_1 F(t)   +  W_{\nu_2 t + \lambda_2 F(t)} ,\nonumber
\end{eqnarray}
where $W$ is a driving Wiener process.
Accordingly, a candidate stochastic equation for the dPMF dynamics of a representative process is given by
\begin{eqnarray}\label{eq:stochEqMF}
X_t = X_0 - \int_0^t \mathbbm{1}_{ \{ X_{s^-} >0 \} } \mdd Z_s + \Lambda N_t \, ,
\end{eqnarray}
where the  process $N_t$ counts the past reset times  of the representative process
\begin{eqnarray*}
N_t = \sum_{n>0} \mathbbm{1}_{[0,t]}(\tau_n+s_n)  \, .
\end{eqnarray*}
The stochastic integral featured above can be classically interpreted if one assumes that the cumulative flux $F$ is differentiable, 
which corresponds to nonexplosive dPMF dynamics.
By contrast, it is unclear how to interpret that same stochastic integral when $F$ is allowed to behave singularly and to exhibit jump discontinuities.
In fact, many interpretations are possible leading to distinct possible notions of blowup solutions.

\subsection{McKean-Vlasov equations with Poissonian attributes}\label{sec:MKV}

In order to justify the definition of the dPMF problem given in Definition~\ref{def:mainProb1}, 
let us consider the PDE problem satisfied by dPMF dynamics under assumptions that these are well behaved~\cite{Delarue:2015}.
By well-behaved, we mean that the cumulative rate $t \mapsto F(t) = \Exp{M_t}$ is differentiable 
on $[0,T)$ for some $T>0$ (which excludes blowups).
Such a well-behaved regime always exists for initial conditions far enough from the spiking threshold,
which includes any elementary initial conditions of the form $p(0,x)=\delta_{x_0}(x)$ with $x_0>0$.
Assuming $f$ known in the well-behaved regime, the probability density $(t,x) \mapsto p(t,x)$ 
of a representative process $X_t$ satisfying \eqref{eq:stochEqMF} solves the Fokker-Plank equation \eqref{eq:mainProb1a}
given by
\begin{eqnarray}
\partial_t p = (\nu_1+\lambda_1 f(t))  \partial_x p + (\nu_2+\lambda_2 f(t))  \partial^2_x p/2  +  f_\epsilon(t) \delta_{\Lambda} \, , \nonumber
\end{eqnarray}
with absorbing boundary condition  $p(t,0) = 0$.
Such an absorbing boundary condition ensures that the process becomes inactive upon reaching zero.
The reset rate $f_\epsilon$ in \eqref{eq:mainProb1a} is expressed as the convolution 
\begin{eqnarray}
f_\epsilon(t)=\int_\epsilon^t f(t-s) \, \mdd P_\epsilon(s) = \int_\epsilon^t f(t-s) \, p_\epsilon(s)  \, \mdd s \, . \nonumber
\end{eqnarray}
owing to the i.i.d. nature of the refractory delays and 
the fact that the process is assumed initially active, i.e., $X_0>0$.
Finally, the Dirac-delta source term $f_\epsilon(t) \delta_\Lambda$ in \eqref{eq:mainProb1a} models the reset
of newly activated processes in $\Lambda$, 
which happens in $t$ with reset rate $f_\epsilon(t)$.

The Fokker-Plank equation \eqref{eq:mainProb1a} is an exemplar of
 McKean-Vlasov equations because its coefficients depend on
the density solution via \eqref{eq:fluxBound}, which was given without justification in the introduction.
Here, we establish this relation from the requirement that \eqref{eq:mainProb1a} conserves probability.
Specifically, by \eqref{eq:mainProb1c}, we have
\begin{eqnarray}
\partial_t \left( \int_{0}^\infty \! p(t,x) \, \mdd x \right) = f_\epsilon(t) - f(t) \, . \nonumber 
\end{eqnarray}
Using \eqref{eq:mainProb1a}, we can evaluate the left term above as
\begin{eqnarray}
\partial_t \left( \int_{0}^\infty \! p(t,x) \, \mdd x \right)  &=&  \int_{0}^\infty  \big( (\nu_1  +  \lambda_1 f(t) ) \partial_x p(t,x) +  ( \nu_2  +  \lambda_2 f(t) ) \partial^2_{x} p(t,x) /2 \big)  \, \mdd x  \nonumber \\
 &&  +  \int_{0}^\infty f_{\epsilon}(t) \delta_{\Lambda}(x) \, \mdd x \, .  \nonumber 
\end{eqnarray}
Performing integration by parts with absorbing boundary condition in zero then yields
\begin{eqnarray}
\partial_t \left( \int_{0}^\infty \! p(t,x) \, \mdd x \right) = -\big(\nu_2 \! + \! \lambda_2 f(t)\big) \partial_x p(t,0)/2 + f_\epsilon(t) \, ,  \nonumber 
\end{eqnarray}
which together with \eqref{eq:mainProb1c} imposes the conservation condition 
$f(t) = (\nu_2 \! + \! \lambda_2 f(t)) \partial p_x(t,0) /2$.
Solving for $f(t)$ yields \eqref{eq:fluxBound}, which we recall as:
\begin{eqnarray*}
f(t) = \frac{\nu_2 \partial_x p(t,0)}{2 - \lambda_2 \partial_x p(t,0)} \, .
\end{eqnarray*}

\begin{remark} 
The relation above shows that well-behaved dPMF dynamics blowup as soon as the first time $T$ for which $\lim_{t \to T^-} \partial_x p(t,0)/2 = 1/\lambda_2$.
This is by contrast with cMF dynamics for which a blowup occurs for the first time at $T$ if $\lim_{t \to T}\partial_x p(t,0)/2 = \infty$.
This finite behavior at blowup makes dPMF dynamics amenable to analytical treatment.
\end{remark}


Observe that so far, we have avoided the problem of immediate explosions by considering initial conditions of the form $p_0(x)=\delta_{x_0}(x)$ 
with $x_0>0$ so that $\partial_x p(0,0)/2=0<1/\lambda_2$.
More generally, to avoid immediate explosions, we will consider nonexplosive initial conditions
as specified by Definition~\ref{def:mainProb2} in the introduction.
Definition~\ref{def:mainProb2} is such that the following result holds:

\begin{proposition}\label{prop:inactProp}
Given a nonexplosive initial condition $(p_0,f_0)$ in $\mathcal{M}(\mathbbm{R}^+) \times \mathcal{M}(\mathbbm{R}^-)$,
consider a solution to the dPMF problem \eqref{eq:mainProb1} with cumulative firing rate function $F:\mathbbm{R} \mapsto \mathbbm{R}$ such that $F(t)=F_0(t)=-\int_{(t,0)} f_0(\mdd s)$, $t \leq 0$.
Then, for all $t \geq 0$, we have
\begin{eqnarray*}
1-\int_0^\infty p(t,x) \, \mdd x = \int_{[0,\infty)}  (1-P_\epsilon(s)) \, \mdd F(t-s) \, .
\end{eqnarray*}
\end{proposition}

\begin{proof}
Given $S>0$, performing integration by parts for Stieltjes integrals, we obtain
\begin{align*}
\int_{[0,S)}   (1-P_\epsilon(s)) \,  \, \mdd F(t-s)
&= 
- \big[ (1-P_\epsilon(s)) F(t-s) \big]^{S^-}_{0} - \int_{[\epsilon,S)}   F(t-s) p_\epsilon(s) \, \mdd s \, , \\
&= 
F(t) - (1-P_\epsilon(S)) F(t-S^-)  - \int_{[\epsilon,S)}   F(t-s) p_\epsilon(s) \, \mdd s \, .
\end{align*}
Since $P_\epsilon$ has finite second-order moment, by Cantelli's inequality, we have
\begin{align*}
1-P_\epsilon(S) = \Prob{s_\epsilon > S} = \Prob{s_\epsilon - \Exp{s_\epsilon} > S - \Exp{s_\epsilon} } = \frac{\Var{s_\epsilon}}{\Var{s_\epsilon} + (S - \Exp{s_\epsilon})^2 } \, .
\end{align*}
By the integrability condition \eqref{eq:mainProb2a}, this implies that
\begin{align*}
 (1-P_\epsilon(S)) \vert F(t-S^-) \vert \leq O \left( \frac{\vert t-S \vert \Var{s_\epsilon}}{\Var{s_\epsilon} + (S - \Exp{s_\epsilon})^2 } \right)
\xrightarrow[S \to \infty]{} 0 \, .
\end{align*}
so that we have
\begin{align*}
\int_{[0,S)}   (1-P_\epsilon(s)) \, \mdd F(t-s) 
= 
F(t)  - \int_{[\epsilon,S)}   F(t-s) p_\epsilon(s) \, \mdd s  =  F(t) -F_\epsilon(t) \, .
\end{align*}
Therefore, using \eqref{eq:mainProb1c}, we have
\begin{align*}
\mdd \left[ \int_0^\infty p(t,x) \, \mdd x + \int_{[0,\infty)}   (1-P_\epsilon(s)) \, \mdd F(t-s) \right] 
= 0 \, .
\end{align*}
The result follows from the normalization condition \eqref{eq:mainProb2c}:
\begin{align*}
\int_0^\infty p_0(x) \, \mdd x + \int_{[0,\infty)} (1-P_\epsilon(s)) \,\mdd F_0(t-s) 
= 1 \, .
\end{align*}

\end{proof}

At this point, it is important to observe that we have only justified Definition~\ref{def:mainProb1} 
for well-behaved dPMF dynamics and it remains unclear how to interpret the meaning of \eqref{eq:mainProb1} 
after a blowup, when $f$ is allowed to behave singularly.
We devote the next section to propose physical solutions for blowup dPMF dynamics.


\begin{remark} The emergence of blowups in the solutions of the above McKean-Vlasov problem has been studied in the past:

\begin{enumerate}
\item  A special case of the problem  \eqref{eq:mainProb1} was considered in~\cite{TTPW} and in~\cite{TTLS} for the restricted case of $\nu_1/\nu_2=\lambda_1/\lambda_2$, for which closed-form analysis is possible.
Such closed-form analysis reveals that in the limit of vanishing delay $\epsilon \to 0^+$, blowup will happen independent of the initial condition if $\lambda_1=\lambda_2=\lambda>\Lambda$, where $\Lambda$ is the reset value.
This is consistent with the stationary solutions of  \eqref{eq:mainProb1} being unphysical for $\lambda>\Lambda$, as these yields negative steady-state firing rate: $f=\nu/(\Lambda-\lambda)$.
\item The nondelayed versions of the problem \eqref{eq:mainProb1}  belong to the extended class of dynamics introduced by Dou \emph{et al.} in~\cite{dou2024noisy}, where these authors also considers the mean-reverting effect of modeling neuronal leakage.
Here, we do not consider neuronal leakage as our focus is on the analytical characterization of blowups, whose occurrences do not require to include leakage.
From a modeling standpoint, this corresponds to considering the class a perfect integrate-and-fire neurons, rather than the more generic leaky integrate-and-fire neurons.
\item Both~\cite{TTPW} and~\cite{dou2022dilating} proposed to define blowup solutions for the problem  \eqref{eq:mainProb1} via a time-change approach.
Although, these solutions present the same blowup trigger conditions, they have conflicting criteria for blowup termination.
A benefit of the criterion for blowup termination presented in~\cite{TTPW} is that the corresponding solutions do not suffer from eternal blowups, 
which allows one to define global solutions with an infinite number of blowups in some parameter regime~\cite{TTLS}.
However, it is not obvious why one criterion should be the physical one in the absence of a unified framework to compare them. 
Recent work using a regularization approach proposed such a framework for the case $\lambda_2 =0$~\cite{dou2024noisy}.
Consistent with the result presented here, this study suggested that the blowup termination criterion in~\cite{dou2022dilating} is not physical.
\end{enumerate}
\end{remark}


\section{Candidate physical solutions via time change}\label{sec:physTC}

In this section, we specify our candidate solutions for physical dPMF dynamics with blowups.
In Section~\ref{sec:TC}, we define the PDE problem associated to the time-changed dPMF dynamics $Y_\sigma=X_{\Psi(\sigma)}$, which is obtained via an implicitly defined change of variable $\sigma=\Phi(t) \Leftrightarrow t=\Psi(\sigma)$ in the absence of blowups.
In Section~\ref{sec:solTC}, we show that assuming $\Psi=\Phi^{-1}$ known and independent of the presence of blowups, there is a unique density function $(\sigma,y) \mapsto q(\sigma,y)$ solving the dPMF time-changed problem.
In Section~\ref{sec:fixed-point}, we define the candidate equation that the inverse time change $\Psi$ must obey for blowup dPMF dynamics to be physical.
In Section~\ref{sec:contract}, we show that our candidate fixed-point equation specifies a unique, globally defined, blowup dPMF dynamics.


\subsection{The time-changed problem}\label{sec:TC}

Following prior works~\cite{TTPW,dou2022dilating}, we look for solutions to the buffered dPMF problem via time change, i.e., under the form $X_t=Y_{\Phi(t)}$ for some increasing function $\Phi$.
The form of the problem \eqref{eq:mainProb1} suggests adopting the time change function
\begin{eqnarray}\label{eq:PhiDef}
\Phi(t) = \nu_2 t + \lambda_2 F (t) = \nu_2 t + \lambda_2 \int_0^t f (s) \, \mdd s \, .
\end{eqnarray}
This motivates defining the space of admissible time-change functions as follows:

\begin{definition}\label{def:PhiBuff}
The set of admissible time change $\mathcal{T}_0$ is the set of functions $\Phi:[0, \infty)\to[0, \infty)$, with $\Phi(0)=0$,
such that for all $y,x \leq -\epsilon$, $x \neq y$, the difference quotients $w_\Phi(y,x)$ satisfy
\begin{eqnarray*}
\nu_2 \leq w_\Phi(y,x)=\frac{\Phi(y)-\Phi(x)}{y-x}  \, .
\end{eqnarray*}
\end{definition}

All time changes $\Phi$ in $\mathcal{T}$ are increasing, one-to-one on their domain. 
We denote their uniquely defined, continuous, inverses by $\Psi=\Phi^{-1}$ owing 
to the important role that inverse time changes will play.
We will also refer to the changed time as $\sigma = \Phi(t)$.
The space of inverse time-change functions is naturally defined as:

\begin{definition}\label{def:PhiBuff}
The set of admissible inverse time change $\mIz$ is the set of functions $\Psi:[0, \infty)\to[0, \infty)$ with $\Psi(0)=0$, such that for all $y,x \geq 0$, $x \neq y$, the difference quotients $w_\Psi(y,x)$ satisfy
\begin{eqnarray*}
0 \leq  w_\Psi(y,x)=\frac{\Psi(y)-\Psi(x)}{y-x} \leq \frac{1}{\nu_2} \, .
\end{eqnarray*}
\end{definition}


The main advantage of considering inverse time change $\Psi \in \mIz$ is that by the
Arz\'ela-Ascoli theorm, these belong to a compact space when considered
on finite interval $[0,T]$.
We will denote the corresponding compact space as $\mIz([0,T])$, and more generally, for
any interval $I \subset \mathbbm{R}$, $0\in I$, we will denote by $\mIz(I)$ 
the set of function $\Psi$ such that $\Psi(0)=0$ and $w_\Psi \leq 1/\nu_2$ over $I$
with the convention that $\mIz(\mathbbm{R}^+)=\mIz$ for conciseness.
For being uniformly Lipschitz increasing function, all functions $\Psi \in \mIz(I)$ admit
positive, bounded Radon-Nikodym derivatives with respect to the Lebesgue measure.
For simplicity, we denote these generalized derivatives as $\Psi'=\mdd \Psi/\mdd \sigma$.
We will consider the inverse time change $\Psi$ 
as the main unknown parametrizing the solutions to dPMF dynamics.
The main caveat to this approach is that the  change of variable $\sigma=\Phi(t) \Leftrightarrow t=\Psi(\sigma)$
becomes degenerate in the presence of blowups.
Bearing this caveat in mind, one can see that this time-changed approach
allows one to consider the dPMF problem \eqref{eq:mainProb1} as a particular case of 
a more general PDE problem that is well-posed for any $\Psi \in \mIz$, independent of $\Psi$ satisfying 
\eqref{eq:PhiDef}, and in particular, independent of the occurrence of blowups.
This more general, time-changed PDE problem can be defined for the following notion of initial conditions:

\begin{definition}\label{def:TCinit}
\begin{subequations}\label{eq:TCinit}
Generic initial conditions are specified as $q(0,y) =  q_0(y)$, $y>0$, and $g(\sigma) = g_0(\sigma)$, $ \sigma \leq 0$, 
where $(q_0,g_0)$ are nonnegative measures in $\mathcal{M}(\mathbbm{R}^+) \times \mathcal{M}(\mathbbm{R}^-)$ satisfying:
\begin{itemize}
\item the integrability condition
\begin{eqnarray}\label{eq:TCinita}
G_0(0) = \sup_{\sigma \leq 0} \left( \frac{\sigma}{\lambda_2} + \int_{(\sigma,0]} g_0(\mdd \tau)  \right) < \infty \, ,
\end{eqnarray}
\item the boundedness condition
\begin{eqnarray}\label{eq:TCinitb}
\partial_y q_0(0)= \lim_{y \downarrow 0}\frac{q_0((0,y])}{y}   < \infty\, ,
\end{eqnarray}
\item the normalization condition
\begin{eqnarray}\label{eq:TCinitc}
\int_{0}^\infty q_0(y) \, \mdd y + \int_{-\infty}^0 (1-Q_0(\sigma)) \, \mdd G_0(\sigma) = 1 \, ,
\end{eqnarray}
\end{itemize}
where we have defined 
$G_0(\sigma)= - g_0((\sigma,0])  + G_0(0)$ and 
$Q_0 (\sigma) = P_\epsilon (-\Psi_0(\sigma))$ with $\Psi_0 \in \mIz(\mathbbm{R}^-)$ specified by:
\begin{align}\label{eq:TCinitd}
\Psi_0(\sigma) 
=  \frac{1}{\nu_2}\sup_{\tau \leq \sigma} \left( \tau - \lambda_2 G_0(\tau)  \right) \, ,\quad \sigma \leq 0 \, .
\end{align}
\end{subequations}
\end{definition}

\begin{remark}
Observe that by definition of $G_0(\sigma)$, $\sigma \leq 0$, $\Psi_0$ satisfies
\begin{align*}
\Psi_0(\sigma) 
&=
\frac{1}{\nu_2}  \left( \sup_{\tau \leq \sigma} \left( \tau + \lambda_2 g_0((\sigma,0]) \right) + \lambda_2 G_0(0)  \right)\, , \\
&=  \frac{1}{\nu_2}  \left(  \sup_{\tau \leq \sigma}\left( \tau + \lambda_2 \int_{(\sigma,0]} g_0(\mdd \tau)   \right)  - \sup_{\tau \leq 0} \left( \tau + \lambda_2 \int_{(\sigma,0]} g_0(\mdd \tau)  \right) \right)  \, ,
\end{align*}
so that we clearly have $\Psi_0(0)=0$.
\end{remark}

Any nonexplosive initial conditions $(p_0,f_0)$ for the original dPMF problem can be mapped 
onto generic initial conditions $(q_0,g_0)$ in the time-changed picture.
This mapping is straightforward when no blowup has occurred in the past, i.e., when the firing rate $f_0$ 
is continuous, so that 
\begin{eqnarray*}
\mathbbm{R}^- \ni t \mapsto \Phi_0(t) = \nu_2 t + \lambda_2 \int_0^t f_0(s) \, \mdd s = \nu_2 t + \lambda_2 F_0(t) 
\end{eqnarray*}
is a continuously differentiable, increasing function. 
As a result, the change of variable $\sigma = \Phi_0(t) \Leftrightarrow t = \Psi_0(\sigma)$ is well defined,
allowing one to specify the time-changed version of $f_0$ as  $g_0= \partial_\sigma G_0$, where $G_0(\sigma)=F_0(\Psi_0(\sigma))$.
Note that we necessarily have $G_0(0)=0$ in this case.
The spatial component of the initial conditions is always identically matched:  $q_0=p_0$.

One can then check that $(q_0,g_0)$ satisfies the properties \eqref{eq:TCinit} and therefore specifies a generic initial condition.
By the integrability condition \eqref{eq:mainProb2a} and the continuity of $F_0$, if $F_0 \neq 0$ (the case $F_0=0$ is trivial), 
there is $A>0$ and $T>0$ such that $A t \leq F_0(t) < 0 $ for all $t \leq T$.
Thus, given $\sigma = \Phi_0(t) \Leftrightarrow t = \Psi_0(\sigma)$, for $t>T$, we have
\begin{eqnarray*}
\frac{G_0(\sigma)}{\sigma} = \frac{F_0(t)}{\Phi_0(t)} = \frac{F_0(t)}{\nu_2t + \lambda_2 F_0(t)} \leq \frac{1}{\lambda_2 + \nu_2/A} < \frac{1}{\lambda_2} \, .
\end{eqnarray*}
This implies that $\sigma-\lambda_2 G(\sigma) \to -\infty$ when $\sigma \to -\infty$, and thus that \eqref{eq:TCinita} holds.
The boundedness condition \eqref{eq:TCinitb} directly follows from \eqref{eq:mainProb2b}.
Finally, \eqref{eq:TCinitc} follows from \eqref{eq:mainProb2c} with the time-changed cumulative distribution $Q_0=P_\epsilon \circ (-\Psi_0)$.

We refer to the initial conditions specified in Definition~\ref{def:TCinit} as generic because they 
generalize nonexplosive initial conditions. 
To see this, note that the spatial component $q_0$ is only required to have a bounded derivative
by \eqref{eq:TCinitb} instead of satisfying the more restrictive nonexplosive condition \eqref{eq:mainProb2b}.
This is because in the time-changed picture, it becomes possible to initialize dPMF problems
during a blowup episode, which involves time for which $\partial_x q(\sigma,0)/2>1/\lambda_2$.
Actually, the notion of generic initial conditions is precisely defined to allow for the possibility of 
blowing up at the initial time. 
Before further discussing this point,
it is instructive to define the time-changed PDE problem for dPMF dynamics as follows:

\begin{definition}\label{def:TCPDE} 
\begin{subequations}
\label{all:TCPDE}
Given a time change  $\Psi \in \mIz$ and a generic initial condition $(q_0,g_0)$ in $\mathcal{M}(\mathbbm{R}^+) \times \mathcal{M}(\mathbbm{R}^-)$, let us extend the definition of $\Psi$ in $\mIz(\mathbbm{R})$ by setting $\Psi(\sigma)=\Psi_0(\sigma)$ for $\sigma \leq 0$ and define the time-changed, refractory-delay distribution $Q_\sigma$ as 
\begin{eqnarray}\label{eq:EtaProb}
Q_\sigma (\eta) 
= 
P_\epsilon \big(\Psi(\sigma) - \Psi(\sigma-\eta)\big) \, .
\end{eqnarray}
Then, for fixed $\Psi \in \mIz$, the time-changed PDE problem for dPMF dynamics consists in finding the density function $(\sigma,y) \mapsto q(\sigma,y)= \Prob{Y_\sigma \in \mdd y} / \mdd y$ solving 
\begin{eqnarray}\label{eq:TCdyn}
\partial_\sigma q 
=
\mu(\sigma)  \partial_y q + \frac{1}{2} \partial^2_y q + \partial_\sigma G_\epsilon(\sigma)  \delta_\Lambda \,   , \quad q(\sigma,0)=0 \, ,
\end{eqnarray}
on $\mathbbm{R}^+\times \mathbbm{R}^+$, where the drift function is given by
\begin{eqnarray}\label{eq:TCdrift}
\mu(\sigma) =  \left( \nu_1 - \frac{\lambda_1}{\lambda_2} \nu_2   \right)\Psi'(\sigma)  +  \frac{\lambda_1}{\lambda_2}  \, ,
\end{eqnarray}
where the cumulative function $G$ is such that
\begin{eqnarray}\label{eq:gDef}
\partial_\sigma G(\sigma) = g(\sigma) = \partial_y q(\sigma,0)/2 \, , \quad G(0)=G_0(0) \, ,
\end{eqnarray}
and where the cumulative, delayed, reset rate is specified via 
\begin{align}\label{eq:GPsig}
G_{\epsilon}(\sigma)= \int_0^\infty G(\sigma-\eta)  \, \mdd Q_\sigma(\eta) \, .
\end{align}
\end{subequations}
\end{definition}

Remember that the generalized derivative $\Psi'$ featuring in \eqref{eq:TCdrift}  is  
the Radon-Nikodym derivative of $\Psi$ with respect to the Lebesgue measure,
 $\Psi'=\mdd \Psi /\mdd \sigma$, which is well defined for all $\Psi \in \mIz$.
The following proposition establishes the connection between the time-changed dPMF problem and the original dPMF 
problem under the assumption that $\Phi$ is differentiable,
so that the time change $\sigma=\Phi(t) \Leftrightarrow t=\Phi(\sigma)$ is well posed.

\begin{proposition}\label{prop:PhiToPsi}
Suppose that $(t,x) \mapsto p(t,x)$ solves the problem  \eqref{eq:mainProb1}  for a differentiable time change $\Phi$, 
then $(\sigma,x) \mapsto q(\sigma,x) = p(\Psi(\sigma),x)$ solves the problem \eqref{all:TCPDE}
for $\Psi=\Phi^{-1}$ with $\Phi(t)=\nu_2 t + \lambda_2 F (t)$.
\end{proposition}

\begin{proof}
Suppose that the density function $(t,x) \mapsto p(t,x)$ solves  \eqref{eq:mainProb1}.
Under assumption of differentiability, the function $t \mapsto \Phi(t)=\nu_2 t + \lambda_2  F(t) \in \mathcal{T}_0$
admits a unique, increasing, differentiable inverse function $\Psi=\Phi^{-1} \in \mIz$.
By change of variable, the density $q(\sigma,x)=p(\Psi(\sigma),x)$ of the process $Y_\sigma=X_{\Psi(\sigma)}$ is such that 
\begin{eqnarray*}
\partial_\sigma q(\sigma,x)
=
\frac{\partial_t p(t,x)}{\Phi'(t)} \, .
\end{eqnarray*}
Given that $\Phi(t)=\nu_2 t + \lambda_2 F (t)$, we may write the original dPMF dynamics as
\begin{eqnarray*}
\partial_t p =  \left( \nu_1 +  \frac{\lambda_1}{\lambda_2} (\Phi'(t) - \nu_2 )\right) \partial_x p + \frac{\Phi'(t)}{2} \partial^2_x p + f_\epsilon(t) \delta_\Lambda \,.
\end{eqnarray*}
Therefore
\begin{eqnarray}\label{eq:pq1}
\partial_\sigma q 
=
 \left(  \left( \nu_1 - \frac{\lambda_1}{\lambda_2} \nu_2   \right) \frac{1}{\Phi'(t)}  +  \frac{\lambda_1}{\lambda_2} \right) \partial_x p + \frac{1}{2} \partial^2_x p + \frac{f_\epsilon(t)}{\Phi'(t)} \delta_\Lambda \, .
\end{eqnarray}
Let us introduce the time-changed cumulative function
\begin{eqnarray*}
G(\sigma) = F(\Psi(\sigma)) \Leftrightarrow G(\Phi(t)) = F(t) \, ,
\end{eqnarray*}
such that for all $\sigma \geq 0$, we have
\begin{eqnarray*}
g(\sigma) = \partial_\sigma G(\sigma) = \Psi'(\sigma) f(\Psi(\sigma)) = \frac{f(t)}{\Phi'(t)} \, .
\end{eqnarray*}
Similarly, the time-changed, cumulative, reset rate is given by
\begin{eqnarray*}
G_\epsilon(\sigma) = F_\epsilon(\Psi(\sigma)) =   \int_{[\epsilon, \infty)} F(\Psi(\sigma)-s) \,  \mdd P_\epsilon(s) =   \int_{[\epsilon, \infty)} G(\sigma-\eta_s(\sigma)) \,  \mdd P_\epsilon(s)  \, ,
\end{eqnarray*}
where, for all $s \geq \epsilon$, we have defined the increasing, continuous, backward delay function as
\begin{eqnarray*}
\mathbbm{R}^+ \ni \sigma \mapsto \eta_s(\sigma) = \sigma - \Phi(\Psi(\sigma)-s) \, .
\end{eqnarray*}
Performing the change of variables $\eta=\eta_\sigma(s)$ yields 
\begin{eqnarray*}
G_\epsilon(\sigma)
=
\int_0^\infty G(\sigma-\eta) \,  \mdd Q_\sigma(\eta) 
\, ,
\end{eqnarray*}
where the collection of distributions $Q_\sigma$ are defined as
\begin{eqnarray*}
Q_\sigma(\eta) =  P_\epsilon(\Psi(\sigma)-\Psi(\sigma-\eta)) \, .
\end{eqnarray*}
Moreover, it is clear that 
\begin{eqnarray*}\label{eq:pq2}
\partial_\sigma G_\epsilon(\sigma)
=
\partial_\sigma \big[ F_\epsilon(\Psi(\sigma)) \big]
=
\Psi'(\sigma) f_\epsilon(\Psi(\sigma)) 
=
\frac{f_\epsilon(t)}{\Phi'(t)} \, .
\end{eqnarray*}
Since for all $x ,t \geq 0$, $\partial_x p(t,x)=\partial_x q(\sigma,x)$, $\partial^2_x p(t,x)=\partial^2_x q(\sigma,x)$, $\Phi'(t)=\Psi'(\sigma)^{-1}$, \eqref{eq:pq1} and \eqref{eq:pq2} implies that $(\sigma,x) \mapsto q(\sigma,x)$ satisfies \eqref{eq:TCdyn} and \eqref{eq:TCdrift}, as announced.
It remains to check that \eqref{eq:gDef} holds.
This follows from observing that conservation of probability implies
\begin{eqnarray*}
\partial_\sigma \left( \int_0^\infty q(\sigma, x) \, \mdd x \right) 
= \Psi'(\sigma) \partial_t  \left(\int_0^\infty p(t,x) \, \mdd x \right) \
= \frac{f_\epsilon(t)}{\Phi'(t)} - \frac{f(t)}{\Phi'(t)} 
=  \partial_\sigma G_\epsilon(\sigma) - g(\sigma) \, ,
\end{eqnarray*}
whereas  the absorbing boundary condition in \eqref{eq:TCdyn} implies 
\begin{eqnarray*}
\int_0^\infty  \partial_\sigma q(\sigma, x) \, \mdd x 
&=&  - \frac{1}{2} \partial_x q(\sigma,0) + \partial_\sigma G_\epsilon(\sigma)  \, .
\end{eqnarray*}
Equating both the above equations shows that \eqref{eq:gDef} holds.
\end{proof}

In the general situation where full blowups occur in the past, i.e., when $f_0$ has atoms, 
the change of variable $\sigma=\Phi_0(t) \Leftrightarrow t =\Psi_0(\sigma)$ is only well defined 
when $t$ is a continuity point of $\Phi_0$, and thus of $F_0$.
As a result, given a nonexplosive initial condition, the corresponding generic initial condition is only 
unambiguously specified via $G_0=F_0 \circ \Psi_0$ for all $\sigma$ such that $t=\Psi_0(\sigma)$ is a continuity point of $F_0$.
By contrast, many generic initial conditions may be associated to nonexplosive initial conditions 
during blowups, i.e., over intervals $[\Phi_0(t^-),\Phi_0(t)]$ when $t$ is a discontinuity point of $F_0$.
Informally, this follows from the form of the time-changed problem \eqref{all:TCPDE}, which
reveals that generic initial conditions only impact dPMF dynamics via
the cumulative reset rate
\begin{eqnarray*}
G_{\epsilon,0}(\sigma)
=
\int_0^\infty G_0(\sigma-\eta) \,  \mdd Q_\sigma(\eta) 
=
\int_0^\infty G_0(\sigma-\eta) \,  \mdd P_\epsilon(\Psi_0(\sigma)-\Psi_0(\sigma-\eta)) \, . 
\end{eqnarray*}
Since we assume that  $P_\epsilon$ has a smooth density and since $\Psi_0$ is flat over discontinuity intervals $[\Phi_0(t^-),\Phi_0(t)]$, $\Phi_0(t^-)<\Phi_0(t)$, it is clear
that restrictions of $G_0$ to  $[\Phi_0(t^-),\Phi_0(t)]$ do not contribute to the integral 
expression of $G_\epsilon$ above, and thus do not impact dPMF dyanmics.
As a result, in the presence of past blowups, every nonexplosive initial condition  $(p_0,f_0)$ in $\mathcal{M}(\mathbbm{R}^+) \times \mathcal{M}(\mathbbm{R}^- )$ 
is associated to an equivalence class of generic initial conditions. 
We specify this equivalence class in Proposition~\ref{prop:initLink} below, where we adopt
the following notations.
Given a nonexplosive initial condition $(p_0,f_0)$ in $\mathcal{M}(\mathbbm{R}^+) \times \mathcal{M}(\mathbbm{R}^- )$,
we denote the discontinuity times of 
$$\mathbbm{R}^- \ni t \mapsto \Phi_0(t) = \nu_2 t + \lambda_2 F_0(t)= \nu_2 t - \lambda_2 \int_{(t,0]} f_0(s) \, \mdd s$$
by $\{ T_k \}_{k \in \mathcal{K}}$, where $\mathcal{K}$ is at most a countable set.
For conciseness, we further introduce the times $U_k=\Phi_0(T_k)$ and $S_k=\Phi_0(T_k^-)$, $k \in \mathcal{K}$, 
so that the set of discontinuity intervals is the countable union $\cup_{k \in \mathcal{K}} [S_k,U_k]$.
Observe that since $\Phi_0$ is increasing with $w_{\Phi_0} \geq \nu_2$, $[S_k,U_k]$, $k \in \mathcal{K}$,
are nonoverlapping intervals.
Given these notations, we have the following result, whose  proof is given in Section~\ref{sec:initproofs}.

\begin{proposition}\label{prop:initLink}
\begin{subequations}\label{all:initLink}
Any nonexplosive initial conditions $(p_0,f_0)$ in $\mathcal{M}(\mathbbm{R}^+) \times \mathcal{M}(\mathbbm{R}^- )$ specifies 
a class of equivalent generic initial conditions $(q_0, g_0)$ where $q_0=p_0$ and where $g_0$ is uniquely specified
on the closed set $\mathbbm{R}^- \setminus \cup_{k \in \mathcal{K}} (S_k,U_k)$
\begin{eqnarray}\label{eq:initLink1}
G_0(\sigma)= -g_0((\sigma,0))=
\left\{
\begin{array}{ccc}
F_0(\Psi_0(\sigma))  & \mathrm{if}  &   \sigma \notin \cup_{k \in \mathcal{K}} [S_k,U_k) \, , \\
  \lim_{\sigma \downarrow S_k} F_0(\Psi_0(\sigma))  &   \mathrm{if}  &   \sigma \in \{ S_k \}_{k \in \mathcal{K}} \, ,
\end{array}
\right.
\end{eqnarray}
but is only constrained on the open set $\cup_{k \in \mathcal{K}} (S_k,U_k)$ to satisfy
\begin{eqnarray}\label{eq:initLink2}
U_k = \inf \left\{ \sigma > S_k \, \bigg \vert \,  G_0(\sigma)- G_0(S_k)  < (\sigma - S_k ) / \lambda_2 \right\} \, .
\end{eqnarray}
In particular, $G_0$ must be such that $G_0(0)=0$ and must satisfy for all $k \in \mathcal{K}$
\begin{eqnarray}\label{eq:initLink3}
G_0(U_k)- G_0(S_k) = (U_k - S_k ) / \lambda_2\, .
\end{eqnarray}
\end{subequations}
\end{proposition}

\begin{remark}\label{rem:intiBlow}
Generic initial conditions $(q_0,g_0)$ that correspond
to nonexplosive initial conditions all satisfy $G_0(0)=F_0(0)=0$ in \eqref{eq:TCinita}.
By contrast, generic initial conditions for which $G_0(0)>0$ correspond
to time-changed dPMF dynamics that blowup at initial time.
Anticipating on future results (see Section~\ref{sec:compRes}), we will see that 
$G_0(0)$ can be interpreted as $D(0)=E(0)$, the common initial value of the excess functions $E$ and $D$
for the buffer mechanism introduced in Section~\ref{sec:bufferAnalysis}.
\end{remark}


\subsection{Solutions to the time-changed problem via first-passage time analysis} \label{sec:solTC}
In this section, we assume that $\Psi \in \mIz$ is a known function. 
In the absence of delay, i.e., for $P_\epsilon=\delta_0$, it is well-known that independent of the initial data, the time-changed problem \eqref{all:TCPDE}  
admits a unique, well-behaved, global solution  $(\sigma,x) \mapsto q(\sigma,x)$ as soon as $\Psi$ is at least H\"older continuous with exponent $\alpha>1/2$~\cite{Roz84}.
In particular, this solution is such that $\sigma \mapsto g(\sigma)=\partial_x q(\sigma,0)/2$ is a continuous function.
This follows from the fact that the problem can be formulated as an initial-boundary-value problem for the heat equation in a moving strip.
This approach allows one to obtain a representation of the solution $\sigma \mapsto g(\sigma)=\partial_x q(\sigma,0)/2$ in terms of the so-called single-layer potential functions, which are uniquely characterized as solutions of certain Volterra integral equations~\cite{Roz84}.

In our case, the inclusion of distributed delays will not impact the existence and uniqueness results discussed above.
That said, our approach to solving the problem \eqref{all:TCPDE} partially differs from other works~\cite{Carrillo:2013,Delarue:2015,dou2022dilating}, 
as it is based on a characterization of these classical solutions via delayed renewal-type equations.
For the sake of completeness, we thus establish the following result:

\begin{proposition}\label{prop:solG}
\begin{subequations}
\label{all:renew}
Given a time change $\Psi  \in \mIz$ and generic initial conditions $(q_0,g_0)$ in $\mathcal{M}(\mathbbm{R}^+) \times \mathcal{M}((-\infty,0))$,
the solution $(\sigma,x) \mapsto q(\sigma,x)$ to the problem \eqref{all:TCPDE} is parametrized by the unique solution $G$ in $C_1(\mathbbm{R}^+)$ of the renewal-type equation
\begin{eqnarray}\label{eq:Grenew}
G(\sigma) = \int_0^\infty H(\sigma,0;x)  q_0(x) \, \mdd x +  \int_0^\sigma H(\sigma,\tau,\Lambda)  \, \mdd G_\epsilon(\tau) \, , \quad G(0)=G_0(0) \, ,
\end{eqnarray}
where for all $x, \tau \geq 0$, $\sigma \in [\tau, \infty) \mapsto H(\sigma,\tau;x)$ is the cumulative distribution of the following first-passage time  problem for the standard Wiener process $W$:
\begin{eqnarray}\label{eq:firstPass}
\sigma_{\tau,x}=\inf \left\{ \sigma >\tau \, \Bigg \vert \, W_\tau=x, \, W_\sigma <  \int_\tau^\sigma \mu(\xi) \, \mdd \xi \right\} \, .
\end{eqnarray}
\end{subequations}
\end{proposition}

\begin{proof}
Given $x, \tau \geq 0$, let us consider for all $\sigma \geq \tau$ the initial-boundary-value problem
\begin{eqnarray}\label{eq:absGreen}
\partial_\sigma \kappa
=
\mu(\sigma)  \partial_y \kappa + \frac{1}{2} \partial^2_y \kappa  \, , \quad \mathrm{with} \quad \kappa(\tau,y) = \delta_x(y)  \quad \mathrm{and} \quad \kappa(\sigma,y) = 0 \, .
\end{eqnarray}
Solutions to the above problem represent the density of a timed-changed dynamics $Y_\sigma$ started in $x$ but without reset. 
Following classical analysis~\cite{Roz84}, one can show that \eqref{eq:absGreen} has a unique solution $(\sigma,y) \mapsto \kappa(\sigma,y)$, which admits a representation in terms of the so-called single-layer potential.
Specifically, denoting $M(\sigma)=\int_0^\sigma \mu(\tau) \, \mdd \tau$, one can look for solutions  \eqref{eq:absGreen} under the form
\begin{eqnarray}\label{eq:singelLayer}
\lefteqn{\kappa(\sigma,y)
=
k(y-x-M(\sigma) + M(\tau), \sigma-\tau)  } \nonumber\\
&& \hspace{50pt} - \int_0^\sigma k(y - M(\sigma) + M(\zeta), \sigma-\zeta) h(\zeta) \, \mdd \zeta   \, ,
\end{eqnarray}
where $k$ denotes the standard Gaussian kernel, i.e., $k(x,t)=e^{-x^2/(2t)}/\sqrt{2 \pi t}$ and where $h$ is some as-of-yet unknown function.
Then, one can determine the function $h$ so that the representation given in \eqref{eq:singelLayer}  satisfies the absorbing boundary condition $\kappa(\sigma,0)=0$ for all $\sigma \geq \tau$.
Taking $y \to 0^+$ in \eqref{eq:singelLayer} shows that $h$ must solve 
\begin{eqnarray}\label{eq:singelLayer}
k(-x - M(\sigma) + M(\tau), \sigma-\tau)  
= \int_0^\sigma k(-M(\sigma) + M(\zeta), \sigma-\zeta) h(\zeta) \, \mdd \zeta   \, .
\end{eqnarray}

It is known that the above Volterra equation of the first type admits a unique continuous solution if $M$ is H\"older continuous with exponent $\alpha>1/2$~\cite{Roz84}, 
a condition that holds for all $\Psi \in \mId$ as
\begin{eqnarray*}
M(\sigma) = \int_0^\sigma \mu(\tau) \, \mdd \tau =  \left( \nu_1 - \frac{\lambda_1}{\lambda_2} \nu_2   \right)\Psi(\sigma)  +  \frac{\lambda_1}{\lambda_2} \sigma 
\end{eqnarray*}
is uniformly Lipschitz.
Moreover, it is known that the unique solution to that equation can be interpreted as the density of the first-passage time $\sigma_{\tau,x}$ defined in \eqref{eq:firstPass} (see, e.g.,~\cite{Peskir:2002aa}),
In particular, we have
\begin{eqnarray}\label{eq:hDef}
h(\sigma)=h(\sigma,\tau;x) = \partial_\sigma H(\sigma,\tau; x) = \frac{1}{2}   \partial_y \kappa(\sigma,0) \, .
\end{eqnarray}

Let us now refer to the unique solution to \eqref{eq:absGreen} for given $x,\tau \geq 0$ as $(\sigma,y) \mapsto \kappa(y,x;\sigma,\tau)$.
By Duhamel's principle, solutions to the problem with reset \eqref{all:TCPDE} must have the form
\begin{eqnarray*}
q(\sigma,y)
=
\int_0^\infty \kappa(y,x;\sigma,0) q_0(x) \, \mdd x + \int_0^\sigma  \kappa(y,\Lambda;\sigma,\tau) \, \mdd G_\epsilon(\tau) \, ,
\end{eqnarray*}
where the delayed, cumulative, reset rate $G_\epsilon$ must still be specified.
By conservation of probability, we have $g(\sigma)=\partial_\sigma G(\sigma)= \partial_y q(\sigma, 0)/2$, so that upon differentiating with respect to $y$ and taking $y \to 0$, we have
\begin{eqnarray}\label{eq:renew2}
g(\sigma)
&=&
\frac{1}{2} \left( \int_0^\infty \partial_y \kappa(0,x;\sigma,0) q_0(x) \, \mdd x + \int_0^\sigma  \partial_y \kappa(0,\Lambda;\sigma,\tau) \, \mdd G_\epsilon(\tau)  \right)\, , \nonumber \\
&=&
 \int_0^\infty h(\sigma,0,x) q_0(x) \, \mdd x + \int_0^\sigma  h(\sigma,\tau,\Lambda)\, \mdd G_\epsilon(\tau)  \, .
\end{eqnarray}
Integrating  \eqref{eq:renew2} with respect to $\sigma$ yields the sought-after characterizing equation \eqref{eq:Grenew}.

It remains to justify that \eqref{eq:Grenew} admits a unique solution in $C_1([0,\infty))$.
The existence and uniqueness is obvious by sequential continuation on the intervals $[\chi_k, \chi_{k+1})$, $k \geq 0$, 
for the increasing sequence of times $\chi_k$, $k \geq 0$, which is defined by 
\begin{eqnarray*}
\chi_0=0 \quad \mathrm{and} \quad \Psi(\chi_{k+1}) = \Psi(\chi_k) + \epsilon \, .
\end{eqnarray*}
Note that this solution must  be in  $C_1([0,\infty))$ because for all $x,\sigma>0$, the functions $\sigma \mapsto H(\sigma,\tau,x)$ are in $C_1([0,\infty))$ with $0 \leq H(\sigma,\tau,x) \leq 1$.

\end{proof}

Proposition~\ref{prop:solG} reveals the key role played by the first-passage times $\sigma_{\tau,x}$ defined in \eqref{eq:firstPass}.
For all $\Psi \in \mIz$, it is well-known that these first passage times admit continuous densities $\sigma \mapsto h(\sigma,\tau,x) = \partial_\sigma H(\sigma,\tau, x)$. 
Using the fact that
\begin{eqnarray*}
\mu_m= \frac{\lambda_1}{\lambda_2} \wedge \frac{\nu_1}{\nu_2}  \leq\frac{\int_\tau^\sigma\mu(u)du}{\sigma-\tau} =\frac{M(\sigma)-M(\tau)}{\sigma-\tau}  \leq  \frac{\lambda_1}{\lambda_2} \vee \frac{\nu_1}{\nu_2}  = \mu_M
\end{eqnarray*}
for  all $\Psi \in \mIz$,
one can further show (see \eqref{eq:boundh1} in Proposition~\ref{prop:boundh}) that for all $x>0$, 
these densities are uniformly bounded over $\mIz$ in the following sense:
\begin{align}\label{eq:upperDensity}
    h(\sigma,\tau,x) \leq \left( \frac{x}{\sqrt{(\sigma-\tau)^3}}+\frac{ \mu_M - \mu_m }{\sqrt{\sigma-\tau}} \right) e^{ - \frac{\left(x- \mu_m (\sigma-\tau) \right)^2 \wedge \left(x- \mu_M (\sigma-\tau) \right)^2}{2(\sigma-\tau)} } \, .
\end{align}
In turn, this result can be used to show the \emph{a priori} boundedness of the rate function $g$ in the time-changed picture:

\begin{proposition}\label{prop:boundedness}
Given generic initial conditions, if $(\sigma,x) \mapsto q(\sigma,x)$ solves the time-changed problem \eqref{all:TCPDE} for an inverse time change $\Psi \in \mIz$,
then $\sigma \mapsto g(\sigma)=\partial_x q(\sigma,0)/2$ admits a uniform upper bound $C_S$ with respect to $\Psi \in \mIz$ on $[0,S]$, $S>0$.
Moreover, the constant $C_S$ is independent of the delay distribution $P_\epsilon$.
\end{proposition}

\begin{proof}
One can check that the function
\begin{eqnarray*}
\sigma \mapsto  \left( \frac{\Lambda}{\sqrt{\sigma^3}}+\frac{ \mu_M - \mu_m }{\sqrt{\sigma}} \right)  e^{ - \frac{\left(\Lambda- \mu_m \sigma \right)^2 \wedge \left(x- \mu_M \sigma \right)^2}{2\sigma} } 
\end{eqnarray*}
is uniformly bounded above by some constant $C$.
Therefore,  \eqref{eq:upperDensity} implies that $h(\sigma,\tau,\Lambda) \leq C$ for all $\sigma,\tau$, $0\leq \sigma \leq \tau \leq S$, so that by nonnegativity of $g$, $h$, and $\eta$, we have
\begin{eqnarray*}
g(\sigma)
&\leq& 
\int_0^\infty h(\sigma,0,x) q_0(x) \, \mdd x + C\big(G_\epsilon(\sigma)-G_\epsilon(0) \big) \, , \\
&=& 
\int_0^\infty h(\sigma,0,x) q_0(x) \, \mdd x - C G_\epsilon(0) + C \int_{-\infty}^0    G(\sigma - \tau)  \, \mdd Q_\sigma(\tau) \, , \\
&\leq& 
\int_0^\infty h(\sigma,0,x) q_0(x) \, \mdd x - C G_\epsilon(0) +  C G(\sigma)  \, .
\end{eqnarray*}
The announced result follows from Gr\"onwall's inequality once we justify that the function
\begin{eqnarray*}
\sigma \mapsto \int_0^\infty  h(\sigma,0,x) q_0(x) \, \mdd x
\end{eqnarray*}
remains finite over $\mathbbm{R}^+$ for all  $\Psi \in \mIz$.
To see this, let us remember that we consider initial conditions such that there is a constant $a,A>0$, 
such that for all $x$, $0 \leq x \leq a$, we have $q_0(x)<Ax$.
Then, by \eqref{eq:upperDensity}, we have
\begin{eqnarray*}
 \int_0^a h(\sigma,0,x) q_0(x) \, \mdd x 
& \leq & 
 A \int_0^a h(\sigma,0,x) x \, \mdd x \, , \\
& \leq &
A   \int_0^a   x \left( \frac{x}{\sqrt{\sigma^3}}+\frac{ \mu_M - \mu_m }{\sqrt{\sigma}} \right) e^{ - \frac{\left(x- \mu_m \sigma \right)^2 \wedge \left(x- \mu_M \sigma \right)^2}{2 \sigma} } \, \mdd x \, .
\end{eqnarray*}
The above integral, which can be expressed in closed form via the error function, defines a uniformly bounded, continuous function of $\sigma$ on $\mathbbm{R}^+$.
Moreover, the functions of the form
\begin{eqnarray}\label{eq:typefunc}
\sigma \mapsto  h_1(\sigma) = \frac{x e^{ - \frac{\left(x- m \sigma \right)^2}{2 \sigma} }}{\sqrt{\sigma^3}}  \quad \mathrm{and} \quad \sigma \mapsto h_2(\sigma) = \frac{ e^{ - \frac{\left(x- m \sigma \right)^2}{2 \sigma} }}{\sqrt{\sigma}} 
\end{eqnarray}
both admit a unique maximizer in $\mathbbm{R}^+$, which we denote by $s_1(x)$ and $s_2(x)$, respectively.
Specifically, we have
\begin{eqnarray*}
s_1(x) = \frac{\sqrt{4m^2x^2+9}-3}{2m^2}
 \quad \mathrm{and} \quad 
s_2(x) = \frac{\sqrt{4m^2x^2+1}-1}{2m^2}
\end{eqnarray*}
It is then straightforward to check that  $x \mapsto  h_1(s_1(x))$ and $x \mapsto  h_2(s_2(x))$ are both bounded functions on $[a,\infty)$.
Thus, there exist some constant $c>0$ upper bounding all functions of the types \eqref{eq:typefunc} with $m=\mu_m$ or $m=\mu_M$.
Therefore
\begin{eqnarray*}
 \int_a^\infty h(\sigma,0,x) q_0(x) \, \mdd x 
& \leq &
 \int_a^\infty   \left( \frac{x}{\sqrt{\sigma^3}}+\frac{ \mu_M - \mu_m }{\sqrt{\sigma}} \right) e^{ - \frac{\left(x- \mu_m \sigma \right)^2 \wedge \left(x- \mu_M \sigma \right)^2}{2 \sigma} } q_0(x)\, \mdd x \, , \\
 & \leq &
 c (1+\mu_M - \mu_m) \int_a^\infty  q_0(x)\, \mdd x \, \\
& \leq &
 c (1+\mu_M - \mu_m) \, .
\end{eqnarray*}
This concludes the proof.

\end{proof}


\subsection{Candidate fixed-point equation for physical time changes with blowups}\label{sec:fixed-point}
In the previous section, we show that assuming an inverse time change $\Psi \in \mIz$, the time-changed
problem \eqref{all:TCPDE} admits a unique density solution.
To fully solve dPMF dynamics in the time-changed picture, it remains to specify $\Psi$ self-consistently.
In the absence of blowups, this can be done straightforwardly by performing the change of variable $\sigma = \Phi(t) \Leftrightarrow t=\Psi(\sigma)$, 
which is then well defined:
\begin{eqnarray*}
\sigma = \Phi(t) = \nu_2 t + \lambda_2 F(t) = \nu_2 \Psi(\sigma) + \lambda_2 F\big(\Psi(\sigma)\big) = \nu_2 \Psi(\sigma) + \lambda_2 G(\sigma)
\end{eqnarray*}
Remembering that the cumulative rate depends on $\Psi$, the above equation yields the sought-after self-consistent characterization,
which is conveniently stated as a fixed-point equation for the inverse time change:
\begin{eqnarray}\label{eq:naiveFP}
\Psi(\sigma) = \big( \sigma - \lambda_2 G[\Psi](\sigma)\big) / \nu_2 \, , \quad \sigma \geq 0 \, .
\end{eqnarray}
However, the above fixed-point equation fails in the presence of blowups since we no longer have that $\Phi \circ \Psi = \mathrm{Id}$.
Actually, one can check that when full blowups occur, solutions $\Psi$ to \eqref{eq:naiveFP} become locally decreasing after $U=\Phi(T)$, where $T$
marks the blowup original time.
This would correspond to having nonmonotonic time changes, which is unphysical.

In \cite{TTPW}, we propose to remedy the possibility of unphysical solutions, 
by imposing that the fixed-point equation to bear on the space of nondecreasing functions. 
Specifically, we specify physical time changes as follows:

\begin{proposition}\label{prop:physFP}
Given generic initial conditions, we define the physical time change as the unique solution to the fixed-point equation
\begin{eqnarray}\label{eq:physFP}
\Psi(\sigma)
=
\frac{1}{\nu_2} \left[ \sup_{0 \leq \tau \leq \sigma} \big( \tau-\lambda_2 G[\Psi](\tau) \big) \right]_+ \, , \quad \sigma \geq 0 \, .
\end{eqnarray}
where $G[\Psi]$ refers to the solution of the quasi-renewal problem~\eqref{all:renew}
and where for all $x \in \mathbbm{R}$, the positive part of $x$ is denoted as $\left[ x \right]_+= x \vee 0$.
\end{proposition}

\begin{remark}  The above proposition calls for several remarks:

\begin{enumerate}
\item In the next section, we show that the fixed-point equation  \eqref{eq:physFP} naturally defines a contraction map, which
ensures the existence and uniqueness of solution dPMF dynamics by the Banach fixed-point theorem.

\item It is natural  to enforce that $\Psi$ be a nondecreasing function via a running supremum operator. 
The need to include a positive part operator in the fixed-point equation  \eqref{eq:physFP} is perhaps less obvious.
We will later see that this inclusion is necessary to treat the case of explosive initial conditions, which are allowed
by our notion of generic initial conditions given in Definition~\ref{def:TCinit}.
As stated in Remark~\ref{rem:intiBlow}, explosive initial conditions correspond $G_0(0)>0$.
The positive part ensures that we only consider  time change for which $\Psi(0)=0$, 
which holds for all $\Psi \in \mIz$. 

\item Solutions to \eqref{eq:physFP} unfold blowups in the time changed picture. 
To see how, remember that a full blowup happens at time $T$ in the original dynamics when the cumulative rate 
$F$ has a jump discontinuity $J_T = F(T)-F(T^-)$ so that $\Phi(T)-\Phi(T^-)= \lambda_2 J_T$, where $J_T$ represents the fraction of synchronously
spiking neurons.
Correspondingly, in the time-changed picture, a full blowup is represented by the fact that $\Psi=\Phi^{-1}$ has a flat section
on $[S,U]$, where $S=\Phi(T^-)$ and $U=\Phi(T)$ are referred to as blowup trigger time and blowup exit time, respectively.
Given a blowup trigger time $U$, Proposition~\ref{prop:physFP} specifies the corresponding blowup exit time as
\begin{eqnarray*}
U = \inf \{ \sigma > S \, \vert \, \sigma -S \geq \lambda_2 \big(G(\sigma) - G(S)\big)   \} \, .
\end{eqnarray*}

\item Proposition~\ref{prop:physFP} specifies only one notion of admissible blowup solutions and other notions have been proposed.
For instance, in \cite{dou2022dilating}, it was proposed to define blowup solutions by enforcing
a blowup exit criterion on the time-changed rate $g$ rather than the cumulative one.
Specifically, it was proposed to set the blowup exit time as the first time
when the blowup trigger condition ceases to hold
\begin{eqnarray*}
U = \inf \{ \sigma > S \, \vert \, g(\sigma) \leq 1/ \lambda_2    \} \, .
\end{eqnarray*}
while imposing $g=1/\delta$ during the blowup interval $[S,U]$.
However, such solutions are prone to unphysical eternal blowups for which $U=\infty$.
\end{enumerate}

\end{remark}

Having made the above preliminary remarks, let us define blowup times  for time changes
that are solutions of \eqref{eq:physFP}.
In the time-changed picture, the set of point for which a blowup occurs is
\begin{eqnarray*}
\left\{ \sigma \geq 0 \, \Big \vert \, \lim_{\tau \uparrow \sigma} g(\tau) =1/\lambda_2 \right\}  = \left\{ \sigma \geq 0 \, \Big \vert \, \lim_{\tau \uparrow \sigma} \partial_x p(\Psi(\tau), 0)/2 =1/\lambda_2 \right\} \, ,
\end{eqnarray*}
which includes nonempty intervals when the time change $\Psi$ has flat sections, i.e., when $\Phi$ has jump discontinuities.
Correspondingly, we define the set of full-blowup times $\mathcal{B}_\sigma$, which will play a central role,
as the interior of the set of blowup times.
As will be clear later, it will be more convenient to define $\mathcal{B}_\sigma$ in terms of the so-called excess function 
\begin{eqnarray}\label{eq:firstD}
\mathbbm{R}^+ \ni \sigma \mapsto D(\sigma) =  G[\Psi](\sigma)+  \big( \nu_2 \Psi(\sigma) -\sigma \big) /\lambda_2  \, .
\end{eqnarray}
which by \eqref{eq:physFP}, is necessarily a nonnegative, continuous function.
For nonexplosive initial conditions, by continuity of $D$, the
open set of full-blowup times can be written
\begin{eqnarray*}
\mathcal{B}_\sigma = \{ \sigma > 0 \, \vert \, D(\sigma) >0 \} = \bigcup_{k \in \mathcal{K}} (S_k,U_k) \, ,
\end{eqnarray*}
where  $\mathcal{K}$ is a countable index set,  $S_k$, $k \in \mathcal{K}$, are the full-blowup trigger times,
and $V_k$, $k \in \mathcal{K}$, are the corresponding full-blowup exit times.
Note that  for all $k \in \mathcal{K}$, this implies 
\begin{eqnarray*}
U_k = \inf \{ \sigma > S_k \, \vert \, \sigma -S_k \geq \lambda_2 (G(\sigma) - G(S_k))   \} \, .
\end{eqnarray*}

Given these definitions, the following proposition elucidates the nature of 
the time-changed dPMF dynamics during full blowups. 
Specifically, the proposition shows that during blowups, 
the dynamics of the time-changed representative process $Y_\sigma$ become noninteractive
in sense that it loses its McKean-Vlasov nature and follows a regular killed Brownian motion dynamics
with constant drift.

\begin{proposition}\label{prop:killedBrown}
Given a full-blowup interval $(S,U)$, $S<U$, and spatial initial condition $q_0(x)=q(S,x)$, the dPMF dynamics specified by \eqref{eq:physFP} solves 
\begin{eqnarray}\label{eq:buDyn}
\sigma \in (S,U):\quad
\partial_\sigma q 
=
\frac{\lambda_1}{\lambda_2}  \partial_x q + \frac{1}{2} \partial^2_x q \,   , \quad q(\sigma,0)=0 \, .
\end{eqnarray}
 
\end{proposition}

\begin{remark}
Note that there is no reset term in \eqref{eq:buDyn} so that there is no need to specify
the nonspatial component of the initial conditions.
\end{remark}

\begin{proof}
Given Definition~\ref{def:TCPDE}, it is enough to show that the  delayed rate $G_\epsilon$  
and the drift $\mu$ are constant over full-blowup intervals, with $\mu=-\lambda_1/\lambda_2$.

Suppose that $\sigma \in \mathcal{B}_\sigma$ is a full-blowup time
so that there is $k \in \mathcal{K}$ such that $S_k(\sigma)< \sigma < U_k(\sigma)$.
For all $\tau \in (S_k(\sigma),U_k(\sigma))$, we have $\Psi(\tau)=\Psi(\sigma)$.
Then, the fact that $\mu=-\lambda_1/\lambda_2$ just follows from \eqref{eq:TCdrift}
with $\Psi'=0$ on $\mathcal{B}_\sigma$.
Observe moreover that setting $\xi=\sigma-\eta$, we have
\begin{align*}
G_\epsilon(\sigma)
&=
\int_0^\infty G(\sigma-\eta) \, \mdd Q_\sigma(\eta) \, , \\
&=
\int_0^\infty G(\sigma-\eta) \, \mdd P_\epsilon(\Psi(\sigma)-\Psi(\sigma-\eta)) \, , \\
&=
\int_{-\infty}^\sigma G(\xi) \, \mdd P_\epsilon(\Psi(\sigma)-\Psi(\xi)) \, .
\end{align*}
Since by assumption, we only consider refractory-delay distributions $P_\epsilon$ 
with support in $[\epsilon,+\infty)$, introducing
\begin{eqnarray*}
\xi^\star(\sigma) = \sup \{ \xi \geq 0 \, \vert \, \Psi(\xi) \leq \Psi(\sigma) - \epsilon\} \, 
\end{eqnarray*}
allows us to further write
\begin{align}\label{eq:starGeps}
G_\epsilon(\sigma)
&=
\int_{(-\infty,\xi^\star(\sigma)]} G(\xi) \, \mdd P_\epsilon(\Psi(\sigma)-\Psi(\xi)) \, . 
\end{align}
Since  for all $\tau \in (S_k(\sigma),U_k(\sigma))$, we have $\Psi(\tau)=\Psi(\sigma)$, 
it must be that we also have $\xi^\star(\tau)=\xi^\star(\sigma)$.
so that $G_\epsilon(\sigma)=G_\epsilon(\tau)$ by  \eqref{eq:starGeps}.
This concludes the proof.
\end{proof}

As a corollary of the above proposition, we show that solutions to \eqref{eq:physFP} 
can only experience physical full blowups. 
By physical, we mean that these full blowups can be shown to represent 
the fraction of synchronously spiking neurons.
In particular, there is no eternal-blowup phenomenon.

\begin{corollary}\label{cor:buTime}
The duration of a full-blowup interval $(S_k,U_k) \in \mathcal{B}_\sigma$, for some $k \in \mathcal{K}$,
is such that $U_k -S_k= \lambda_2 J_k$,  where $J_k$, $0 < J_k \leq 1$, denotes the blowup size, i.e., 
the probability that a representative process $X_t$ spikes at $T_k=\Psi(S_k)=\Psi(U_k)$.
Moreover, given the spatial initial condition $q(S_k,x)$, the blowup size $J_k$ is uniquely defined as 
\begin{eqnarray}\label{eq:Pi}
J_k = \inf \left\{ P > 0 \, \bigg \vert \, P \geq \int_{0}^\infty H_\mathcal{B}(\lambda_2 P, x )  q(S_k,x ) \, \dd x \right\} \, .
\end{eqnarray}
with
\begin{eqnarray*}\label{eq:HFPT}
H_\mathcal{B}(\sigma,x)  
= \frac{1}{2} 
\left( 
\mathrm{Erfc} \left( \frac{x-\lambda_1\sigma/\lambda_2}{\sqrt{2 \sigma}}\right) + e^{2 \lambda_1 x/\lambda_2} \mathrm{Erfc} \left( \frac{x + \lambda_1 \sigma/\lambda_2}{\sqrt{2 \sigma}}\right) 
\right) \, .
\end{eqnarray*}
\end{corollary}

\begin{proof}
If $\sigma \in \mathcal{B}_\sigma$, there is $k \in \mathcal{K}$ such that $S_k< \sigma < U_k$.
Then, by Proposition~\ref{prop:killedBrown}, the quasi-renewal equation \eqref{eq:Grenew} loses 
its renewal character to read
\begin{align*}
G(\sigma) 
&=
G(S_k)+ \int_0^\infty  H(\sigma,S_k,x)   q(S_0(\sigma), x)  \, \mdd x   \, .
\end{align*}
Moreover, still by Proposition~\ref{prop:killedBrown},  $\sigma \mapsto  H(\sigma,S_0(\sigma),x)=\Prob{\tau_{x,S_0(\sigma)} \leq \sigma}$ 
is the cumulative distribution of  
\begin{eqnarray*}
\tau_{x,S_0(\sigma)} = \inf \{ \sigma \geq S_0(\sigma) \, \vert \, W_{S_k}=x , \,W_\sigma \leq (\lambda_1/\lambda_2) (\sigma-S_k)  \} \, ,
\end{eqnarray*}
where $W$ denotes the canonical Wiener process.
This cumulative distribution is known analytically to have the form
\begin{align*}
\Prob{\tau_{x,S_k} \leq \sigma} =  H_\mathcal{B}(\sigma-S_k,x)   \, .
\end{align*}
This shows that
\begin{align}\label{eq:noRenew}
G(\sigma) 
-
G(S_k) =  \int_0^\infty  H_\mathcal{B}(\sigma - S_k,x)   q(S_k, x)  \, \mdd x   \, .
\end{align}

Since for all $k \in \mathcal{K}$, the blowup exit time is defined as
\begin{align*}
U_k
&= \inf \left\{  \tau > S_k \, \big \vert \, \sigma-S_k \geq \lambda_2 (G(\sigma)-G(S_k))   \right\} \, ,
\end{align*}
setting $P=(\sigma-S_k)/\lambda_2 > 0$, we have
\begin{align*}
J_k 
&=(U_k-S_k)/\lambda_2 \, , \\
&=  \inf \left\{  (\sigma - S_k)/\lambda_2 > 0 \, \big \vert \,   (\sigma-S_k)/\lambda_2 \geq  G(\sigma)-G(S_k) \right\} \, , \\
&=  \inf \left\{  P > 0 \, \big \vert \, P  \geq  G(\lambda_2P + S_k)-G(S_k)     \right\} \, , \\
&=  \inf \left\{  P > 0 \, \big \vert \,   P \geq   \int_0^\infty  H_\mathcal{B}(\lambda_2P,x)   q(S_k, x)  \, \mdd x \right\} \, , 
\end{align*}
where the last equality follows from \eqref{eq:noRenew}.
Then, observing that $ H(\cdot,S_k,x) \geq 0$, we have
$$J_k \leq  \Vert H(\cdot,S_k,x) \Vert_{S_k, \infty} \int_0^\infty    q(S_k, x)  \, \mdd x \leq 1 \, ,$$
which concludes the proof.
\end{proof}

\subsection{Contraction argument}\label{sec:contract}

Given $\sigma>0$, remember that we denote by $\mIz([0,\sigma])$ the set a function $\Psi \in \mIz$ restricted to the interval $[0,\sigma]$.
The core result allowing to show the existence and uniqueness of buffered dPMF dynamics is the following:

\begin{proposition}\label{prop:contract}
Given generic initial conditions $(q_0,g_0)$ in $\mathcal{M}(\mathbbm{R}^+) \times \mathcal{M}(\mathbbm{R}^-)$,
the map $\mathcal{G} : \mIz \to \mIz$ specified by
\begin{eqnarray}\label{eq:mFz}
\Psi \mapsto \mathcal{G}[\Psi] = \left\{ \sigma \mapsto   \frac{1}{\nu_2} \left[ \sup_{0 \leq \tau \leq \sigma} \left( \tau-\lambda_2 G[\Psi](\tau) \right) \right]_+ \right\}
\end{eqnarray}
is a contraction on $\mIz([0,\sigma])$ for small enough $\sigma>0$.
\end{proposition}

To prove Proposition~\ref{prop:contract}, we will need two technical lemmas.
Both lemmas provide useful bounds on the cumulative distribution $H(\sigma,t,x)=\Prob{\sigma_{x,t}\leq \sigma}$ 
of the first-passage time $\sigma_{x,t}$ defined in \eqref{eq:firstPass}.
Recall that the first-passage time $\sigma_{x,t}$ depends on $\Psi$ via $\mu$, the deterministic drift of
the time-changed Wiener process $Y_\sigma$, whose relation to $\Psi$ is given in \eqref{eq:TCdrift}.
Both lemmas are obtained via elementary probabilistic proofs, which are given in Section~\ref{sec:buffcontract}.

The first lemma, Lemma~\ref{lem:Hprop}, allows to compare cumulative distribution $H(\sigma,t,x)$ obtained for distinct choice
of time change, e.g., $\Psi_a$, $\Psi_b \in \mId$. 
For ease of notation, indexation by $a$ and $b$ will indicate throughout the proof dependence on $\Psi_a$ and $\Psi_b$, respectively.
For instance, we write $H_a=H[\Psi_a]$ and $H_b=H[\Psi_b]$.

\begin{lemma}\label{lem:Hprop}
\begin{subequations}
\label{all:lemHprop}
Set $C=2|\nu_1-\lambda_1\nu_2/\lambda_1|$ and consider $0<\tau_a, \, \tau_b\leq\sigma$  and $\Psi_a,\Psi_b \in \mIz([0,\infty))$ with $\Vert \Psi_a-\Psi_b \Vert_{0,\sigma} \leq\beta$. 
Then for all $x \geq 0$:
    \begin{align}\label{eq:lemHprop1}
    \vert H_a(\sigma,\tau,x)-H_b(\sigma,\tau,x) \vert
    \leq 
    &  \big[ H_a(\sigma,\tau;x)-H_a(\sigma,\tau,x+C\beta) \big] \\& \vee  \big[ H_a(\sigma,\tau,x-C\beta) - H_a(\sigma,\tau,x)  \big] \nonumber
\end{align}
\begin{align}\label{eq:lemHprop2}
        \vert H_a(\sigma,\tau_a,x)-H_a(\sigma,\tau_b,x) \vert \leq  
        & \big(  \big[ H_a(\sigma,\tau_b,x)-H_a(\sigma,\tau_b, x+C\vert \tau_b-\tau_a \vert /\nu_2 )  \big] \\
        & \vee   \big[ H_a(\sigma,\tau_b, x-C \vert \tau_b-\tau_a \vert /\nu_2 ) - H_a(\sigma,\tau_b,x)  \big] \big) \nonumber\\
        &+ \Vert h_a(\cdot,\tau_a,x)\Vert_{\tau_a, \sigma} \vert \tau_b-\tau_a \vert  \, , \nonumber
    \end{align} 
    with $h_a(\cdot,\tau,x) = \partial_\sigma H_a(\cdot,\tau,x)$.
 \end{subequations}    
\end{lemma}

The above lemma allows one to substitute the comparison of $H_a(\sigma,\tau,x)$ and $H_b(\sigma,\tau,x)$ with a
comparison of the same function, say,  $x \mapsto H_a(\sigma,\tau,x)$ evaluated for distinct value of $x$.
The second lemma, Lemma~\ref{lem:FP}, provides crude yet useful Lipschitz bounds on the function $x \mapsto H(\sigma,\tau,x)$.
 
\begin{lemma}\label{lem:FP}
\begin{subequations}
\label{all:lemFP}
For all $\sigma > \tau>0$ and all $x,y>0$, we have
\begin{align}\label{eq:lemFP1}
0 \leq \sigma-\tau<1/(8 \mu_M^2) \quad \Rightarrow \quad |H(\sigma,\tau, x)-H(\sigma,\tau, y)|  \leq \frac{4 |x-y|}{\sqrt{\sigma-\tau}}
\end{align}
 \begin{align}\label{eq:lemFP2}
(\sigma-\tau)^{1/3}+\int_\tau^\sigma \mu(\xi)\, \mdd \xi \leq x \leq y \quad \Rightarrow \quad |H(\sigma,\tau,x)-H(\sigma,\tau,y)|\leq 24\sqrt{\sigma-t}|x-y|
\end{align}
\end{subequations}
\end{lemma}

\begin{proof}[Proof of Proposition~\ref{prop:contract}]
We proceed in several steps: 

$(i)$
The fact that $\mathcal{G}$ induces a mapping $\mIz \rightarrow \mIz$ directly follows from its definition in \eqref{eq:mFz}.
Let us justify that this mapping loses its renewal character for small enough $\sigma>0$, i.e., for $0<\sigma<\Psi^{-1}(\epsilon))$.
To see this, remember that for all $\Psi \in \mIz$, $0 \leq w_\Psi \leq 1/\nu_2$ and  $\Psi(0) =0$, so that  $ \nu_2 \epsilon \leq \Psi^{-1}(\epsilon)$.
Furthermore, observe that $\Psi^{-1}(\epsilon)$ is a lower bound to the time-changed refractory periods after zero.
This implies that a dPMF representative process may not fire more than once on $[0,\Psi^{-1}(\epsilon))\supset  [0,\nu_2 \epsilon]$.
As a result, the renewal-type equation \eqref{eq:Grenew} characterizing the cumulative function $G$ 
for a fixed choice of $\Psi$ loses its renewal character on $[0,\nu_2 \epsilon]$ and reads
\begin{eqnarray}
G(\sigma)
&=&  \int_{0}^\infty  H(\sigma,0,x)   q_0(x) \, \dd x + \int_0^\sigma H(\sigma,\tau,\Lambda) \, \dd G_{0,\epsilon}(\tau) \, , \nonumber\\
&=&  \int_{0}^\infty  H(\sigma,0,x)   q_0(x) \, \dd x + \int_0^\sigma H(\sigma,\tau,\Lambda) \, \dd F_{0,\epsilon}(\Psi(\tau))  \label{eq:GNotRenew} \, . 
\end{eqnarray}
In the equation above, the loss of renewal character is indicated by the fact that the second integral term depends on $F_{0,\epsilon}$,
which is given as an initial condition.
As $H(\sigma,\tau,x)$ only depends on $\Psi$ via $\{ \Psi(\xi) \}_{\tau \leq \xi \leq \sigma}$,  \eqref{eq:GNotRenew} 
shows that specifying $\mathcal{G} [\Psi]$ on $[0, \sigma]$, $0<\sigma<\nu_2 \epsilon$, only requires knowledge of $\Psi$ on $[0,\sigma]$.
Therefore, $\mathcal{G}$ is a mapping $\mIz([0,\sigma]) \rightarrow \mIz([0,\sigma])$ for all $\sigma$, $0 < \sigma \leq \nu_2 \epsilon$.

$(ii)$
Let us now consider two functions $\Psi_a$ and $\Psi_b$ in  $\mIz$.
By $(i)$, both cumulative functions $G_{a}$ and  $G_{b}$ satisfy the nonrenewal 
equation \eqref{eq:GNotRenew} on the interval $ [0, \nu_2 \epsilon ]$ with identical initial conditions $(q_0,F_0)$.
Our goal is to show that for small enough $\sigma$, we have  $\big \vert \mathcal{G}[\Psi_a](\sigma) - \mathcal{G}[\Psi_b](\sigma) \big \vert \leq \Vert \Psi_a - \Psi_b \Vert_{[0,\sigma]}/2$.
For all $0 \leq \sigma \leq \nu_2 \epsilon$, we have
\begin{align*}
     G_{a}(\sigma)-G_{b}(\sigma) = I_1(\sigma)+I_2(\sigma) +I_3(\sigma)
\end{align*}
where the three integral terms $I_1$, $I_2$, and $I_3$ are defined as
\begin{align*}
I_1(\sigma) &=  \int_0^\infty \big[H_a(\sigma,0,x)-H_b(\sigma,0,x) \big] q_0(x) \, \mdd x \, , \\
I_2(\sigma) &=  \int_0^\sigma  \big[ H_a(\sigma,\tau,\Lambda)-H_b(\sigma,\tau,\Lambda) \big]  \, \mdd F_{0, \epsilon}(\Psi_a(\tau)) \, , \\
I_3(\sigma) &= \int_0^\sigma H_b(\sigma,\tau,\Lambda) \mdd \big[ F_{0, \epsilon} (\Psi_a(\tau))- F_{0, \epsilon} (\Psi_b(\tau)) \big]  \, .
\end{align*}
In the following, we upper bound each of these integral terms.

\emph{Integral term $I_1(\sigma)$.}
Set $\beta(\sigma)= \Vert \Psi_a - \Psi_b \Vert_{[0,\sigma]}$ for notation convenience.
Observe that for all $\Psi_a, \Psi_b \in \mIz$, and for all $\sigma>0$, we have
\begin{eqnarray*}
 \beta(\sigma) = \Vert \Psi_a - \Psi_b \Vert_{[0,\sigma]} \leq \sigma/\nu_2  \, , 
\end{eqnarray*}
so that $ \beta(\sigma)$ can be made arbitrarily small when $\sigma \to 0$.
With that in mind, we split $I_1(\sigma)=I_{11}(\sigma)+I_{12}(\sigma)$ on two domains: $[0,2\sigma^{1/3}]$ and $[2\sigma^{1/3},\infty)$.
Furthermore, we assume $\sigma \leq \sigma_0=\nu_2 \epsilon \wedge (\alpha/2)^3 \Rightarrow 2\sigma^{1/3}\leq \alpha$, where
$\alpha>0$ is such that for every $x$ on $[0,\alpha]$ implies $q_0(x)\leq Ax$ for some constant $A$.
In order to bound $I_{11}(\sigma)$ and  $I_{12}(\sigma)$, we invoke \eqref{eq:lemHprop1}  from Lemma~\ref{lem:Hprop}  
with $C=2 \vert \nu_1-\lambda_1\nu_2/\lambda_1 \vert$ and write
\begin{align}\label{eq:ab123}
    \left \vert H_a(\sigma,0,x)-H_b(\sigma,0,x) \right \vert \nonumber
    \leq &
    \big[ H_a(\sigma,0,x)-H_a(\sigma,0,x+C \beta(\sigma)) \big]\\
   & \vee  \big[ H_a(\sigma,0,x-C \beta(\sigma)) - H_a(\sigma,0,x) \big]  \, .
\end{align}

For $I_{11}(\sigma)$, set $\sigma_{11}= \sigma_0 \wedge1/(8 \mu_M^2)$, so that
we can apply \eqref{eq:lemFP1} from Lemma~\ref{lem:FP} to get for  $\sigma \leq \sigma_{11}$,
\begin{align*}
    \left \vert H_a(\sigma,0,x)-H_a(\sigma,0,x \pm C \beta(\sigma)) \right \vert    \leq \frac{4 C \beta(\sigma)}{\sqrt{\sigma}} \, ,
\end{align*}
Combined with \eqref{eq:ab123}, this implies that
\begin{align*}
\vert I_{11}(\sigma) \vert
& \leq \int_0^{2\sigma^{1/3}} \vert H_a(\sigma,0;x)-H_b(\sigma,0;x) \vert  q_0(x) \, \mdd x \, , \\
&\leq \int_0^{2\sigma^{1/3}} \frac{4 C \beta(\sigma)}{\sqrt{\sigma}}  (Ax)\, \mdd x \, ,\\
&\leq 8 A C \beta(\sigma) \sigma^{1/6} \, .
\end{align*}

For $I_{12}(\sigma)$, setting $\sigma_{12}=(\mu_M+C/\nu_2)^{-3/2}$, observe that for all $\sigma$, $0<\sigma\leq\sigma_{12}$,
 we have 
 \begin{eqnarray*}
2\sigma^{1/3}>\sigma^{1/3}+(\mu_M+C/\nu_2) \sigma > \sigma^{1/3}+\int_0^\sigma \mu(\xi) \, \mdd \xi + C \beta(\sigma) \, .
\end{eqnarray*}
Then, set $\sigma_1=\sigma_{11} \wedge \sigma_{12}$ and for all $\sigma$, $0 \leq \sigma \leq \sigma_1$, and for $x\geq 2\sigma^{1/3}$,
apply \eqref{eq:lemFP2} from Lemma~\ref{lem:FP}  to get 
\begin{align*}
   &x - C\beta(\sigma) \geq 2 \sigma^{1/3}- C\beta(\sigma) >\sigma^{1/3}+\int_0^\sigma \mu(\xi) \, \mdd \xi \quad \Rightarrow  \\
   &\hspace{40pt}
    \left \vert H_a(\sigma,0,x)-H_a(\sigma,0,x \pm C \beta(\sigma)) \right \vert    \leq 24 \sqrt{\sigma}\beta(\sigma) \, .
\end{align*}
Combined with \eqref{eq:ab123}, this implies that
\begin{align*}
\vert I_{12}(\sigma) \vert
    & \leq \int_{2\sigma^{1/3}}^{\infty}   \left \vert H_a(\sigma,0,x)-H_b(\sigma,0,x) \right \vert  q_0(x) \, \mdd x \, , \\
    &\leq 24 C \sqrt{\sigma} \beta(\sigma) \int_{2\sigma^{1/3}}^{\infty}q_0(x)\mdd x \, , \\
    & \leq 24 C \sqrt{\sigma} \beta(\sigma) \, .
\end{align*}

Altogether, this shows that for all $\sigma$, $0 \leq \sigma \leq \sigma_1$, we have
\begin{align*}
\vert I_1(\sigma) \vert
= \vert I_{11}(\sigma) \vert + \vert I_{12}(\sigma) \vert 
\leq 
\left(  8 C   \sigma^{1/6} + 24 C \sqrt{\sigma}  \right) \beta(\sigma) \, .
\end{align*}

\emph{Integral term $I_2(\sigma)$.}
Invoking \eqref{eq:lemHprop1}  from Lemma~\ref{lem:Hprop} again
with $C=2 \vert \nu_1-\lambda_1\nu_2/\lambda_1 \vert$, we have
\begin{align}\label{eq:ab123b}
    \left \vert H_a(\sigma,\tau,\Lambda)-H_b(\sigma,\tau,\Lambda) \right \vert \nonumber
    \leq &
    \big[ H_a(\sigma,\tau,\Lambda)-H_a(\sigma,\tau, \Lambda+C \beta(\sigma)) \big]\\
   & \vee  \big[ H_a(\sigma,\tau,\Lambda-C \beta(\sigma)) - H_a(\sigma,\tau,\Lambda) \big]  \, .
\end{align}
There exists $\sigma_\Lambda>0$ such that for every $\sigma$, $0<\sigma\leq \sigma_\Lambda$, we have:
$$\Lambda>  \sigma^{1/3}+ (\mu_M +C/ \nu_2) \sigma > \sigma^{1/3}+\int_0^\sigma \mu(\xi) \, \mdd \xi +C \beta(\sigma) \, .$$
Set $\sigma_2=\sigma_\Lambda \wedge \sigma_1$.
By \eqref{eq:lemFP2} from Lemma \eqref{lem:FP}, for all $\sigma>0$ such that  $\sigma\leq\sigma_2$, we have 
\begin{align*}
 \left \vert H_a(\sigma,\tau,\Lambda)-H_a(\sigma,\tau,\Lambda \pm C \beta(\sigma)) \right \vert
\leq
24 C \beta(\sigma)\sqrt{\sigma-\tau}
\leq 
24 C \beta(\sigma) \sqrt{\sigma}  \, .
\end{align*}
Combined with \eqref{eq:ab123b}, this implies that
\begin{align*}
\vert I_2(\sigma) \vert
&=
\int_0^\sigma \vert H_a(\sigma,\tau,\Lambda)-H_b(\sigma,\tau,\Lambda) \vert \, \mdd F_0(\Psi_a(\tau)-\epsilon) \, , \\
&\leq 24 C  \sqrt{\sigma}  \beta(\sigma) (F_0(\Psi_a(\sigma)-\epsilon) - F_0(-\epsilon)) \, , \\
&\leq 24 C  \sqrt{\sigma}  \beta(\sigma)  \, ,
\end{align*}
where we have used the fact that for all $\sigma \leq \sigma_2 \leq \nu_2 \epsilon$, we necessarily have $\Psi_a(\sigma) \leq \epsilon$ so that
\begin{eqnarray*}
F_0(\Psi_a(\sigma)-\epsilon) - F_0(-\epsilon) \leq -F_0(-\epsilon) = \int_{-\epsilon}^0 f(s) \, \mdd s \leq 1 \, .
\end{eqnarray*}

\emph{Integral term $I_3(\sigma)$.}
Pick $\sigma$, $0 \leq \sigma \leq \sigma_2$. 
Since $ H_b(\sigma,\sigma,\Lambda)=0$ and $\Psi_a(0)=\Psi_b(0)$, performing integration by parts,
which is justified by Lemma~\ref{lem:HLipsch}, yields
\begin{align*}
I_3(\sigma) 
&= \int_0^\sigma H_b(\sigma,\tau,\Lambda) \mdd \big[ F_{0, \epsilon} (\Psi_a(\tau) )- F_{0, \epsilon} (\Psi_b(\tau)) \big]  \, , \\
&= \int_0^\sigma \partial_\tau H_b(\sigma,\tau,\Lambda)  \big[ F_{0, \epsilon} (\Psi_a(\tau) )- F_{0, \epsilon} (\Psi_b(\tau)) \big] \mdd \tau  \, . 
\end{align*}
Moreover, assuming that the delay distribution $P_\epsilon$ has a smooth density denoted by $p_\epsilon \in C_\infty(\mathbbm{R}^+)$, 
as a convolution, $F_{0, \epsilon} = F_0 * p_\epsilon \in C_\infty(\mathbbm{R}^+)$ is also a smooth function.
Therefore, we have
\begin{align*}
\vert F_{0, \epsilon} (\Psi_a(\tau) )- F_{0, \epsilon} (\Psi_b(\tau)) \vert 
\leq \Vert F'_{0, \epsilon}  \Vert_{0, \Psi_a(\tau) \vee \Psi_b(\tau)}  \vert \Psi_b-\Psi_a \vert  
 \leq \Vert F'_{0, \epsilon}  \Vert_{0,  \epsilon}  \beta(\sigma) \, ,\\
\end{align*}
so that 
\begin{align*}
\vert I_3 \vert 
\leq 
- \int_0^\sigma \partial_\tau H_b(\sigma,\tau,\Lambda)\dd\tau\; \Vert F'_{0, \epsilon}  \Vert_{0,  \epsilon} \beta(\sigma) 
\leq 
D  \Vert F'_{0, \epsilon}  \Vert_{0,  \epsilon}  \, \sigma \beta(\sigma)  \, ,
\end{align*}
where we have used the fact that by Lemma~\ref{lem:HLipsch}, $\partial_\tau H$ is uniformly bounded 
over $\{ (\sigma,\tau) \, \vert \, 0 \leq \tau \leq \sigma \}$ by a constant $D$ 
that depends only on $\mu_m$, $\mu_M$, and $\Lambda$,

\emph{Conclusion.}
Altogether, for all $\sigma$, $0\leq \sigma \leq \sigma_2$, we have 
\begin{align*}
\vert G_{a}(\sigma)-G_{b}(\sigma) \vert
&\leq 
 \vert I_1(\sigma) \vert + \vert I_2(\sigma) \vert + \vert I_3(\sigma) \vert  \leq K^\star(\sigma) \beta(\sigma) \, , 
 \end{align*}
 with
 \begin{align*}
 K^\star(\sigma) =  8 C   \sigma^{1/6} + 48  C  \sqrt{\sigma} + D  \Vert F'_{0, \epsilon}  \Vert_{0,  \epsilon} \, \sigma  
 \xrightarrow[\sigma \downarrow 0]{}  
0 \, . 
\end{align*}
In particular, there is $\sigma_{1/2}>0$ such that
$$ 0 \leq \sigma \leq \sigma_{1/2} \quad \Rightarrow \quad  K^\star(\sigma) \leq \nu_2/(2\lambda_2) \, , $$
so that for all $\sigma$, $0 \leq \sigma \leq\sigma_{1/2}$, we have
\begin{eqnarray*}
\big \vert \mathcal{G}[\Psi_a](\sigma) - \mathcal{G}[\Psi_b](\sigma) \big \vert 
&=& 
\left \vert \left[ \sup_{0 \leq \tau \leq \sigma} \left( \frac{\tau-\lambda_2 G_a(\tau)}{\nu_2}  \right) \right]_+ - \left[\sup_{0 \leq \tau \leq \sigma} \left( \frac{\tau-\lambda_2 G_b(\tau)}{\nu_2}  \right) \right]_+ \right \vert \, , \\
&\leq& 
\left \vert  \sup_{0 \leq \tau \leq \sigma} \left( \frac{\tau-\lambda_2 G_a(\tau)}{\nu_2}  \right)  - \sup_{0 \leq \tau \leq \sigma} \left( \frac{\tau-\lambda_2 G_b(\tau)}{\nu_2}  \right)  \right \vert \, , \\
&\leq& 
\frac{\lambda_2}{\nu_2} \sup_{0 \leq \tau \leq \sigma} \Big \vert   G_a(\tau) -  G_b(\tau)  \Big \vert \, , \\
&\leq& 
\frac{\lambda_2}{\nu_2}  \left( \frac{\nu_2}{2 \lambda_2}  \beta(\sigma) \right) \, , \\
&\leq& 
 \frac{1}{2}  \Vert \Psi_a - \Psi_b \Vert_{0,\sigma} \, ,
\end{eqnarray*}
which directly implies 
\begin{eqnarray*}
\Vert\mathcal{G}[\Psi_a](\sigma) - \mathcal{G}[\Psi_b](\sigma)\Vert_{0,\sigma}
\leq
 \frac{1}{2}  \Vert \Psi_a - \Psi_b \Vert_{0,\sigma} \, .
\end{eqnarray*}
This shows that there is $\sigma_{1/2}>0$ such that $\mathcal{G}$ is a contraction on $\mIz([0,\sigma_{1/2}])$. 
\end{proof}

\begin{remark}
The contraction argument presented above is classical and closely related 
to PDE techniques used in~\cite{Carrillo:2013} and probabilistic arguments developed in~\cite{Delarue:2015}.
Our approach differs from~\cite{Carrillo:2013} as it considers a fixed-point equation that 
can be interpreted probabilistically in terms of a quasi-renewal process, 
whereas \cite{Carrillo:2013} considers an equivalent fixed-point equation derived in the context of a free-boundary problem.
Adopting a quasi-renewal probabilistic interpretation is necessary to formulate our proposed 
physical fixed-point  equation \eqref{eq:physFP} by means of a running supremum operator.
Moreover, our arguments differ from~\cite{Delarue:2015} for being specifically
developed to account for the physical occurrence of full blowups in mean-field models, which also includes
random refractory delays (as opposed to delayed interactions).
\end{remark}

Proposition~\ref{prop:contract} directly implies the following existence and uniqueness theorem:

\begin{theorem}\label{th:globSol}
For all  $\nu_1, \nu_2, \lambda_1, \lambda_2, \Lambda, \epsilon>0$, given generic initial conditions, there is a unique global solution $\Psi$ to the dPMF fixed-point problem \eqref{eq:physFP} in $\mIz=\mIz([0,\infty))$.
Furthermore, this solution specifies a solution to the dPMF problem \eqref{eq:mainProb1} on as $\mathbbm{R}^+ \times \mathbbm{R}^+ \ni (t,x) \mapsto p(x,t)=q(\Phi(t),x)$, where $\Phi$ is the right-continuous inverse of $\Psi$ and where $(\sigma,x) \mapsto q(\sigma,x)$ is unique solution to the time-changed problem \eqref{all:TCPDE} for $\Psi$.
\end{theorem}

\begin{proof}[Proof of Theorem~\ref{th:globSol}] 
This is a direct consequence of Banach's fixed-point theorem. For completeness, we recapitulate the arguments in the following.
%

By Proposition~\ref{prop:contract}, the fixed-point equation \eqref{eq:fixedPoint} can be written as $\Psi = \mathcal{G}[\Psi]$, 
where the operator  $\mathcal{G}$ is such that for generic initial conditions,
its restriction to $\mIz([0,\sigma])$ induces a contraction $\mathcal{G}: \mIz([0,\sigma]) \to  \mIz([0,\sigma])$ for small enough $\sigma>0$.
By Banach's fixed-point theorem, the fixed-point equation \eqref{eq:fixedPoint} admits a unique local solution $\Psi$ in $\mIz([0,\sigma])$.

A global solution is obtained by continuation, which by Proposition~\ref{prop:contract}, is possible for all $\sigma$ with initial conditions
given by $g_0=\{ g(\tau) \}_{ \tau \leq \sigma}$ and $q_0=q(\sigma,\cdot)$.
Indeed, these specify generic initial conditions as long as the boundedness condition \eqref{eq:TCinitb} holds.
It turns out that continuation is possible for all $\sigma>0$ as by Proposition~\ref{prop:boundedness}, $\sigma \mapsto \partial_\sigma G(\sigma)=g(\sigma)=\partial_x q(\sigma,0)/2$ 
must be uniformly bounded with respect to $\Psi \in \mIz$ over all compact interval $[0,T]$, $T>0$.
This shows that there is a unique global solution $\Psi \in \mIz$ to the fixed-point equation \eqref{eq:fixedPoint}.
and therefore, a unique time change $\Psi=\Phi^{-1}: \mathbbm{R}^+ \to  \mathbbm{R}^+$
specifying the dPMF dynamics solving problem \eqref{eq:mainProb1}.

Finally, we show that the resulting dPMF dynamics is globally defined, which only requires to show that  $\lim_{\sigma \to \infty} \Psi(\sigma) = \infty$.
By conservation of probability, observe that for all $t \geq 0$, we have
\begin{align*}
F(t) - F_\epsilon(t) \leq \int_0^\infty p(t,x) \, \mdd x + F(t) -F_\epsilon(t) = 1 \, .
\end{align*}
Moreover, we also have for all $t \geq 0$
\begin{align*}
F(t+\epsilon)-F(t)
&= \int_0^\infty H(\Phi(t+\epsilon), \Phi(t), x) p(t,x) \, \mdd x + \int_t^{t+\epsilon} H(\Phi(t+\epsilon), \Phi(s), x) \,  \mdd F_\epsilon(s) \, , \\
& \leq  \int_0^\infty  p(t,x) \, \mdd x + F_\epsilon(t+\epsilon) - F_\epsilon(t)\, , \\
& = \left(  \int_0^\infty  p(t,x) \, \mdd x + F(t) - F_\epsilon(t) \right) + F_\epsilon(t+\epsilon) - F(t) \, , \\
& = 1 + F_\epsilon(t+\epsilon) - F(t) \, ,
\end{align*}
Since the refractory-delay distribution $P_\epsilon$ has support in $[\epsilon, \infty)$, we have
\begin{eqnarray*}
F_\epsilon(t+\epsilon)  = \int_{[\epsilon,\infty)} F(t+\epsilon-s) \, \mdd P_\epsilon(s) \leq F(t) \, ,
\end{eqnarray*}
where we use the fact that $F$ is nondecreasing. 
This shows that for all $t \geq0$, we must have $F(t+\epsilon)-F(t) \leq 1$.
Thus, we have
\begin{align*}
\Phi(t)  
&= \nu_2 t + \lambda_2 F(t)  \, , \\
&=  \nu_2 t + \lambda_2  \left( F(t)-F(\epsilon \lfloor t/\epsilon \rfloor) + \sum_{n =0}^{ \lfloor t/\epsilon \rfloor}\big( F((n+1)\epsilon)-F(n\epsilon) \big) \right) \, , \\
&\leq  \nu_2 t + \lambda_2 \big( t/\epsilon +1 \big)  \, , 
\end{align*}
which upon setting $\sigma=\Phi(t) \Leftrightarrow t = \Psi(\sigma)$, directly implies
\begin{align*}
 \Psi(\sigma) \geq \frac{\sigma - \lambda_2}{\nu_2 + \lambda_2/\epsilon} \xrightarrow[\sigma \to \infty]{} \infty \, . 
\end{align*}

\end{proof}

\begin{remark} \label{rem:globSol}
It is worth making a few remarks about Theorem~\ref{th:globSol}:
\begin{enumerate}
\item
Observe that Theorem~\ref{th:globSol} only specifies one of many admissible blowup solutions to the dPMF problem  \eqref{eq:mainProb1}
and that this solution is uniquely defined as the solution to the fixed-point problem~\eqref{eq:physFP}.
The purpose of the remaining part of this manuscript is to show that the solution given by Theorem~\ref{th:globSol} is the unique physical one.
To that end, we will show that these solutions can alternatively be obtained via a limiting process acting 
on certain regularized dynamics, the so-called buffered dPMF dynamics, which admits a well-defined, physical notion of blowups.
\item Theorem~\ref{th:globSol} is valid for generic initial conditions, which means that it covers
the case of initial blowups for which $G_0(0)>0$ in \eqref{eq:TCinita}.
This is a crucial point as in the time-changed picture, global solutions are obtained by continuation during blowup intervals $(S,U)$, where we have 
$$\Psi(\sigma)=\frac{1}{\nu_2} \sup_{0 \leq \tau \leq \sigma} \big( \tau-\lambda_2 G(\sigma) \big)> \frac{\sigma-\lambda_2 G(\sigma) }{\nu_2} \, , \quad S<\sigma<U \, .$$
Observe that the above condition implies that the $\sigma$-shifted version of the integrability condition \eqref{eq:TCinita} reads
$$G_\sigma(0) = \sup_{\tau \leq \sigma} \left( \frac{\tau-\sigma}{\lambda_2} - (G(\tau)-G(\sigma)) \right) = \frac{\nu_2}{\lambda_2} \left(\Psi(\sigma)-\frac{\sigma-\lambda_2 G(\sigma)}{\nu_2} \right) > 0 \, ,$$
confirming the need to consider generic initial conditions.
\item
Recall that the reason to define $\Phi=\Psi^{-1}$ as the right-continuous inverse of $\Phi$ in the above theorem
is for the cumulative function $F$ to be \emph{c\`adl\`ag}, a property that we only require to state 
the stochastic equation \eqref{eq:stochEqMF} governing a representative dPMF dynamics.
\item Observe that the benefit of considering distributed delays with a smooth density $p_\epsilon$ 
supported on $[\epsilon, \infty)$, $\epsilon>0$, is two-fold.
On one hand, it regularizes the reset processes via the reset rate $f_{\epsilon}$, which is
key in establishing the contraction result from Proposition ~\ref{prop:contract}.
On the other hand, the fact that a representative neuron cannot spike more than once
over a refractory period $\epsilon$ provides a simple bound to show the global
nature of the resulting dPMF solution in Theorem~\ref{th:globSol}.
We will later show that refractory delays also play an important role in avoiding eternal blowups in Section~\ref{sec:compRes}.
\end{enumerate}
\end{remark}


\section{Regularization via buffer mechanism}\label{sec:bufferAnalysis}

In this section, we analyze the buffer mechanism given in Definition~\ref{def:mainBuffer2} that produces $\delta$-buffered dynamics, which will always be well defined and will exhibit regularized, buffered blowups.
In Section~\ref{sec:buffDyn}, we justify the form of the $\delta$-buffered PDE problem for dPMF dynamics given in Definition~\ref{def:mainBuffer1}.
In Section~\ref{sec:buffMech}, we formalize the buffer mechanism as a mapping between input and output functions, independent of the PDE context of dPMF dynamics.
In Section~\ref{sec:runExtr}, we give useful characterizations of this mapping in terms of $\delta$-tilted, running extrema.


\subsection{The buffered regularized dynamics}\label{sec:buffDyn}

Here, we rigorously state the buffered version of the dPMF problem \eqref{eq:mainProb1}, 
which amalgamates Definitions~\ref{def:mainBuffer1} and~\ref{def:mainBuffer2} given in the introduction.
To understand the resulting definition, it is helpful to have a mental picture in mind. 
This mental picture is inspired by analogy with the physical system graphically depicted in Fig.~\ref{fig:buffMech}.
In this physical analogy, we think of interactions as a process by which particles exchange water 
and we impose that this exchange of water goes through a reservoir.
There is no limit to the reservoir intake which proceeds at rate $f$,
 as if water was poured in an open container from the top.
 By contrast, the reservoir outflow rate is capped, as
if water was pumped at the bottom of the reservoir by
a pump with finite capacity. 
If the intake is larger than the pump capacity, i.e.,  $f>1/\delta$ 
or if the reservoir holds excess water, i.e.,  $E>0$,
 the pump operates  at the maximum outflow $\tf=1/\delta$ and 
 excess water accumulates in the reservoir with rate $\partial_tE = f-1/\delta$.
 Otherwise, the pump operates  at the same rate as the intake rate, i.e., $\tf=f$,
 and the reservoir does not accumulate excess water. 
Blowups are defined as these periods of time
when the reservoir transiently holds water, i.e., when $E>0$.
The following definition formalizes the buffered dynamics informally discussed above:

\begin{definition}\label{def:bufferPDE} 
\begin{subequations}
\label{all:bufferPDE}
Given nonexplosive initial conditions $(p_0,f_0)$ in $\mathcal{M}(\mathbbm{R}^+) \times \mathcal{M}((-\infty,0))$ and $\delta>0$, the $\delta$-buffered PDE problem associated to a dPMF dynamics consists in finding the density function $(t,x) \mapsto p(t,x)= \mdd \Prob{0<X_t \leq x} / \mdd x$ solving 
\begin{eqnarray}\label{eq:bufferPDE}
\partial_t p  &=& \big(\nu_1 + \lambda_1 \tf (t) \big) \partial_x p + \big(\nu_2 + \lambda_2 \tf (t)  \big) \partial^2_{x} p /2  + f_\epsilon(t) \delta_{\Lambda}  \, ,  \quad \quad p(t,0) = 0 \, ,
\end{eqnarray}
on $[0,T) \times [0, \infty)$ for some (possibly infinite) $T>0$ where $ \tf $ denotes the $\delta$-buffered rate 
\begin{eqnarray}\label{eq:buffer}
\tf (t)= (1-B(t)) f(t) + B(t) /\delta \, .
\end{eqnarray}
The blowup indicator function $B$ is specified as a $\{0,1\}$-valued function by
\begin{eqnarray}\label{eq:indicator}
B(t)  = \mathbbm{1}_{\{ f(t) > 1/\delta \} \cup \{E(t) > 0\}}    \, ,
\end{eqnarray}
where the interaction excess function $E$ satisfies
\begin{eqnarray}\label{eq:excess}
\partial_t E(t)  &=&  B(t) (f(t) -1/\delta)  \, , \quad E(0) = 0\, .
\end{eqnarray}
\end{subequations}
\end{definition}


\begin{remark}
A few remarks about our buffered dPMF dynamics are in order:
\begin{enumerate} 
\item 
The assumption that the reservoir is initially empty $E(0)=0$ follows from the fact that we consider nonexplosive initial conditions (see Definition~\ref{def:mainProb2}).
It is important to note, however, that $\delta$-buffered dynamics are well-defined independent of the nonexplosive character of the initial conditions.
In fact, all the results derived for the $\delta$-buffered mechanism will be valid only assuming
$$0\leq E(0)<\infty \quad \mathrm{and} \quad \limsup_{x \to 0^+} \frac{p_0((0,x])}{x} < \infty \, .$$
The positive initial value $E(0)>0$ can be interpreted in the time-changed picture via the notion of generic initial conditions $(q_0,g_0)$,
which allows for initialization during a blowup episode (see Definition~\ref{def:TCinit}).
Specifically, we will see that considering generic initial conditions $(q_0,g_0)$ imposes to set $E(0)=G_0(0)$.
We give complementary information about the meaning of $E(0)>0$ in Remark~\ref{rem:E0}.

\item
The finite-size interacting particle systems inspiring cMF and dPMF naturally satisfy a conservation principle about interactions.
This principle can be stated as follows: if a spiking avalanche triggers at time $t$ with $n$ simultaneously spiking particles, 
then every particle that has survived the avalanche without spiking must receive impulses from all the spiking particles.
This amounts to receiving jump updates of size $-\lambda n/N$.
In other words, there is no loss of interactions. 
The buffer mechanism is intended to enforce this conservation principle in the regularized system.
\item
One can  imagine interacting particle systems whose dynamics will converge to the proposed buffered dPMF dynamics under hypothesis of propagation of chaos.
This just requires modeling a reservoir state for interactions, whereby interactions may be thought of as particles as well.
\item
The alternative notion of blowup solutions  considered by Dou \emph{et al.} in~\cite{dou2022dilating} can also be obtained
as some limit of certain regularized dynamics.
Specifically, one can check that these solutions can be recovered when choosing the blowup indicator function $B(t)  = \mathbbm{1}_{\{ f(t) > 1/\delta \}}$ and taking $\delta \downarrow 0$.
However, this notion of blowup solutions  neglects the buffer mechanism enacted by the excess function $E$.
As a result, these solutions do not conserve the interaction rate and may exhibit eternal blowups.
\item
Recently, Dou \emph{et al.} also considered the idea of defining blowup solutions by examining regularized version of dPMF dynamics~\cite{dou2024noisy}.
The regularization process considered there is however different for replacing the hard-threshold condition of integrate-and-fire neurons with a soft-threshold condition,
which allowed for the consideration of solutions defined past blowups.
The proposed solutions are also interaction conserving but emerge as limit of solutions for dynamics that do not satisfy an absorbing boundary condition.
As a result, these solutions involves the introduction of a boundary measure, which we will not need here.
\end{enumerate}
\end{remark}

The regularization of the buffering mechanism follows from the fact that \eqref{eq:buffer} and \eqref{eq:indicator} imply that the interaction rate $\tf $ is upper bounded by $1/\delta$ so that the drift and diffusion coefficients are both uniformly upper bounded.
Note, however, that the firing rate $f$ can transiently exceed $1/\delta$, as expected during blowups.
Actually, one can see that by conservation of probability, we have:

\begin{proposition}\label{prop:fpx} 
Solutions to \eqref{all:bufferPDE} are such that
\begin{eqnarray}\label{eq:fpx}
f(t)= (1-B(t))\frac{\nu_2 z(t)}{1-\lambda_2 z(t)} + B(t) \left( \nu_2 + \frac{\lambda_2}{\delta} \right)z(t) \quad \mathrm{with} \quad z(t)=\frac{\partial_x p(t,0)}{2} \, .
\end{eqnarray}
\end{proposition}

\begin{proof}
Assume that a solution to \eqref{all:bufferPDE} exists.
Then, as before, the conservation of probability in the delayed dynamics imposes on one hand that for all $t \geq 0$
\begin{eqnarray*}
\partial_t \left( \int_{0}^\infty \! p(t,x) \, \mdd x \right) = f_\epsilon(t) - f(t) \, ,
\end{eqnarray*}
and direct integration of \eqref{eq:bufferPDE} with absorbing boundary condition imposes on the other hand that
\begin{eqnarray}
\partial_t \left( \int_{0}^\infty \! p(t,x) \, \mdd x \right) = -\big(\nu_2 \! + \! \lambda_2 \tf (t)\big) \partial_x p(t,0)/2 + f_\epsilon(t) \, .  \nonumber 
\end{eqnarray}
From the definition of the buffered rate of interaction in \eqref{eq:buffer}, if $B(t)=0$, we have $\tf (t)=f$ so that \eqref{eq:fluxBound} holds as for the nonbuffered dPMF dynamics.
By contrast, if $B(t)=1$, we have $\tf (t)=1/\delta$, so that $f(t)= (\nu_2  +  \lambda_2/\delta) \partial_x p(t,0)/2$.
This shows that \eqref{eq:fpx} holds.
\end{proof}

In Proposition \eqref{prop:fpx}, we denote the function $t \mapsto \partial_x p(t,0)/2$ as $z$ because this function will play a key role in understanding buffered blowups.
This key role follows from the following corollary to Proposition~\ref{prop:fpx}:

\begin{corollary}\label{corr:fpx} Solutions to \eqref{all:bufferPDE} are such that
\begin{eqnarray*}
\left( f(t) > \frac{1}{\delta}  \Leftrightarrow z(t) >z_\delta=\frac{1}{\lambda_2+\nu_2 \delta} \right) \quad \mathrm{and} \quad f(t) \leq  \left( \nu_2 + \frac{\lambda_2}{\delta}\right) z(t) \, .
\end{eqnarray*}
\end{corollary}

\begin{proof}
Note that both
\begin{eqnarray*}
z \mapsto  \frac{\nu_2 z}{1-\lambda_2 z}   \quad \mathrm{and} \quad z \mapsto  \left( \nu_2 + \frac{\lambda_2}{\nu_2}\right) z \, ,
\end{eqnarray*}
are increasing functions that take the identical value $1/\delta$ when $z=z_\delta$.
Therefore, independent of the value of $B$, we have $f(t) > 1/\delta$ if and only if $z(t) > z_\delta$.
The equality $f(t) =  ( \nu_2 + \lambda_2/\delta) z(t)$ holds when $B(t)=1$.
One can check that for all $0 \leq z <z_\delta$, we have
\begin{equation*}
 \frac{\nu_2 z}{1-\lambda_2 z} <  \left( \nu_2 + \frac{\lambda_2}{\delta} \right) \, , 
\end{equation*}
which proves the inequality $f(t) \leq   ( \nu_2 + \lambda_2/\delta) z(t)$ when $B(t)=0$.
\end{proof}

Corollary~\ref{corr:fpx} indicates that buffered blowups occur if $z$ exceeds the threshold 
value $z_\delta=1/(\lambda_2+\nu_2 \delta)$, similar to standard dPMF dynamics.
However, by contrast with dPMF dynamics, Corollary~\ref{corr:fpx} also indicates that the buffering 
mechanism precludes $f$ to explode when $z$ exceeds that threshold, at least as long as $z$ remains finite.
In that respect, we will later see that $z$ is a continuous function over all $\mathbbm{R}^+$.
Assuming that $z$ is continuous allows us to consider  the buffer mechanism on its own merit  in the next section, 
independent of the  PDE setting associated to dPMF problem.


\subsection{The buffer mechanism as a multivalued mapping}\label{sec:buffMech}

To characterize the properties of buffered blowups, it is convenient to consider the buffer mechanism as
a multivalued function between a time-dependent input $z(t)$ and an output  $\theta(t)$.
In the case of buffered dPMF dynamics, the input value is given by $z(t)=\partial_x p(t,0)/2$,
whereas the output value is the firing rate, i.e., $\theta(t)=f(t)$.
Informally, the occurrence of blowups in nonbuffered dPMF dynamics follows from the fact that the 
 function 
\begin{eqnarray*}
z \mapsto K(z) =   \frac{ \nu_2 z}{1 - \lambda_2 z} 
\end{eqnarray*}
which maps $z(t)$ onto $\theta(t)=K(z(t))$ by \eqref{eq:fluxBound}
explodes at a finite value of $1/\lambda_2$.
Mathematically, the idea of the buffer mechanism is to switch from the explosive functional dependence $\theta=K \circ z$
to a constitutively stable one, e.g.,  $\theta=L \circ z$ with
\begin{eqnarray*}
z \mapsto L(z) = \left( \nu_2 + \frac{\lambda_2}{\delta}\right) z \, ,
\end{eqnarray*}
whenever $z$ transiently exceeds a threshold value $z_\delta$ (see Fig.~\ref{fig:MathMech}a).
To ensure conservation of interaction rates in the buffered dynamics, one must also allow that $\theta=K \circ z$ can hold
below the threshold $z_\delta$, as long as there is excess interaction in the fictitious reservoir, i.e., $E>0$.
As a result, the buffer mechanism can be viewed as a bivalued function when $z<z_\delta$, with two branches
given by $K$ if $E=0$ or $L$ if $E>0$ (see Fig.~\ref{fig:MathMech}a).
The mathematical formulation of the ideas stated above are given in the following definition of the buffer mechanism:

\begin{figure}[htbp]
\begin{center}
\includegraphics[width=\textwidth]{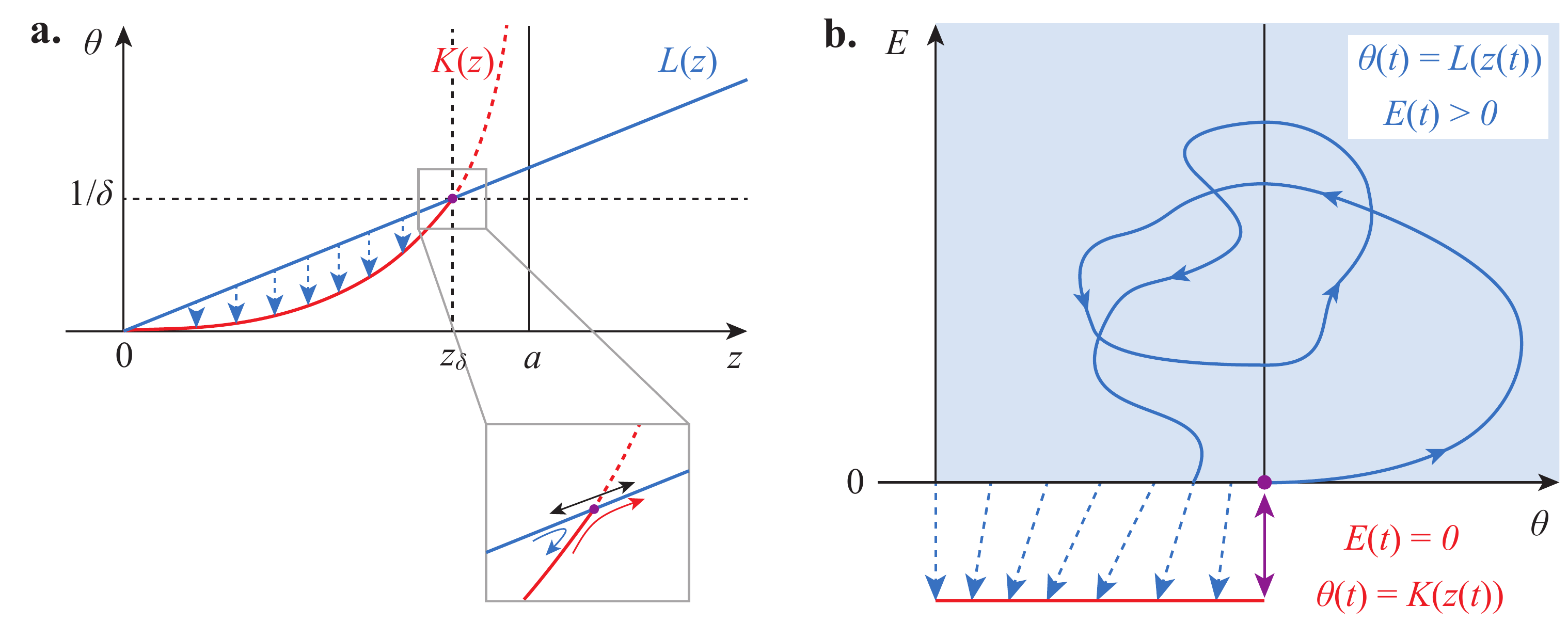}
\caption{{\bf Mathematical formulation of the buffer mechanism.}
The buffer mechanism switches from the explosive functional dependence $\theta=K \circ z$ (red curve in {\bf a.}) to a constitutively stable one $\theta=L \circ z$ (blue curve in {\bf a.}), whenever $z$ transiently exceeds a threshold value $z_\delta$ (red arrow in the inset figure in {\bf a}).
The variable $E$ keeps track of the excess interaction that needs to be buffered to maintain conservation of the interaction rate when switching from $K$ to $L$.
As a result of introducing this excess variable, buffered dynamics unfolds in the $(\theta,E)$-plane (blue curves in {\bf b.}) until it reaches an exit time $t$ such that $\theta(t) \leq 0$ and $E(t^-)=0$ and reverts to regular dynamics (red line in {\bf{b}}).
The solution $\theta$ generally exhibits a jump discontinuity at those times $t$ when $\theta(t)<0$ (blue dashed arrows n {\bf a.} and {\bf b.}).
Transitions back to the regular dynamics without jump discontinuity are possible but correspond to marginal cases (blue arrow in the inset figure in {\bf a}).
}
\label{fig:MathMech}
\end{center}
\end{figure}

\begin{definition}\label{def:bufferMech} 
\begin{subequations}
\label{all:bufferMech}
Given $\delta,a>0$, consider two continuous increasing functions $K: [0,a) \to \mathbbm{R^+}$ and 
$L: \mathbbm{R^+} \to \mathbbm{R^+}$  with $\lim_{z\to a^-} K(z)=\infty$ and
such that:
\begin{eqnarray*}
0<K(z) < L(z) \quad \Leftrightarrow \quad 0<z<z_\delta =L^{-1}(1/\delta)=K^{-1}(1/\delta) \, .
\end{eqnarray*}
Given a function $z: \mathbbm{R}^+ \to  \mathbbm{R}^+$, solving the $\delta$-buffer problem associated with $K$ and $L$ consists in finding 
a function $\theta: \mathbbm{R}^+ \to  \mathbbm{R}^+$ that solves the system of equations: 
\begin{eqnarray}
&\theta(t) = B(t) L(z(t))+ (1-B(t)) K(z(t))  \quad \mathrm{with} \quad B(t) = \mathbbm{1}_{\{\theta(t)>1/\delta\} \cup \{E(t)>0\}} \, , &\label{eq:buff1} \\
&\partial_t E(t) =   B(t)(\theta(t)-1/\delta)  =    B(t) (L(z(t))-1/\delta) \quad \mathrm{with} \quad  E(0) \geq 0   \, . &  \label{eq:buff3}
\end{eqnarray}
\end{subequations}
\end{definition}


The following lemma, which is nothing but a reformulation of Property~\ref{prop:fpx}, implies that the blowup indicator 
function featured in \eqref{all:bufferMech} can be written as
\begin{eqnarray}
B(t) = \mathbbm{1}_{\{\theta(t)>1/\delta\} \cup \{E(t)>0\}} = \mathbbm{1}_{\{z(t)>z_\delta\} \cup \{E(t)>0\}} \, ,
\end{eqnarray}
showing that $B$ only depends implicitly on $\theta$ via $E$.

\begin{lemma}\label{lem:buffcond}
Independent of the function $E$, any function $\theta$ satisfying \eqref{eq:buff1}  is such that
\begin{eqnarray}
&\big(\theta(t) > 1/\delta \quad \Leftrightarrow \quad z(t) >  z_\delta = K^{-1}(1/\delta)=L^{-1}(1/\delta) \big)  \, ,  & \nonumber\\
&\forall t \geq 0 \, , \quad \theta(t) \leq L(z(t)) \, .&  \nonumber
\end{eqnarray}
\end{lemma}

\begin{proof}
Independent of the function $E$, for any function $\theta$ defined by \eqref{eq:buff1}, we have then have the following:
If $z \leq z_\delta =K^{-1}(1/\delta)=L^{-1}(1/\delta)$, we have
\begin{eqnarray*}
\theta(t) = B(t) L(z(t))+ (1-B(t)) K(z(t)) \leq L(z(t)) \leq L(L^{-1}(1/\delta)) =\frac{1}{\delta} \, .
\end{eqnarray*}
If $z(t) >  z_\delta =K^{-1}(1/\delta)=L^{-1}(1/\delta)$, then $B(t)=1$ and we have
\begin{eqnarray*}
\theta(t) = B(t) L(z(t))+ (1-B(t)) K(z(t)) = L(z(t)) > L(L^{-1}(1/\delta)) =\frac{1}{\delta} \, .
\end{eqnarray*}
This shows that any function $\theta$ defined by \eqref{eq:buff1}  satisfies Lemma~\ref{lem:buffcond}.
\end{proof}

Observe that given Definition~\ref{def:bufferMech}, one can think of the buffered mechanism
as alternating two forms of mapping in the  $(\theta,E)$-plane (see Fig.~\ref{fig:MathMech}b). 
Outside blowups, i.e., when $B=0$, $(\theta,E)=(K\circ z, 0)$ is confined to the half line $\{ \theta<1/\delta, E=0\}$, 
which may only be left at the point $(1/\delta,0)$, when a blowup triggers. 
During blowups, i.e, when $B=1$, $(\theta,E)=(L \circ z,E)$ evolves in the open two-dimensional half-plane $E>0$, 
until it approaches the boundary $\{\theta \leq 1/\delta, E=0\}$ at some point $(L(z_b),0)$.
After experiencing a jump discontinuity $K(z_b)-L(z_b)<0$,  the blowup episode ends so that $B=0$ and 
$(\theta,E)=(K\circ z, 0)$ on the half line $\{\theta<1/\delta, E=0\}$, starting from the point $(K(z_b),0)$.

The picture presented in Fig.~\ref{fig:MathMech} suggests that there is a unique solution $\theta$ to  
$\delta$-buffer problems \eqref{all:bufferMech} with continuous input function $z$, at least if $z$ 
has a finite number of crossings with $z_\delta$ over all finite intervals. In Section~\ref{sec:buffproofs}, we show via
monotone-sequence arguments that $\delta$-buffer problems have a unique solution
for all continuous input functions, irrespective the finiteness of $z_\delta$-crossings.

\begin{proposition}\label{prop:buffExist}
Given a continuous function $z: \mathbbm{R}^+ \to  \mathbbm{R}^+$, the $\delta$-buffered problem \eqref{all:bufferMech} admits a unique solution $\theta: \mathbbm{R}^+ \to \mathbbm{R}^+$.
\end{proposition}

For illustration, we plot in Fig.~\ref{fig:Buffer} the numerical solution 
to a $\delta$-buffered problem \eqref{all:bufferMech} for an example of input function $z$.
The picture presented in Fig.~\ref{fig:MathMech}, which is supported by  Fig.~\ref{fig:Buffer},
suggests that $\delta$-buffered solutions have at most a countable set of discontinuities.
We rigorously establish this fact in Section~\ref{sec:buffproofs} by proving the following result:

\begin{proposition}\label{prop:buffReg}
Suppose that $\theta$ solves  the $\delta$-buffer problem \eqref{all:bufferMech} for a continuous function $z:  \mathbbm{R}^+ \to  \mathbbm{R}^+$,
then $\theta$ is a \cadlag function with an at most countable set of negative jump discontinuities such that
\begin{eqnarray}
1/\delta > \lim_{s \to t^-} \theta(s) > \theta(t)=\lim_{t \to t^+} \theta(t) \, .
\end{eqnarray}
\end{proposition}

\begin{figure}[htbp]
\begin{center}
\includegraphics[width=\textwidth]{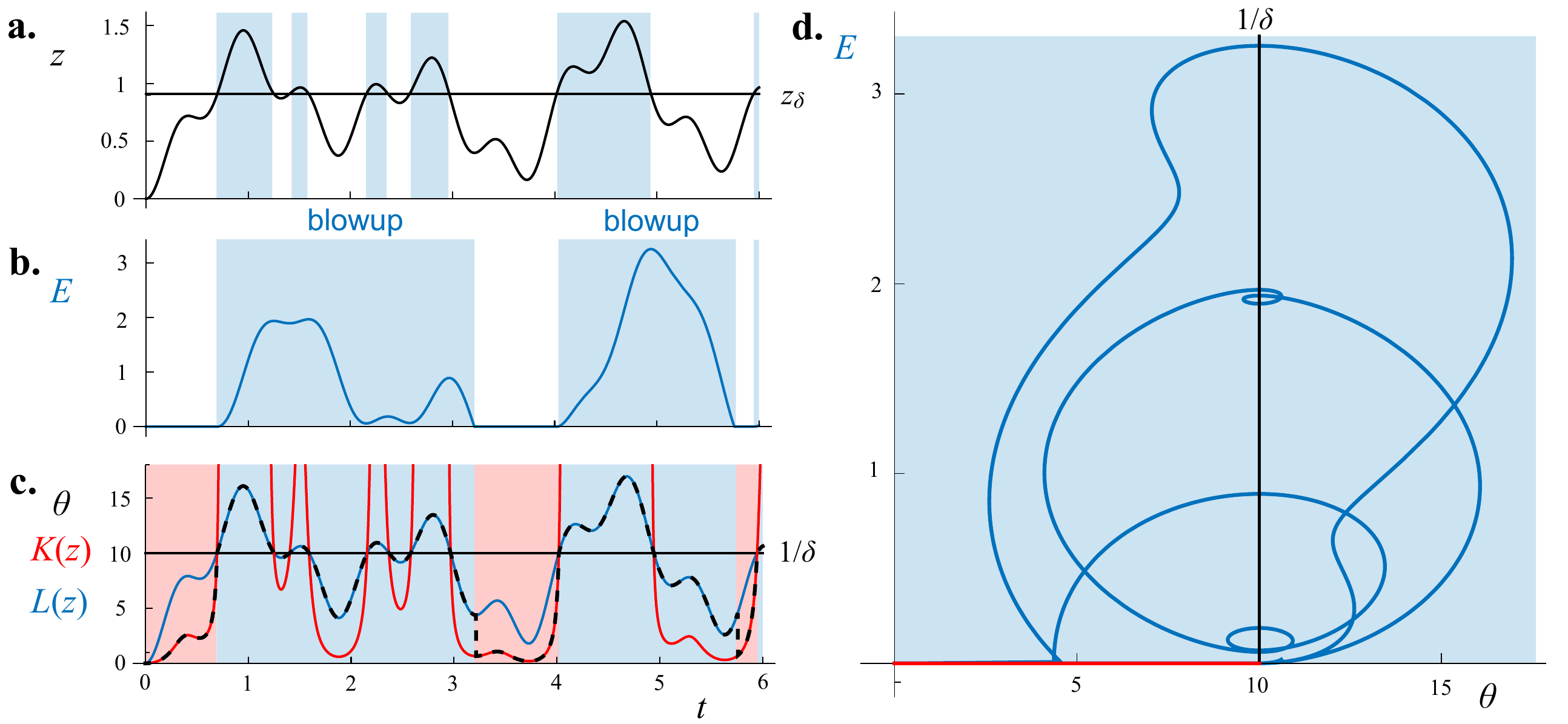}
\caption{{\bf Example of a $\delta$-buffered solution.}
Given a continuous input function $t \mapsto z(t)$, there is a unique \cadlag solution the $\delta$-buffered problem for functions $K$ and $L$ satisfying the conditions of~\ref{def:bufferMech}.
Blowups are triggered in region of times where $z>z_\delta$ (blue regions in {\bf a.}).
Triggered blowup extend past these time regions as long as the excess function remains positive, i.e., $E>0$ (blue regions in {\bf b.}).
Solutions $\theta$ to the $\delta$-buffered problem (dashed black line in {\bf c.}) are obtained via a switching mechanism between $L$ and $K$ that is controlled by the blowup
indicator function $B$.
These solutions exhibit a jump discontinuities at the blowup exit time when they exit the region $E>0$ (see {\bf d.}).
  }
\label{fig:Buffer}
\end{center}
\end{figure}


\subsection{Characterization in terms of running extrema}\label{sec:runExtr}

We conclude our analysis of the buffer mechanism by giving a useful characterization 
of its action in terms of $\delta$-tilted, running extrema.
Such a characterization will be required to formulate the fixed-point problem satisfied by buffered dPMF dynamics
in the time-changed picture.
Before stating this characterization, observe that given a continuous input function $z$,
Proposition~\ref{prop:buffReg} directly implies that $E$ is continuous and $B$ is lower semicontinuous.
This observation allows one to define  the open set of full-blowup times for 
$\delta$-buffered dPMF dynamics as $\mathcal{B}_t=\{ t>0 \, \vert \, E(t) >0 \}$.
Following the same argument as in Section~\ref{sec:fixed-point}, there is a countable index set $\mathcal{K}$ such that 
\begin{eqnarray*}
\mathcal{B}_t=\{ t>0 \, \vert \, E(t) >0 \} = \bigcup_{k \in \mathcal{K}} (T_k,V_k) \, ,
\end{eqnarray*}
where the full-blowup intervals $(T_k,V_k)$  are nonoverlapping.
As  in Section~\ref{sec:fixed-point}, we refer to the countable set  $T_k$, $k \in \mathcal{K}$, as the full-blowup trigger times and to $V_k$, $k \in \mathcal{K}$, as the full-blowup exit times.
We also introduce the following useful time definitions:

\begin{definition}\label{def:blowupTimes}
For all $t \geq 0$, we define the last-blowup-trigger time $T_0$ and the next-blowup-exit time $V_1$ as
\begin{eqnarray*}
T_0(t) = \left[ \sup \{ s \leq t \, \vert \, E(s)=0 \} \right]_+ \quad \mathrm{and} \quad V_1(t) = \inf \{ s \geq t \, \vert \, E(s)=0 \}  \, .
\end{eqnarray*}
where for all $x \in \mathbbm{R}$, the positive part of $x$ is denoted as $\left[ x \right]_+= x \vee 0$.
\end{definition}

The times defined above are useful to represent the excess function as an integral of the interaction rate alone.

\begin{proposition}\label{prop:E}
For all $t \geq 0$, we have 
\begin{eqnarray}\label{eq:E}
E(t) = \int_{T_0(t)}^t (\theta(s)-1/\delta) \, \mdd s + \mathbbm{1}_{\{T_0(t)=0 \}} E(0)\, ,
\end{eqnarray}
so that
\begin{eqnarray*}
V_1(t) =  \inf \left\{ v > t  \, \bigg \vert \, \int_{T_0(t)}^v (\theta(s) - 1/\delta)\, \mdd s  \leq 0 \right\} \, .
\end{eqnarray*}
\end{proposition}

\begin{proof}
If $t \notin \mathcal{B}_t$, then $T_0(t)=t$, so that as expected, \eqref{eq:E} yields $E(t) = 0$.
If $t \in \mathcal{B}_t$, we have $B(t)=1$, and
\begin{eqnarray*}
E(t) = E(t)-E(T_0(t)) + E(T_0(t)) = \int_{T_0(t)}^t (\theta(s)-1/\delta) \, \mdd s + E(T_0(t))   \, .
\end{eqnarray*}
If $T_0(t)>0$, then it must be that $E(T_0(t))=0$ by continuity of $E$,
so that \eqref{eq:E} generally holds.
Finally, the characterization of $V_1(t)$ in terms of $T_0(t)$ 
directly follows from its definition as $V_1(t) = \inf \{ s \geq t \, \vert \, E(s)=0 \}$.
\end{proof}

\begin{remark}\label{rem:E0}
We will consider the buffer mechanism when it acts on dPMF dynamics for initial conditions considered at a running time $t>0$.
In that context, it is instructive to specify the initial value of the $t$-shifted excess
function $E_t=E(\cdot - t)$  consistently with the $t$-shifted initial conditions $(p_t,f_t)$, where $p_t=p(t,\cdot)$ and $f_t=\{f(s-t)\}_{s \leq t}$.
Since we always assume that $(p_0,g_0)$ are nonexplosive initial conditions and since the role of $\theta$ is played by $f$ in $\delta$-buffered dPMF dynamics, by Proposition~\ref{prop:E} we have
$$T_0(t)>0  \quad \Rightarrow \quad E_t(0)=E(t)=\int_{T_0(t)}^t (f(s)-1/\delta) \, \mdd s \, .$$
Therefore, if $t$ is not a full-blowup time, i.e., $t \notin \mathcal{B}_t$, $T_0(t)=t$ so that $E_t(0)=0$, but  if $t$ is  a full-blowup time, i.e. $t \in \mathcal{B}_t$,  $0<T_0(t)<t$ so that $E_t(0)> 0$.
\end{remark}

We are now in a position to give a characterization of the buffer mechanism 
in terms of the $\delta$-titled running infimum of the cumulative function $t \ni \mathbbm{R}^+ \mapsto \Theta(t)=\int_0^t \theta(s) \, \mdd s$, which is
directly related to the last-blowup-trigger time $T_0$.

\begin{proposition}\label{prop:PhiTheta}
Suppose that $\theta$ solves  the $\delta$-buffer problem \eqref{all:bufferMech} for a continuous function $z:  \mathbbm{R}^+ \to  \mathbbm{R}^+$ and an initial excess value $E(0)\geq 0$.
Given the cumulative function $\mathbbm{R}^+  \ni t \mapsto \Theta(t)=\int_0^t \theta(s) \, \mdd s + \Theta(0)$ with $\Theta(0)=E(0)$,
consider the function $\tT = \Theta-E: \mathbbm{R}^+ \to  \mathbbm{R}$, where $E$ is the excess function associated to $\theta$.
Then,  $ \tT: \mathbbm{R}^+ \to  \mathbbm{R}^+$ is such that
\begin{eqnarray}
\tT(t) 
= \int_0^t \widetilde{\theta}(s) \, \mdd s 
=  \frac{t}{\delta} +  \left[ \inf_{0 \leq s \leq t}  \left( \Theta(s) - \frac{s}{\delta} \right)  \right]_-
=  \Theta(T_0(t)) + \frac{t-T_0(t)}{\delta} \, ,\label{eq:Phi} 
\end{eqnarray}
where $\widetilde{\theta}$ denotes the $\delta$-buffered version of $\delta$
and where for all $x \in \mathbbm{R}$, we denote the negative part of $x$ as $\left[ x \right]_-= x \wedge 0$.
\end{proposition}

\begin{proof}


To show the first equality, observe that by definition of the excess function~\eqref{eq:buff3}, we have
\begin{eqnarray*}
\tT (t) 
&=& 
\Theta(t)-E(t) \, , \\
&=& 
\int_0^t \theta(s) \, \mdd s +E(0) - \int_0^t B(s) (\theta(s) -1/\delta ) \, \mdd s  - E(0)\, , \\
&=& 
\int_0^t  \big((1-B(s))\theta(s) \, \mdd s  +  B(s)/\delta \big) \, \mdd s \, , 
\end{eqnarray*}
which shows the first equality in \eqref{eq:Phi} upon recognizing that $ \widetilde{\theta} = \big( (1-B)\theta +  B/\delta\big)$. 

To show the second equality in  \eqref{eq:Phi}, observe then that by construction $\widetilde{\theta} \leq 1/\delta$, so that we have
\begin{eqnarray*}
v \geq t  \geq 0 \quad \Rightarrow \quad \tT (v)-\tT (t) 
= 
\int_t^v \widetilde{\theta}(s) \, \mdd s  
\leq (v-t)/\delta \, .
\end{eqnarray*}
Therefore $t \mapsto \tT (t) - t/\delta$ is a nonincreasing function $\mathbbm{R}^+ \to \mathbbm{R}$
and for all $t>0$, we have
\begin{eqnarray*}
\tT(t) - t/\delta 
=
 \inf_{0 \leq s \leq t} ( \tT (s) - s/\delta) 
 \leq 
 \inf_{0 \leq s \leq t}  ( \Theta(s) - s/\delta) \, , 
\end{eqnarray*}
where the last inequality follows from the fact that $\tT = \Theta- E\leq \Theta$.
We must distinguish two cases:

$(i)$ Consider the case for which $0$ is not a blowup time, i.e., when  $E(0)=0$.
For this case, one can disregard the negative part operator in \eqref{eq:Phi} as
\begin{eqnarray*}
\inf_{0 \leq s \leq t}  \left( \Theta(s) - \frac{s}{\delta} \right)   \leq \Theta(0)=E(0)=0 \, .
\end{eqnarray*}
Thus the second inequality in \eqref{eq:Phi} will follow from establishing that 
\begin{eqnarray*}
\tT(t) - t/\delta 
 \geq 
 \inf_{0 \leq s \leq t}  ( \Theta(s) - s/\delta) \, .
\end{eqnarray*}
If $\tT(t) = \Theta(t)$, there is nothing to show.
Suppose then that $\tT(t) < \Theta(t)$, which means that  $E(t)>0$, so that $t$ is in a full-blowup interval.
Consider the last-blowup-trigger time $T_0(t)$, which satisfies $0 \leq T_0(t) <t$ with $E(T_0(t))=0$.
By Proposition~\ref{prop:E}, we then have
\begin{align}\label{eq:ThetaT}
\Theta(t) - \Theta(T_0(t))  
&= \int_{T_0(t)}^t (\theta(s) \, \mdd s \, , \nonumber \\
&= \int_{T_0(t)}^t (\theta(s)-1/\delta) \, \mdd s + (t-T_0(t))/\delta \, , \nonumber \\
&= E(t) + (t-T_0(t))/\delta \, .
\end{align}
This implies that
\begin{eqnarray*}
\tT(t) - t/\delta = \Theta(t)  - E(t) - t/\delta = \Theta(T_0(t))  -T_0(t)/\delta \geq   \inf_{0 \leq s \leq t}  ( \Theta(s) - s/\delta)  \, .
\end{eqnarray*}
which is the desired inequality.

$(ii)$ Next, consider the case for which $0$ is a blowup time, i.e., when  $E(0)>0$, and pick $t \geq 0$.
If $T_0(t)>0$, the same reasoning as in $(i)$ applies to show that the second equality  in \eqref{eq:Phi} holds.
Otherwise $T_0(t)=0$ and it must be that for all $s$, $0 \leq s \leq t$, $E(s)=\Theta(s)-\tT(s)>0$,
which implies that
\begin{eqnarray*}
\inf_{0 \leq s \leq t}  ( \Theta(s) - s/\delta) > \inf_{0 \leq s \leq t}  ( \tT(s) - s/\delta) = \tT(t) - t/\delta \, ,
\end{eqnarray*}
 where the last equality follows from the nonincreasing character of $t \mapsto \tT (t) - t/\delta$.
 But by Proposition~\ref{prop:E}, we  have
 \begin{align}\label{eq:ThetaT2}
\tT(t)  
&= \Theta(t) - E(t) \, , \nonumber \\
&= \int_0^t \theta(s) \, \mdd s + E(0) - \left( \int_0^t (\theta(s)-1/\delta) \, \mdd s + E(0) \right) \, , \nonumber \\
&= t/\delta\, .
\end{align}
This shows that 
\begin{equation*}
\inf_{0 \leq s \leq t}  ( \Theta(s) - s/\delta) >\tT(t) - t/\delta = 0 \, .
\end{equation*}
Therefore, when $E(0)>0$ and $T_0(t)=0$, we have
\begin{equation*}
\tT (t) = t/\delta = t/\delta + \frac{t}{\delta} +  \left[ \inf_{0 \leq s \leq t}  \left( \Theta(s) - \frac{s}{\delta} \right)  \right]_- \, .
\end{equation*}
showing that the second equality  in \eqref{eq:Phi} always holds.

The last equality in \eqref{eq:Phi} is a direct consequence of the definition of $T_0(t)$ as the last-blowup-trigger time since
\begin{eqnarray*}
T_0(t) = \left[ \sup \{ s \leq t \, \vert \, E(s)=0 \} \right]_+ = \left[ \sup \{ s \leq t \, \vert \, \tT(s)=\Theta(s) \}  \right]_+ \, .
\end{eqnarray*}
This concludes the proof.
\end{proof}

The above characterization of the $\delta$-buffer mechanism in terms of $\delta$-tilted running infimum will provide us
with the starting point of our arguments justify the physicality of our fixed-point equation \eqref{eq:physFP}.
Such a justification will require the time-changed version of Proposition~\ref{prop:PhiTheta}
under the assumptions that the initial conditions are nonexplosive, i.e., for $E(0)=0$ and that $\Theta$ is an increasing functon.
We give this time-changed version as a corollary bearing on the inverses
of the functions $\Theta$ and $\tT$, which we denote $\Gamma = \Theta^{-1}$ and  $\tGamma=\tT^{-1}$, respectively.
The time change considered  is $\sigma=\tT(t) \Leftrightarrow t=\tGamma(\sigma)$,
which allows us to introduce the time-changed, last-blowup-trigger time as $S_0(\sigma)=\tT(T_0(t))=\tT(T_0(\tGamma(\sigma))) \geq 0$.

\begin{corollary}\label{cor:PsiGamma}
If $\Theta(0)=E(0)=0$ and $\Theta$ is an increasing function, the inverse functions $\Gamma = \Theta^{-1}:\mathbbm{R}^+ \to  \mathbbm{R}^+$ and $\tGamma=\tT^{-1}: \mathbbm{R}^+ \to  \mathbbm{R}^+$ are  such that
\begin{eqnarray}
\tGamma(\sigma) 
=
\delta \sigma + \left[ \sup_{0 \leq \tau \leq \sigma} ( \Gamma(\tau) - \delta \tau ) \right]_+
=
\Gamma(S_0(\sigma)) + \delta (\sigma-S_0(\sigma))  \, .
\end{eqnarray}
\end{corollary}

The proof is given in Section~\ref{sec:RSfunction}. 


\section{Validation of candidate physical solutions as limit buffered solutions}\label{sec:limBuff}

In this section, we confirm that the blowup dPMF dynamics specified by Theorem~\ref{th:mainRes} represent physical solutions.
In Section~\ref{sec:buffTC}, we leverage the characterization of the buffer mechanism in terms of running extrema 
to formulate the fixed-point problem satisfied by  $\delta$-buffered dPMF dynamics.
In Section~\ref{sec:buffSol}, we show that, independent of the occurrence of buffered blowups, 
this fixed-point problem uniquely specifies a $\delta$-buffered dPMF dynamics, 
which is always physical.
ln Section~\ref{sec:recov}, we prove Theorem~\ref{th:mainRes2} showing that our candidate dPMF dynamics are physical 
as these correspond to $\delta$-buffered dPMF dynamics in the limit $\delta \downarrow 0$.


\subsection{Buffered fixed-point equation in the time-changed picture}\label{sec:buffTC}
Just like nonregularized dynamics, buffered dPMF dynamics are amenable to analysis 
in a time-changed approach.
In fact, following the same reasoning as in Section~\ref{sec:TC}, the form of 
the buffered problem \eqref{all:bufferPDE} suggests to adopt the time change function
\begin{eqnarray}\label{eq:deltaPhiDef}
\Phi(t) = \nu_2 t + \lambda_2 \tF (t) = \nu_2 t + \lambda_2 \int_0^t \tf (s) \, \mdd s \, .
\end{eqnarray}
where the rate of interaction $\tf$ is the $\delta$-buffered version of the firing rate $f$.
This motivates defining the space of admissible $\delta$-buffered time-change functions as follows:

\begin{definition}\label{def:deltaPhiBuff}
The set of admissible $\delta$-buffered time changes $\mathcal{T}_\delta$ is the set of functions $\Phi:[0, \infty)\to[0, \infty)$, with $\Phi(0)=0$,
such that for all $y,x \leq -\epsilon$, $x \neq y$, the difference quotients $w_\Phi(y,x)$ satisfy
\begin{eqnarray*}
\nu_2 \leq w_\Phi(y,x)=\frac{\Phi(y)-\Phi(x)}{y-x} \leq \nu_2+\frac{\lambda_2}{\delta} = \frac{1}{\delta_2} \, .
\end{eqnarray*}
\end{definition}

As a result of buffer regularization, all time changes $\Phi$ in $\mathcal{T}_\delta: \mathbbm{R}^+ \to \mathbbm{R}^+$ 
are continuous, increasing, one-to-one functions,
and so are their inverse time changes $\Psi=\Phi^{-1}$:

\begin{definition}\label{def:deltaPhiBuff}
The set of admissible inverse time change $\mId$ is the set of functions $\Psi:[0, \infty)\to[0, \infty)$ with $\Psi(0)=0$, such that for all $y,x \leq \chi_0$, $x \neq y$, the difference quotients $w_\Psi(y,x)$ satisfy
\begin{eqnarray*}
\delta_2= \frac{\delta}{\lambda_2+\nu_2 \delta} \leq w_\Psi(y,x)=\frac{\Psi(y)-\Psi(x)}{y-x} \leq \frac{1}{\nu_2} \, .
\end{eqnarray*}
\end{definition}

A key feature of $\delta$-buffered dPMF dynamics is that the change of variable $\sigma = \Phi(t) \Leftrightarrow t = \Psi(\sigma)$ is always well-defined for these dynamics.
As a result, there is a complete equivalence between original-time, $\delta$-buffered dPMF dynamics solving \eqref{all:bufferPDE} in Definition~\ref{def:bufferPDE} 
and the corresponding time-changed dynamics.
Assuming $\Psi$ known, the latter time-changed dynamics are still specified as solutions to 
the time-changed problem  \eqref{all:TCPDE} in Definition~\ref{def:TCPDE} but for $\Psi \in \mId \subsetneq \mIz$.
In particular, Proposition~\ref{prop:solG} applies to $\delta$-buffered dynamics and
given any $\Psi \in \mId$, there is  a density function $(\sigma,x) \mapsto q(\sigma, x)$ solving 
the time-changed problem \eqref{all:TCPDE}. 
Therefore, finding a $\delta$-buffered solution consists in self-consistently specifying
an inverse time change $\Psi \in \mId$, so that the time change $\Phi=\Psi^{-1}$ satisfies \eqref{eq:deltaPhiDef}.

In the following, we give such a  specification by leveraging the characterization 
of the $\delta$-buffered solution in terms of $\delta$-tilted running extrema 
functions given by Proposition~\ref{prop:PhiTheta} in the original time picture 
and by Corollary ~\ref{cor:PsiGamma} in the time-changed picture.
Specifically, we show that functions in $\Psi \in \mId$ that represent inverse time changes 
for buffered dPMF dynamics are characterized as solution to the following fixed-point problem:

\begin{proposition}\label{prop:PsiToPhi}
Given generic initial conditions, if $(\sigma,x) \mapsto q(\sigma,x)$ solves the time-changed problem \eqref{all:TCPDE} for an inverse time change $\Psi \in \mId$,
then  $(t,x) \mapsto p(t,x)=q(\Phi(t),x)$, $\Phi=\Psi^{-1} \in  \mathcal{T}_\delta$, solves the $\delta$-buffered problem \eqref{all:bufferPDE} if and only if
\begin{eqnarray}\label{eq:fixedPoint}
\Psi(\sigma) = \delta_2 \sigma + \left[ \sup_{0 \leq \tau \leq \sigma} \left( \frac{\tau-\lambda_2 G[\Psi]\tau)}{\nu_2} - \delta_2 \tau \right) \right]_+ \, \quad \mathrm{with} \quad \delta_2 = \frac{\delta}{\lambda_2 + \nu_2\delta}  \, .
\end{eqnarray}
\end{proposition}

It is instructive to illustrate the content of the above proposition.
This is done in Fig.~\ref{fig:PsiPhi} which emphasizes that the fixed-point equation \eqref{eq:fixedPoint} 
is more than a mere time-changed formulation of the buffer mechanism.
In the original picture (Fig.~\ref{fig:PsiPhi}a), the time change $\Phi \in \mathcal{T}_\delta$ (black curve) is obtained 
as a $(1/\delta_2)$-tilded running infimum of the function $t \mapsto \nu_2 +\lambda_2 F(t)$ (blue curve).
By construction of the buffer mechanism, this function is always an increasing function
and presents slope discontinuities at the end of buffered-blowup episodes.
In the time-changed picture (Fig.~\ref{fig:PsiPhi}b), the time change $\Psi \in \mId$ (black curve) is obtained 
as a $\delta_2$-tilded running supremum of the function $\sigma \mapsto (\sigma-\lambda_2 G(\sigma))/\nu_2$ (green curve).
By contrast with the original picture, the latter function is $C_1$ but not necessarily monotonic during buffered blowups.
In particular, this function differs from the inverse of  $t \mapsto \nu_2 +\lambda_2 F(t)$ obtained by mere time change (blue dashed curve).

\begin{figure}[htbp]
\begin{center}
\includegraphics[width=\textwidth]{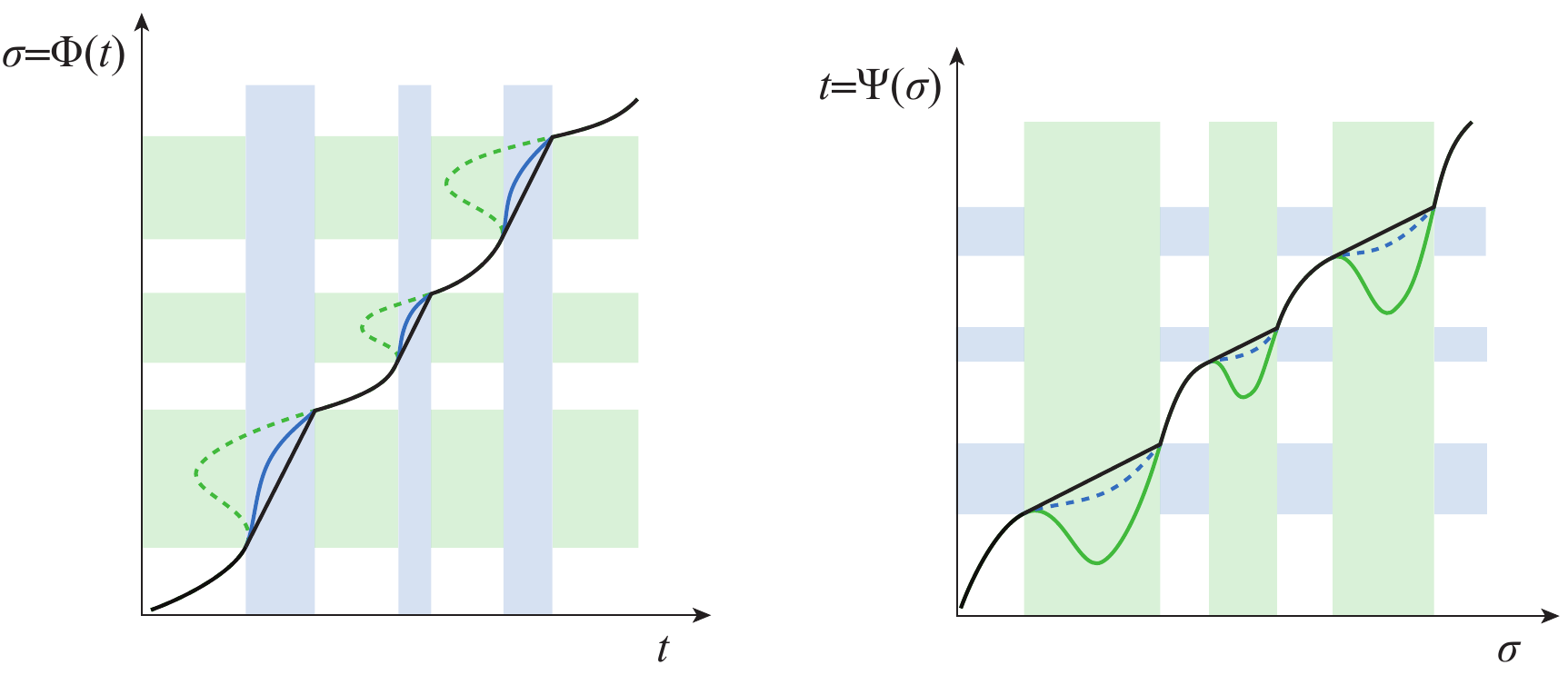}
\caption{{\bf The time-change picture.}
Left: The time change $t \mapsto \Phi(t)=\nu_2 t + \lambda_2 F(t)$ (black curve) reveals three blowup episodes, when $\Phi$ has maximum slope $1/\delta_2$ (blue regions).
During these blowups, the function $t \mapsto \nu_2 t + \lambda_2 F(t)$ (blue curves) transiently exceeds $\Phi$ until the excess function $E=(F-\tF)/\lambda_2$ becomes zero again.
The function $t \mapsto \nu_2 t + \lambda_2 F(t)$ is a monotonic $C_1$ interpolation of the regular dynamics (outside of blowups) except at the blowup exit time where it has a slope discontinuity.
Right: The matching inverse time-change function $t \mapsto \Psi(t)=\Phi^{-1}(t)$ (black curve)  exhibits three blowup episodes, when $\Psi$ has minimum slope $\delta_2$ (green regions).
These blowup episodes correspond to regions of time when $\Psi$ transiently exceeds the inverse function of $t \mapsto \nu_2 t + \lambda_2 F(t)$ (blue dashed curves) as well as the
function $\sigma \mapsto \big(\sigma-\lambda_2 G(\sigma)\big)/\nu_2$ (green curve). 
The latter function represents a nonmonotonic $C_1$ interpolation of the regular dynamics in the time-changed picture (outside of blowups) but does not specify a single-valued function in direct time $t$ (green dashed line).
}
\label{fig:PsiPhi}
\end{center}
\end{figure}

Before proving Proposition~\ref{prop:PsiToPhi}, observe that as $G$ is always an increasing $C_1$ function
and that all functions $\Psi$ satisfying \eqref{eq:fixedPoint} belong to $\mId$.
In fact, recognizing that $G$ is $C_1$ also implies that solutions $\Psi$ to \eqref{eq:fixedPoint} admit \cadlag generalized derivatives.
This fact, which will be extensively used in the proof of Proposition~\ref{prop:PsiToPhi}, is a direct consequence of the following lemma, 
whose proof is included in Section~\ref{sec:RSfunction}  for completeness.

\begin{lemma}\label{lem:Sup}
For all $\Gamma \in C_1(\mathbbm{R}^+)$, there is an at most countable set of points $\mathcal{D} \subset \mathbbm{R}^+$ such that the running supremum function
\begin{eqnarray}\label{eq:supGamma}
 \mathbbm{R}^+ \ni \sigma  \mapsto \tGamma(\sigma) =  \sup_{0 \leq \tau \leq \sigma} \Gamma(\tau)
\end{eqnarray}
is continuously differentiable on  $\mathbbm{R}^+ \setminus \mathcal{D}$.
Moreover,  for all $\sigma \in \mathcal{D}$, $\tGamma$ admits a left-derivative $\tGamma'_-(\sigma)$ and a right-derivative $\tGamma'_+(\sigma)$ with $0 = \tGamma'_-(\sigma) < \tGamma'_+(\sigma) = \Gamma'(\sigma)$.
\end{lemma}

We are now able to prove Proposition~\ref{prop:PsiToPhi}:

\begin{proof}[Proof of Proposition~\ref{prop:PsiToPhi}]
Suppose that $(\sigma,x) \mapsto q(\sigma,x)$ solves the time-changed problem \eqref{all:TCPDE} for 
some generic initial conditions $(q_0,g_0)$ in $\mathcal{M}(\mathbbm{R}^+) \times \mathcal{M}((-\infty,0))$ and some inverse time change function $\Psi \in \mId$.
Set $p(t,x)=q(\Phi(t),x)$ with $\Phi=\Psi^{-1}$ and observe that for all $\Psi \in \mId$, both $\Psi$ and $\Phi$ admits bounded Radon-Nikodym derivatives.
This observation allows one to reverse the change of variables from the proof of Proposition~\ref{prop:PhiToPsi} to get
\begin{eqnarray*}
\partial_t p 
=
 \left(   \nu_1 - \frac{\lambda_1}{\lambda_2} \nu_2     +  \frac{\lambda_1}{\lambda_2} \Phi'(t) \right) \partial_x p + \frac{\Phi'(t)}{ 2 } \, \partial^2_x p  + f_\epsilon(t) \delta_\Lambda \, ,
\end{eqnarray*}
where $f$ is the density of the cumulative functions $F=G \circ \Phi$.
Observe that generic initial conditions then specifies $F(0)=G(0)=G_0(0) \geq 0$.
We thus have
\begin{eqnarray*}
f(t) = \Phi'(t)g(\Phi(t)) =  \Phi'(t)\partial_x q(\Phi(t),0)/2 = \Phi'(t)\partial_x p(t,0)/2  \, .
\end{eqnarray*}
Term-by-term comparison with the  problem \eqref{all:bufferPDE} 
shows that $(t,x) \mapsto p(t,x)=q(\Phi(t),x)$ is a solution if and only if $\Phi'(t) = \nu_2 + \lambda_2 \tf(t)$,
where $\tf$ denotes the $\delta$-buffered rate of $f$ for
the functions
\begin{eqnarray}\label{eq:deltaBuffDef}
K(z) =   \frac{ \nu_2 z}{1 - \lambda_2 z} \, , \quad L(z) =  \frac{z}{\delta_2} \, ,  \quad \mathrm{and} \quad  z(t) = g(\Phi(t))  \, .
\end{eqnarray}
To establishing Proposition~\ref{prop:PsiToPhi}, we prove in the following that  given fixed generic initial conditions, the fixed-point condition \eqref{eq:fixedPoint} holds if and only if $f$ solves the $\delta$-buffer problem stated above with initial excess value $E(0)=G_0(0)$.

\emph{Necessary condition.}
 Suppose $f$ solves the $\delta$-buffer problem for the function $K$, $L$, $z$ given in \eqref{eq:deltaBuffDef} with $E(0)=G_0(0)$. 
 Then:

 \emph{Case of explosive initial condition: $E(0)=G_0(0)>0$.}
 If $E(0)=G_0(0)>0$, the next-blowup-exit time is such that
\begin{align*}
V_1(0)
=  \inf \left\{ t > 0  \, \big \vert \, E(t) \leq 0 \right\}  
=  \inf \left\{ t > 0  \, \bigg \vert \, \int_{0}^t (f(s) - 1/\delta)\, \mdd s + E(0) \leq 0 \right\} > 0 \, .
\end{align*}
Let us denote the corresponding time-changed blowup exit time by $U_1(0) = \Phi(V_1(0))>\nu_2 V_1(0)>0$.
For all $t$, $0 \leq t <V_1(0)$, since $ F(0)=G_0(0)=E(0)$, we have
\begin{align*}
 \int_{0}^t (f(s) - 1/\delta)\, \mdd s 
 = F(t) - F(0) - t/\delta + E(0)
  =F(t)-t/\delta >0 \, .
\end{align*}
As $F=G \circ \Phi \Leftrightarrow G = F \circ \Psi$ for all $\Psi \in \mId$, this implies that
for all $\sigma$, $0 \leq \sigma < U_1(0)$, we have
$G(\sigma) = F ( \Psi (\sigma)) > \Psi(\sigma)/
\delta$. Moreover, for all blowup times, we have
\begin{align*}
\Phi(t) = \nu_2 t + \lambda_2 t/\delta = t/\delta_2 \, , \quad  0 \leq t \leq V_1(0) 
\quad \Leftrightarrow \quad
\Psi(\sigma) = \delta_2 \sigma  \, , \quad  0 \leq \sigma \leq U_1(0) \, .
\end{align*}
Therefore, $G(\sigma) > (\delta_2/\delta) \sigma$ for all $\sigma$, $0 \leq \sigma \leq U_1(0)$, 
 so that
\begin{align*}
\frac{\sigma-\lambda_2 G(\sigma)}{\nu_2} \leq  \frac{\sigma-\lambda_2 (\delta_2/\delta) \sigma}{\nu_2} = \delta_2 \sigma \, .
\end{align*}
with equality if and only if $\sigma=U_1(0)$.
In other words,
\begin{align*}
U_1(0)
=  \inf \left\{ \sigma > 0  \, \big \vert \,(\sigma-\lambda_2 G(\sigma))/\nu_2 - \delta_2 \sigma \geq 0 \right\}  \, ,
\end{align*}
and for all $\sigma$, $0 \leq \sigma \leq U_1(0)$, we have
\begin{align*}
 \left[ \sup_{0 \leq \tau \leq \sigma} \left( \frac{\tau-\lambda_2 G[\Psi]\tau)}{\nu_2} - \delta_2 \tau \right) \right]_+ = 0 \, ,
\end{align*}
which shows that the fixed-point condition \eqref{eq:fixedPoint} holds up to $U_1(0)$.

There is nothing else to show if $U_1(0)=\infty$ (which can only happen if $V_1(0)=\infty$).
If $U_1(0)<\infty$, note that we necessarily have $D(U_1(0))=E(V_1(0))=0$.
As a result, showing that the fixed-point condition \eqref{eq:fixedPoint} holds 
for $\sigma > U_1(0)$, is equivalent to showing that 
\eqref{eq:fixedPoint} holds for nonexplosive initial conditions, which we establish next.

\emph{Case of nonexplosive initial condition: $G_0(0)=0$.}

 $(i)$ If $E(0)=G_0(0)=0$, by Proposition~\ref{prop:PhiTheta},  the cumulative function $F(t) =   \int_0^t f(s) \, \mdd s $ 
 and the excess function $E: \mathbbm{R}^+ \to \mathbbm{R}^+$ solving the  $\delta$-buffer mechanism are such that 
\begin{eqnarray*}
 F(t) - E(t) =  \tF(t) = \int_0^t \tf(s) \, \mdd s  \, ,
\end{eqnarray*}
which allows one to relate $\Phi$ to the cumulative function $F$ via
\begin{eqnarray*}
\Phi(t) = \nu_2 t + \lambda_2 \tF(t) = \nu_2 t + \lambda_2 (F(t) - E(t)) \, .
\end{eqnarray*}
Then, setting $\sigma=\Phi(t)$, $(t,x) \mapsto p(t,x)=q(\Phi(t),x)$ is a solution of the problem \eqref{all:bufferPDE} if
\begin{eqnarray*}
\sigma = \nu_2 \Psi(\sigma)+ \lambda_2 (F(\Psi(\sigma)) - E(\Psi(\sigma))) = \nu_2 \Psi(\sigma)+ \lambda_2 (G(\sigma) - D(\sigma))  \, .
\end{eqnarray*}
where we have used the time-changed excess function $D = E \circ \Psi$.
Consequently
\begin{eqnarray}\label{eq:PsiG}
\Psi(\sigma) = (\sigma - \lambda_2 (G(\sigma) - D(\sigma)))/\nu_2 \, .
\end{eqnarray}


$(ii)$ If $E(0)=G_0(0)=0$, by Proposition~\ref{prop:PhiTheta}, we  have
\begin{eqnarray*}
\tF(t) 
=  \frac{t}{\delta} + \inf_{0 \leq s \leq t}  \left( F(s) - \frac{s}{\delta} \right)  \, ,
\end{eqnarray*}
which implies that 
\begin{eqnarray*}
\Phi(t) 
&=&  \left( \nu_2+ \frac{\lambda_2}{\delta} \right) t + \inf_{0 \leq s \leq t}  \left( \nu_2 s + \lambda_2 F(s) -  \left( \nu_2+ \frac{\lambda_2}{\delta} \right) s \right) \, , \\
&=&  \frac{t}{\delta_2}   + \inf_{0 \leq s \leq t}  \left( \nu_2 s + \lambda_2 F(s) -  \frac{s}{\delta_2} \right) \, .
\end{eqnarray*}
Let us denote by $\Gamma: \mathbbm{R}^+ \to \mathbbm{R}^+$ the inverse of the increasing function $t \mapsto \nu_2 t + \lambda_2 F(t)$.
As $E(0)=G_0(0)=0$, we can apply Corollary~\ref{cor:PsiGamma}, by which there is a last-blowup-trigger time $S_0(\sigma) \leq \sigma$
 such that for all $\tau$, $S_0(\sigma) \leq \tau \leq \sigma$, we have
\begin{eqnarray}\label{eq:PsiG2}
\Psi(\tau) = \Gamma(S_0(\sigma)) + \delta_2 (\tau - S_0(\sigma)) = \Psi(S_0(\sigma)) + \delta_2 (\tau - S_0(\sigma)) \, ,
\end{eqnarray}
where the fact that $\Gamma(S_0(\sigma))=\Psi(S_0(\sigma))$ follows from  $D(S_0(\sigma))=E(T_0(t))=0$.
Indeed, this implies that
\begin{align*}
\Gamma(S_0(\sigma))
&=
\Gamma \big[ \Phi(T_0(t)) \big]\, , \\
 &= \Gamma \big[ \nu_2 T_0(t) +\lambda_2 \tF(T_0(t)) \big] \, , \\
  &= \Gamma \big[ \nu_2 T_0(t) +\lambda_2 (F(T_0(t))-E(T_0(t))) \big] \, , \\
   &=\Gamma \big[  \nu_2 T_0(t) +\lambda_2 F(T_0(t)) \big] \, , \\
 &= T_0(t) \, , \\
  &= \Psi(S_0(t)) \, , 
\end{align*}
where we used the definition of $\Gamma$ as the inverse function of $t \mapsto \nu_2 t + \lambda_2 F(t)$.

$(iii)$  Equating \eqref{eq:PsiG2} and \eqref{eq:PsiG} evaluated for $\tau$ yields
\begin{eqnarray*}
(\tau  - \lambda_2 G(\tau))/\nu_2 - \Psi(S_0(\sigma)) 
=
\delta_2 (\tau - S_0(\sigma)) - \lambda_2 D(\tau) / \nu_2 
<
\delta_2 (\tau - S_0(\sigma))   \, .
\end{eqnarray*}
This shows that
\begin{eqnarray*}
\sup_{S_0(\sigma) \leq \tau \leq \sigma} \left( \frac{\tau  - \lambda_2 G(\tau)}{\nu_2} - \Psi(S_0(\sigma)) - \delta_2 (\tau - S_0(\sigma)) \right) = 0 \, ,
\end{eqnarray*}
so that adding the above zero quantity to \eqref{eq:PsiG2} evaluated at $\sigma$ yields
\begin{eqnarray*}
\Psi(\sigma) 
&=&   \Psi(S_0(\sigma)) + \delta_2 (\sigma - S_0(\sigma)) + \\
&&\hspace{10pt}  \sup_{S_0(\sigma) \leq \tau \leq \sigma} \left( \frac{\tau  - \lambda_2 G(\tau)}{\nu_2}  - \Psi(S_0(\sigma)) - \delta_2 (\tau - S_0(\sigma)) \right)  \, , \\
&=&   \delta_2 \sigma + \sup_{S_0(\sigma) \leq \tau \leq \sigma} \left( \frac{\tau  - \lambda_2 G(\tau)}{\nu_2} - \delta_2 \tau  \right)  \, . \\
\end{eqnarray*}

$(iv)$ To conclude and show that the fixed-point condition \eqref{eq:fixedPoint} must hold, there remains to show that
\begin{eqnarray*}
\sup_{S_0(\sigma) \leq \tau \leq \sigma} \left( \frac{\tau  - \lambda_2 G(\tau)}{\nu_2} - \delta_2 \tau  \right) 
=
\sup_{0 \leq \tau \leq \sigma} \left( \frac{\tau  - \lambda_2 G(\tau)}{\nu_2} - \delta_2 \tau  \right) \, .
\end{eqnarray*}
Let us proceed by contradiction and suppose there is $\tau$, $0 \leq \tau <S_0(\sigma)$ such that
\begin{eqnarray*}
 \frac{\tau  - \lambda_2 G(\tau)}{\nu_2} - \delta_2 \tau > \sup_{S_0(\sigma) \leq \xi \leq \sigma} \left( \frac{\xi  - \lambda_2 G(\xi)}{\nu_2} - \delta_2 \xi  \right) \geq \frac{S_0(\sigma) - \lambda_2 G(S_0(\sigma))}{\nu_2} - \delta_2 S_0(\sigma) \, . 
\end{eqnarray*}
By \eqref{eq:PsiG}, we have
\begin{eqnarray*}
 \frac{\tau  - \lambda_2 G(\tau)}{\nu_2} = \Psi(\tau) - \frac{\lambda_2 D(\tau)}{\nu_2}  
 \quad
 \mathrm{and}
 \quad
 \frac{S_0(\sigma) - \lambda_2 G(S_0(\sigma))}{\nu_2} = \Psi(S_0(\sigma))  \, .
\end{eqnarray*}
Thus, we must have
\begin{eqnarray*}
\Psi(S_0(\sigma)) - \Psi(\tau) < \delta_2 (S_0(\sigma) - \tau) - \lambda_2 D(\tau)/\nu_2 \, ,
\end{eqnarray*}
with $D= E \circ \Psi$ nonnegative.
This implies that $w_\Psi(S_0(\sigma),\tau)<\delta_2$, which contradicts the fact the $\Psi \in \mId$.
This establishes the necessary condition.

 \emph{Sufficient condition.}
Suppose that the fixed-point condition \eqref{eq:fixedPoint} holds for some function $\Psi \in \mId$ and some $C_1$ increasing function $G$.
Define $\Phi=\Psi^{-1} \in \mathcal{T}_\delta$, $F=G \circ \Phi$, and $\tF = (\Phi-\nu_2 \mathrm{Id})/\lambda_2$, which are all  increasing functions.
As $G$ is assumed $C_1$, applying Lemma~\ref{lem:Sup}  to \eqref{eq:fixedPoint} with
$$\sigma \mapsto \Gamma(\sigma) =(\sigma-\lambda_2 G(\sigma))/\nu_2-\delta_2 \sigma  $$
shows that there is a countable set of points $\mathcal{D}$ such that $\Psi$ is continuously differentiable on $\mathbbm{R}^+ \setminus \mathcal{D}$  with
\begin{eqnarray}
&\displaystyle \Psi'(\sigma) = \frac{1-\lambda_2 g(\sigma)}{\nu_2} \quad \mathrm{or} \quad \delta_2 \quad \mathrm{if} \quad \sigma \notin \mathcal{D} \, ,& \nonumber\\
& \displaystyle \Psi_-(\sigma)=\delta_2<\Psi'_+(\sigma)= \frac{1-\lambda_2 g(\sigma)}{\nu_2}   \quad \mathrm{if} \quad  \sigma \in \mathcal{D} \, .& \nonumber
\end{eqnarray}
This allows us to specify with no lack of generality  $\Psi'$ and $\Phi'$ as \cadlag versions of the generalized derivatives of $\Psi$ and $\Phi$, respectively.
It also allows us to define $f = F'$ and $\tf = \tF'$  as \cadlag functions as well.
Since $\Psi \in \mId$, for all $t>0$, there is $\sigma$ such that $t=\Psi(\sigma)>0$.
By definition of $F$, and $\tF$, let us define
\begin{eqnarray}\label{eq:Et}
E(t) 
=
F(t) - \tF(t) \
=
G(\Phi(t)) - \frac{\Phi(t) - \nu_2 t }{ \lambda_2}
=
\frac{\nu_2}{\lambda_2} \left( \Psi(\sigma) - \frac{\sigma - \lambda_2 G(\sigma)}{\nu_2} \right)  \, .
\end{eqnarray}
Note that by the fixed-point condition \eqref{eq:fixedPoint}, we have
\begin{align*}
\Psi(\sigma)- \frac{\sigma - \lambda_2 G(\sigma)}{\nu_2} 
&= 
\left[ \sup_{0 \leq \tau \leq \sigma} \left( \frac{\sigma - \lambda_2 G(\sigma)}{\nu_2} - \delta_2 \sigma \right) \right]_+ - \left(\frac{\sigma - \lambda_2 G(\sigma)}{\nu_2} - \delta_2 \sigma \right) \, , \\
&\geq 
 \sup_{0 \leq \tau \leq \sigma} \left( \frac{\sigma - \lambda_2 G(\sigma)}{\nu_2} - \delta_2 \sigma \right)  - \left(\frac{\sigma - \lambda_2 G(\sigma)}{\nu_2} - \delta_2 \sigma \right) \geq 0 \,  ,
\end{align*}
and the initial condition 
\begin{align*}
E(0) = \frac{\nu_2}{\lambda_2} \left( \Psi(0) + \frac{\lambda_2 G(0)}{\nu_2} \right) = G_0(0) \, .
\end{align*}
We need to show that  $f$ solves the $\delta$-buffer problem associated to the functions $K$, $L$, and $z$ given in \eqref{eq:deltaBuffDef} 
with excess function $E=F-\tF$ and $\delta$-buffered rate $\tf$.

$(i)$ If $E(t)>0$, by \eqref{eq:Et}, we have
\begin{eqnarray*}
 \Psi(\sigma) - \delta_2 \sigma  =  \frac{\sigma - \lambda_2 G(\sigma)}{\nu_2 } - \delta_2 \sigma + \frac{\lambda_2}{\nu_2} E(t)   > \frac{\sigma - \lambda_2 G(\sigma)}{\nu_2 } - \delta_2 \sigma \, .
\end{eqnarray*}
By continuity of $G$, there is $d>0$ such that for all $\tau \in (\sigma-d,\sigma+d)$
\begin{eqnarray*}
\left \vert  \frac{\tau - \lambda_2 G(\tau)}{\nu_2 } - \delta_2 \tau -  \left( \frac{ \sigma - \lambda_2 G(\sigma)}{\nu_2 } - \delta_2 \sigma \right) \right \vert <  \frac{\lambda_2}{\nu_2} E(t) \, .
\end{eqnarray*}
This implies that for all $\tau \in (\sigma-d,\sigma+d)$, the quantity
\begin{eqnarray*}
\Psi(\tau) - \delta_2 \tau 
&=&
 \sup_{0 \leq \xi \leq \tau} \left( \frac{\xi  - \lambda_2 G(\xi)}{\nu_2} - \delta_2 \xi  \right) \, , \\
 &=&  \sup_{0 \leq \xi \leq \sigma-d} \left( \frac{\xi  - \lambda_2 G(\xi)}{\nu_2} - \delta_2 \xi  \right) \, ,
\end{eqnarray*}
is independent of $\tau$, so that $\Psi'(\tau)=\delta_2$. In particular, we have
\begin{eqnarray*}
f(t) = \Phi'(t) g(\Phi(t)) = \frac{g(\Phi(t))}{\Psi'(\sigma)} = \frac{z(t)}{\delta_2} = L(z(t))\, .
\end{eqnarray*}

$(ii)$ Suppose  that $E(t)=0$ so that by \eqref{eq:Et}, we must have
\begin{eqnarray*}
 \Psi(\sigma) = \frac{\sigma - \lambda_2 G(\sigma)}{\nu_2} \, ,
\end{eqnarray*}
where $G$ is  $C_1$.
Observe then that the condition
\begin{eqnarray*}
\frac{1 - \lambda_2 g(\sigma)}{\nu_2}> \delta_2 = \frac{\delta}{\lambda_2 + \nu_2 \delta}
\end{eqnarray*}
is equivalent to
\begin{eqnarray*}
g(\sigma)=g(\Phi(t))=z(t)<z_\delta = \frac{1}{\lambda_2 + \nu_2 \delta} \, .
\end{eqnarray*}

\emph{(a)} If $z(t)<z_\delta$, there is an open interval $(\sigma-d,\sigma+d)$, $d>0$, over which 
\begin{eqnarray*}
\sigma \mapsto \frac{\sigma - \lambda_2 G(\sigma)}{\nu_2} - \delta_2 \sigma
\end{eqnarray*}
is increasing.  
Therefore, for all $\tau$, $\sigma \leq \tau < \sigma+d$,
\begin{eqnarray*}
\Psi(\tau) - \delta_2 \tau 
=
 \sup_{0 \leq \xi \leq \tau} \left( \frac{\xi  - \lambda_2 G(\xi)}{\nu_2} - \delta_2 \xi  \right) 
=  
 \frac{\tau  - \lambda_2 G(\tau)}{\nu_2} - \delta_2 \tau \, ,
\end{eqnarray*}
which implies that for all $s=\Psi(\tau)$, we have 
\begin{eqnarray*}
\Psi'(\tau)  
=
 \frac{1 - \lambda_2 g(\tau)}{\nu_2}
 =
 \frac{1 - \lambda_2 g(\Phi(s))}{\nu_2}
  =
 \frac{1 - \lambda_2 z(s)}{\nu_2}
  \, ,
\end{eqnarray*}
Therefore
\begin{eqnarray*}
f(s) = \Phi'(s) g(\Phi(s)) = \frac{g(\Phi(s))}{\Psi'(\tau)} = \frac{\nu_2 z(s)}{1 - \lambda_2 z(s)} = K(z(s))\, ,
\end{eqnarray*}
so that by right-continuity, we have $f(t) = \lim_{s \to t^+} f(s) = K(z(t))$.

\emph{(b)} If  $z(t)>z_\delta$, there is an open interval $(\sigma-d,\sigma+d)$, $d>0$, over which 
\begin{eqnarray*}
\sigma \mapsto \frac{\sigma - \lambda_2 G(\sigma)}{\nu_2} - \delta_2 \sigma
\end{eqnarray*}
is decreasing.  
Therefore, for all $\tau$, $\sigma \leq \tau < \sigma+d$,
\begin{eqnarray*}
\Psi(\tau) - \delta_2 \tau 
=  
 \frac{\sigma  - \lambda_2 G(\sigma)}{\nu_2} - \delta_2 \sigma \, ,
\end{eqnarray*}
is independent of $\tau$.
This implies that for all $s=\Psi(\tau)$, we have $\Psi'(\tau) = \delta_2$ and
\begin{eqnarray*}
f(s) = \Phi'(s) g(\Phi(s)) = \frac{g(\Phi(s))}{\Psi'(\tau)} = \frac{z(s)}{\delta_2} = L(z(s))\, ,
\end{eqnarray*}
so that by right-continuity, we have $f(t) = \lim_{s \to t^+} f(s) = L(z(t))$.

\emph{(iii)}  Setting $B(t)=\mathbbm{1}_{\{E(t)>0 \} \cup \{ z(t)>z_\delta\}}$, steps \emph{(i)} and  \emph{(ii)} above show that
\begin{eqnarray*}
f(t) = (1-B(t)) K(z(t)) + B(t)L(z(t)) \, ,
\end{eqnarray*}
where we have noticed that the case $z(t)=z_\delta$ is irrelevant as $L(z_\delta)=K(z_\delta)$.
Moreover, we have
\begin{eqnarray*}
\tf(t)
=
\tF'(t)
=
 \left(\lim_{s \downarrow t}\Phi'(s)-\nu_2 \right)/\lambda_2
\end{eqnarray*}
so that 
\begin{eqnarray*}
\tf(t)
=
\left\{
\begin{array}{ccc}
1 / \delta  & \mathrm{if}  & \Phi'(t)=1/\delta_2  \, ,  \\
K(z(t))  &  \mathrm{if}  &  \displaystyle \Phi'(t)= \nu_2 /(1 - \lambda_2 z(t)) \, .
\end{array}
\right.
\end{eqnarray*}
This means that steps \emph{(i)}, \emph{(ii)}, and \emph{(iii)} also show that
\begin{eqnarray*}
\tf(t) = (1-B(t)) K(z(t)) + B(t)/\delta \, .
\end{eqnarray*}
Finally, we have
\begin{eqnarray*}
E(t) 
=
F(t) - \tF(t) 
=
\int_0^t (f(s)-\tf(s)) \, \mdd s 
=
\int_0^t B(s)\big(L(z(s))-1/\delta \big) \, \mdd s  \, ,
\end{eqnarray*}
which shows that $f$ solves  \eqref{eq:deltaBuffDef} with $\delta$-buffered rate $\tf$ and excess function $E$ with $E(0)=G_0(0)$.
This concludes the proof.
\end{proof}


\subsection{Existence and uniqueness of buffered solutions}\label{sec:buffSol}
In this section, following the same arguments as for the nonbuffered dynamics,
we leverage the fixed-point problem \eqref{eq:fixedPoint} to justify that there is a 
unique solution to the $\delta$-buffered dPMF problem \eqref{all:bufferPDE}.

\begin{proposition}\label{prop:contract2}
Given generic initial conditions $(q_0,g_0)$ in $\mathcal{M}(\mathbbm{R}^+) \times \mathcal{M}(\mathbbm{R}^-)$, the map $\mathcal{G}_\delta : \mId \to \mId$ specified by
\begin{eqnarray}\label{eq:mF}
\Psi \mapsto \mathcal{G}_\delta[\Psi] = \left\{ \sigma \mapsto \delta_2 \sigma +  \left[ \sup_{0 \leq \tau \leq \sigma} \left( \frac{\tau-\lambda_2 G[\Psi](\tau)}{\nu_2} - \delta_2 \tau \right) \right]_+ \right\}
\end{eqnarray}
is a contraction on $\mId([0,\sigma])$ for small enough $\sigma>0$.
\end{proposition}

\begin{proof}
The same exact arguments as for the proof of Proposition~\ref{prop:contract2} apply, except for
the map $\mathcal{G}_\delta$ being defined $\mId \to \mId$.
\end{proof}

\begin{theorem}\label{th:deltaGlobSol}
Given generic initial conditions $(q_0,g_0)$ in $\mathcal{M}(\mathbbm{R}^+) \times \mathcal{M}(\mathbbm{R}^-)$, for all buffer parameter $\delta>0$, there is a unique global solution $\Psi$ to the dPMF fixed-point problem \eqref{eq:fixedPoint} in $\mId=\mId([0,\infty))$.
Moreover, this solution specifies the unique solution to the $\delta$-buffered dPMF problem \eqref{all:bufferPDE} as 
 $\mathbbm{R}^+ \times \mathbbm{R}^+ \ni (t,x) \mapsto p(x,t)=q(\Phi(t),x)$, where $\Phi=\Psi^{-1}$ and where
$(\sigma,x) \mapsto q(\sigma,x)$ is unique solution to the time-changed problem \eqref{all:TCPDE} for $\Psi$.
\end{theorem}

\begin{proof} 

The same exact Banach fixed-point arguments as for the proof of Theorem~\ref{th:globSol} apply
to the contraction map $\mathcal{G}_\delta:\mId \to \mId$ from Proposition~\ref{prop:contract2}.
These arguments show that  the solution $\Psi \in \mId$ to the fixed-point problem \eqref{eq:fixedPoint} specifies
a solution to the $\delta$-buffered dPMF problem \eqref{all:bufferPDE}.

By contrast with the proof of Theorem~\ref{th:globSol}, however, the $\delta$-buffered time change 
$\Phi(t)=\nu_2 t + \lambda_2 \tF (t)$ is always a continuous, increasing function so that
the change of variable $\sigma=\Phi(t) \Leftrightarrow t = \Psi(\sigma)$ is always
well defined. 
This implies that Proposition~\ref{prop:PhiToPsi} holds for $\delta$-buffered dPMF dynamics independent of the occurrence of (regularized) blowups,
which establishes the uniqueness of the solution to the $\delta$-buffered dPMF problem \eqref{all:bufferPDE}.
\end{proof}

We claim that the $\delta$-buffered solutions from Theorem~\ref{th:deltaGlobSol} represent physical
solutions as they correspond to a regularized mean-field particle system that obeys the same conservation
principle as the original dPMF dynamics.
That conservation principle states that all spiking particles must impact the evolution of non-spiking particles.
This amounts to a form of interaction-rate conservation, which is enacted by the buffer mechanism.


\subsection{Recovering physical candidate solutions when $\delta \downarrow 0$}\label{sec:recov}
In this section, we show that one can define our candidate physical dPMF dynamics by considering $\delta$-buffered dPMF solutions
in the limit $\delta \downarrow 0$. 
To do so, we will proceed by joint subsequence extractions of converging 
cumulative rate functions $G_n \to G$ and converging inverse time changes $\Psi_n \to \Psi$.
We then show that all the thus obtained pairs $(G,\Psi)$ must satisfy the same fixed-point problem 
as the one defining our candidate physical solutions in Proposition~\ref{prop:physFP}.
This result  will validate the physicality of the latter solutions.

To set up the extraction framework,
let us consider the set of  $\delta$-buffered dPMF dynamics, $\delta>0$, with the same refractory delay 
distribution $P_\epsilon$ in $\mathcal{M}([\epsilon,\infty))$  and  for the same generic initial conditions $(q_0,g_0)$ in $\mathcal{M}(\mathbbm{R}^+) \times \mathcal{M}((-\infty,0))$.
Each of these dynamics is equivalently specified by a unique cumulative function $G_\delta \in C_1(\mathbbm{R}^+)$ or a unique inverse time change 
$\Psi_\delta \in \mId$ such that $G_\delta=G[\Psi_\delta] \Leftrightarrow \Psi_\delta =\Psi[G_\delta]$.
By Proposition~\ref{prop:boundedness}, for all $\Psi \in \mIz= \cup_{\delta \geq 0} \mId $, the functions $g = \partial_\sigma G = \partial_\sigma G[\Psi]$ 
are  bounded over $[0,\sigma]$, $\sigma>0$, by a constant that depends neither on $\Psi_\delta \in \mIz$ 
nor on $P_\epsilon$ in $\mathcal{M}([\epsilon,\infty))$.
Therefore,  by Arz\`ela-Ascoli theorem, for all $\sigma>0$, the set
\begin{eqnarray*}
\{ G[\Psi] \in  C_1([0,\sigma]) \,  \vert \,  \Psi \in \mIz = \cup_{\delta \geq 0} \mId \} \, , 
\end{eqnarray*}
is relatively compact.
This allows one to extract a compactly converging subsequence $G_n=G_{\delta_n}$, $n \in \mathbbm{N}$ with $\delta_n \downarrow 0$.
Let us denote the necessarily continuous, nondecreasing limit function as $G=\lim_{n \to \infty} G_n$.

The inverse time changes $\Psi_n=\Psi_{\delta_n}\in \mIn \subset \mIz$ associated to $G_n$ are given via the fixed-point equation \eqref{eq:physFP} as
\begin{eqnarray*}
\Psi_n(\sigma) = \Psi[G_n](\sigma) = \delta_n \sigma + \left[ \sup_{0 \leq \tau \leq \sigma} \left( \frac{\sigma - \lambda_2 G_n(\sigma)}{\nu_2} - \delta_ n  \sigma \right) \right]_+ \, .
\end{eqnarray*}
Given any real functions $G_a$ and $G_b$, we have 
\begin{align*}
& 
\hspace{-10pt} \big \vert \Psi[G_a](\sigma) -  \Psi[G_b](\sigma)  \big \vert  \\
&\leq
\sigma \vert \delta_a-\delta_b \vert + 
\left \vert  \left[ \sup_{0 \leq \tau \leq \sigma} \left( \frac{\tau-\lambda_2 G_a(\tau)}{\nu_2} - \delta_a \tau \right) \right]_+ - \left[ \sup_{0 \leq \tau \leq \sigma} \left( \frac{\tau-\lambda_2 G_b(\tau)}{\nu_2} - \delta_b \tau \right)\right]_+ \right \vert \, , \\
&\leq
\sigma \vert \delta_a-\delta_b \vert + 
\left \vert \sup_{0 \leq \tau \leq \sigma} \left( \frac{\tau-\lambda_2 G_a(\tau)}{\nu_2} - \delta_a \tau \right) - \sup_{0 \leq \tau \leq \sigma} \left( \frac{\tau-\lambda_2 G_b(\tau)}{\nu_2} - \delta_b \tau \right) \right \vert \, , \\
&\leq
\sigma \vert \delta_a-\delta_b \vert + 
\sup_{0 \leq \tau \leq \sigma} \left \vert \frac{\lambda_2 (G_b(\tau) - G_a(\tau))}{\nu_2} + (\delta_b - \delta_a) \tau  \right \vert \, , \\
&\leq
2 \sigma \vert \delta_a-\delta_b \vert  
+\frac{\lambda_2}{\nu_2} \sup_{0 \leq \tau \leq \sigma}\left \vert   G_a(\tau) -  G_b(\tau)  \right \vert \, , \\
& =
2 \sigma \vert \delta_a-\delta_b \vert  + \frac{\lambda_2}{\nu_2}  \Vert G_a - G_b \Vert_{0,\sigma} \, .
\end{align*}
This shows that the sequence $\Psi_n$ also compactly converges toward a continuous, nondecreasing function $\Psi$
satisfying
\begin{eqnarray} \label{eq:PsiFromG}
\Psi(\sigma) =  \frac{1}{\nu_2} \left[ \sup_{0 \leq \tau \leq \sigma} \big(\sigma - \lambda_2 G(\sigma)  \big) \right]_+ \quad \mathrm{for \: all} \quad \sigma \geq 0 \, .
\end{eqnarray}
Thus the limit cumulative function $G$ and the limit inverse time change $\Psi$ obtained when $\delta_n \downarrow 0 $ satisfy the 
same relation as in the fixed-point equation \eqref{eq:physFP} defining our physical dPMF candidate solutions.

To show that these limit functions are actually identical to the physical dPMF candidate solutions,
it remains to show that $G$ satisfy the renewal-type equation \eqref{eq:Grenew} for the inverse time change $\Psi$, i.e., $G=G[\Psi]$
or $ \lim_{n\to \infty} G[\Psi_n] = G[\lim_{n \to \infty} \Psi_n]$.

\begin{proposition}\label{prop:GHEta}
Consider a compactly converging sequence of cumulative rate functions $G_n \to G$ for $\delta_n \downarrow 0$ and its associated compactly converging sequence of inverse time change $\Psi_n=\Psi[G_n] \to \Psi[G]=\Psi$.
Then, the limit functions $G$ and $\Psi$ satisfy  \eqref{eq:Grenew}, i.e., 
\begin{align}\label{eq:GrenewLim}
G(\sigma)
&=
 \int_0^\infty H(\sigma,0,x) \, q_0(x) \, \mdd x + \int_\tau^\sigma H(\sigma,\tau,\Lambda) \, \mdd G_\epsilon(\tau)  \, , \quad G(0)=G_0(0) \, ,
\end{align}
where the cumulative first-passage time kernel $H=H[\Psi]$ and the cumulative reset rate $G_\epsilon=G_\epsilon[\Psi]$ 
are defined in terms of  $\Psi$ as in Proposition~\ref{prop:solG}.
\end{proposition}

To prove the above Proposition, we will use the two following technical lemmas,
whose proves rely on Lemma~\ref{lem:Hprop} and Lemma~\ref{lem:FP}  and are given in Section~\ref{sec:buffextract}.

\begin{lemma}\label{lem:Hconv}
Assume $\Psi_n\rightarrow\Psi$ compactly on $\mathbbm{R}^+$. 
Then given $x>0$, the functions $(\sigma,\tau) \mapsto H_n(\sigma,\tau,x)=H[\Psi_n](\sigma,\tau,x)$ converge uniformly toward $H[\Psi](\sigma,\tau,x)=H(\sigma,\tau,x)$ on the subsets of the form $A_\xi=\{(\sigma,\tau)\;|\; \xi \geq \sigma\geq\tau\geq0\}$, $\xi>0$.
\end{lemma}

\begin{lemma}\label{lem:HLipsch}
Given $x,\sigma>0$, $\tau \mapsto  H(\sigma,\tau,x)$ is Lipchitz continuous on $[0,\xi]$, $\xi>0$ for a constant that is independent of $\Psi \in \mIz$.
\end{lemma}

\begin{proof}[Proof of Proposition~\ref{prop:GHEta}]

Consider the refractory-delay distributions 
\begin{align*}
Q_{\sigma,n}(\eta) = P_\epsilon(\Psi_n(\sigma)-\Psi_n(\sigma-\eta)) 
\quad \mathrm{and} \quad
Q_\sigma(\eta) = P_\epsilon(\Psi(\sigma)-\Psi(\sigma-\eta)) \, ,
\end{align*}
and let us denote for conciseness
\begin{align*}
G_{\epsilon,n}(\sigma)= \int_0^\infty G_n(\sigma-\eta) \, \mdd Q_{\sigma,n}(\eta)  \, .
\end{align*}
By Proposition~\ref{prop:solG}, for all $n \in \mathbbm{N}$, we have 
\begin{align*}
G_n(\sigma)
&=
 \int_0^\infty H_n(\sigma,0,x) \, q_0(x) \, \mdd x +  \int_\tau^\sigma H_n(\sigma,\tau,\Lambda) \, \mdd  G_{\epsilon,n}(\tau)  \, ,
\end{align*}
where the cumulative first-passage time kernel $H_n$ also depends on $\Psi_n$.
We will prove the announced result by showing that for all $\sigma>0$, the quantities 
\begin{align*}
\Delta I_{1,n}(\sigma)
&=  \int_0^\infty H_n(\sigma,0,x) \, q_0(x) \, \mdd x -   \int_0^\infty H(\sigma,0,x) \, q_0(x) \, \mdd x   \, , \\
\Delta I_{2,n}(\sigma)
&=  \int_0^\sigma H_n(\sigma,\tau,\Lambda) \, \mdd G_{\epsilon,n}(\tau-\eta)  - \int_0^\sigma H(\sigma,\tau,\Lambda) \, \mdd G_\epsilon(\tau-\eta)   \, , 
\end{align*}
both converge to zero when $n \to \infty$.

Lemma ~\ref{lem:Hconv} implies that for all $x>0$, $x \mapsto H_n(\sigma,0,x)$ converges toward $x \mapsto H_n(\sigma,0,x)$ pointwise.
Moreover, given generic initial conditions, the same arguments as in the proof of Proposition~\ref{prop:boundedness} shows that the functions
$x \mapsto H_n(\sigma,0,x) q_0(x)$ are uniformly bounded by an integrable function.
Thus, the convergence of the first term, $\Delta I_{1,n}(\sigma) \to 0$ when $n \to \infty$, directly follows from dominated convergence.

For the second term, let us  write
\begin{align*}
\Delta I_{2,n}(\sigma)
&= \int_0^\sigma \big [ H_n(\sigma,\tau,\Lambda)-  H(\sigma,\tau,x) \big ]  \, \mdd G_{\epsilon,n}(\tau)  + \int_0^\sigma H(\sigma,\tau,\Lambda) \, \mdd \left[ G_{\epsilon,n}(\tau) - G_\epsilon(\tau)\right]  \, , \\
&=\int_0^\sigma \big [ H_n(\sigma,\tau,\Lambda)-  H(\sigma,\tau,x) \big ]  \, \mdd G_{\epsilon,n}(\tau)  + H(\sigma,\tau,\Lambda) \left[G_{\epsilon,n}(\sigma)- G_\epsilon(\sigma)\right]  \\
& \hspace{20pt}   -\int_0^\sigma \partial_\tau H(\sigma,\tau,\Lambda) \left[G_{\epsilon,n}(\tau) - G_\epsilon(\tau) \right]  \, \mdd \tau  \, ,
\end{align*}
where the last equality follows from integration by parts.
This is justified because $\tau \mapsto  H(\sigma,\tau,x)$ is Lipchitz continuous by Lemma~\ref{lem:HLipsch}. 
As $H \geq 0$ and for all $n \in \mathbbm{N}$, $G_n$ is increasing, we have
\begin{align*}
 \big \vert \Delta I_{2,n}(\sigma) \big \vert
& \leq \big \Vert H_n(\sigma,\cdot,\Lambda)-  H(\sigma,\cdot,\Lambda) \big \Vert_{0,\sigma}  \big( G_{\epsilon,n}(\sigma)-G_{\epsilon,n}(0) \big)     \\
& \hspace{20pt} + H(\sigma,\tau,\Lambda)   \big \vert G_{\epsilon,n}(\sigma)  -G_\epsilon(\sigma)  \big \vert  \\
& \hspace{20pt} -\int_0^\sigma \big \vert  \partial_\tau H(\sigma,\tau,\Lambda) \big \vert \big \vert G_{\epsilon,n}(\tau) -G_\epsilon(\tau-\eta) \big \vert \, \mdd \tau  \, . 
\end{align*}
For all $\sigma>0$, we know by Proposition~\ref{lem:Hconv} that $\tau \mapsto H_n(\sigma,\tau,\Lambda)$ converges compactly toward 
$\tau \mapsto H_n(\sigma,\tau,\Lambda)$, whereas $G_{\epsilon,n}$ and $G_\epsilon$ are all uniformly bounded on the set $[0,\sigma]$.
Therefore
\begin{eqnarray*}
\big \Vert H_n(\sigma,\cdot,\Lambda)-  H(\sigma,\cdot,\Lambda) \big \Vert_{0,\sigma}  \big( G_{\epsilon,n}(\sigma)-G_{\epsilon,n}(0) \big) 
  \xrightarrow[n \to \infty]{} 0 \, .
\end{eqnarray*}
It remains to show that
\begin{align*}
H(\sigma,\tau,\Lambda)   \big \vert G_{\epsilon,n}(\sigma)  -G_\epsilon(\sigma)  \big \vert  -\int_0^\sigma \big \vert  \partial_\tau H(\sigma,\tau,\Lambda) \big \vert \big \vert G_{\epsilon,n}(\tau) -G_\epsilon(\tau-\eta) \big \vert \, \mdd \tau   \xrightarrow[n \to \infty]{} 0 \, .
\end{align*}
To that end, remember that $H \leq 1$ and that by Lemma~\ref{lem:HLipsch}, $\partial_\tau H$ is bounded over $\{ (\sigma,\tau) \, \vert \, 0 \leq \tau \leq \sigma \}$
uniformly in $n$ and $\tau$.
Thus, by dominated convergence, the result will follow from establishing the point-wise convergence
of $G_{\epsilon,n}$ toward $G_\epsilon$.
This follows from observing that for all $\sigma>0$
\begin{align*}
\vert G_{\epsilon,n}(\sigma)  -G_\epsilon(\sigma)  \big \vert 
&\leq 
\int_0^\infty \big \vert  G_n(\sigma-\eta)  -G(\sigma-\eta) \big \vert \, \mdd Q_{\sigma,n}(\eta)  \\
& \hspace{40pt}
+\left \vert \int_0^\infty G(\sigma-\eta) \, \mdd \big[ Q_{\sigma,n}(\eta) -Q_\sigma(\eta)  \big]\right \vert \, , \\
&\leq
 \Vert  G_n  -G \Vert_{0,\sigma}  
+\left \vert \int_0^\infty G(\sigma-\eta) \, \mdd \big[ Q_{\sigma,n}(\eta) -Q_\sigma(\eta)  \big]\right \vert \, .
\end{align*}
The integral term above
converges to zero  by continuity of $G$ since  $Q_{\sigma,n}$
weakly converges toward $Q_\sigma$ as $n \to \infty$. 
This is justified by our assumption that  $P_\epsilon$ has a (smooth) density so that, in particular, $P_\epsilon$ is continuous,
and we have the point-wise convergence
\begin{eqnarray*}
Q_{\sigma,n}(\eta) = P_\epsilon(\Psi_n(\sigma)-\Psi_n(\sigma-\eta)) \xrightarrow[n \to \infty]{} P_\epsilon(\Psi(\sigma)-\Psi(\sigma-\eta)) = Q_\sigma(\eta) \, .
\end{eqnarray*}
Since  for all $\sigma>0$, we also have  $\Vert  G_n  -G \Vert_{0,\sigma}  \to 0$ as $n \to \infty$, this shows 
the pointwise convergence $G_{\epsilon,n} \to G_\epsilon$ as $n \to \infty$ and establishes the limit quasi-renewal equation \eqref{eq:GrenewLim}.
\end{proof}

Since we know that \eqref{eq:PsiFromG} holds, Proposition~\ref{prop:GHEta} shows that for any converging
sequence $G_n(=G_{\delta_n}) \to G$ with $\delta_n \downarrow 0$, the limit time change $\Psi[G] \in \mIz$
parametrizes the same time-changed dPMF dynamics as in Proposition~\ref{prop:physFP}.
Given the relative compactness of $\mIz([0,\sigma])$, $\sigma>0$, this shows that the physical dPMF solutions stated in Theorem~\ref{th:mainRes} 
is the limit of the corresponding $\delta$-buffered dynamics when $\delta \downarrow 0$, 
thereby establishing Theorem~\ref{th:mainRes2}.


\section{Complementary results}\label{sec:compRes}

In this section, we give some complementary results that while not central
to our core results, explain the interplay of the time change approach, the buffer mechanism, 
and the delayed nature of the dynamics in establishing these core results.
In Section~\ref{sec:buffMTC}, we exhibit how the time change $\sigma=\Phi(t) \Leftrightarrow t=\Psi(\sigma)$ regularizes the buffer mechanism when $\delta>0$.
In Section~\ref{sec:buffPers}, we show that the action of the buffered mechanism persists in physical dPMF dynamics when $\delta \downarrow 0$.
In Section~\ref{sec:delayBlowup}, we justify that considering refractory delays bounded away from zero (with $\epsilon>0$) allows one to avoid unphysical eternal blowups.


\subsection{Time change regularizes the buffer mechanism}\label{sec:buffMTC}
In Section~\ref{sec:recov}, we have demonstrated that our notion of physical solutions for 
blowup dPMF dynamics can be recovered as limit $\delta$-buffered solutions when $\delta \downarrow 0$.
This demonstration hinges on the fact that the limiting process $\delta \downarrow 0$ is well-behaved
when $\delta$-buffered solutions are considered in the time-changed picture.
To understand how the time change  $\sigma=\Phi(t) \Leftrightarrow t=\Psi(\sigma)$ regularizes the buffer mechanism, 
it is  instructive to transform the fixed-point equation \eqref{eq:fixedPoint} from Proposition~\ref{prop:PsiToPhi} into 
a statement about the time-changed, buffered dPMF dynamics.


\begin{proposition}\label{prop:bufferTCdPMF}
\begin{subequations}\label{all:bufferTCDdPMF}
Given $\delta>0$ and generic initial conditions $(q_0,g_0)$ in $\mathcal{M}(\mathbbm{R}^+) \times \mathcal{M}(\mathbbm{R}^-)$, the density function $(\sigma,x) \mapsto q(\sigma,x)$ of the time-changed, $\delta$-buffered dPMF dynamics satisfying  \eqref{eq:fixedPoint} solves the PDE problem
\begin{align}\label{eq:bufferTCPDE}
\partial_\sigma q = \left( \left( \lambda_1-\frac{\nu_1}{\nu_2} \lambda_2 \right) \tg(\sigma) + \frac{\nu_1}{\nu_2} \right) \partial_x q+\frac{1}{2} \partial_x^2 q+ \partial_\sigma G_\epsilon(\sigma)   \delta_\Lambda \, , \quad q(\sigma,0)=0 \, ,
\end{align}
on $\mathbbm{R}^+) \times \mathbbm{R}^+$  with cumulative function $G$ and cumulative reset rate $G_\epsilon$ such that
\begin{eqnarray}\label{eq:TCgeta1}
\partial_\sigma G(\sigma) = g(\sigma) = \partial_x q(\sigma,0)/2  \, , \quad G(0)=G_0(0) \, , 
\end{eqnarray}
\begin{eqnarray}\label{eq:TCgeta2}
G_\epsilon(\sigma) = \int G(\sigma - \eta) \, \mdd P_\epsilon(\Psi(\sigma) - \Psi(\sigma-\eta) ) \, , 
\end{eqnarray}
and where the inverse time change $\Psi=\Phi^{-1}$ is specified via 
\begin{eqnarray}\label{eq:tG}
\Psi(\sigma) = (\sigma - \lambda_2 \tG(\sigma))/\nu_2 =  \left(\sigma - \lambda_2 \int_0^\sigma \tg(\tau) \, \mdd \tau \right)/\nu_2 \, .
\end{eqnarray}
In turn, $\tg$  is defined as the $(\delta_2/\delta)$-buffered version of $g$
\begin{eqnarray}\label{eq:TCbuffer}
\tg (\sigma)= (1-A(\sigma)) g(\sigma) + A(\sigma)\, \delta_2/\delta \, ,
\end{eqnarray}
where the blowup indicator function $A$ is specified as a $\{0,1\}$-valued function by
\begin{eqnarray}\label{eq:TCindicator}
A(\sigma)  = \mathbbm{1}_{\{ g(\sigma) > \delta_2/\delta \} \cup \{D(\sigma) > 0\}}    \, .
\end{eqnarray}
and where the excess function $D$ satisfies
\begin{eqnarray}\label{eq:TCexcess}
\partial_\sigma D(\sigma)  &=&  A(\sigma) (g(\sigma) - \delta_2/ \delta)  \quad \mathrm{with} \quad D(0)=G_0(0)\, .
\end{eqnarray}
\end{subequations}
\end{proposition}

\begin{proof}
We proceed in several steps.

$(i)$
Assume that $\Psi$ solves the fixed-point equation \eqref{eq:fixedPoint} with $C_1$ cumulative function denoted as $G=G[\Psi]$.
Let us set $\Phi=\Psi^{-1}$ and $F=G \circ \Phi$, so that the excess function $E$ 
for the original $\delta$-buffered dPMF dynamics can be specified by  
\begin{eqnarray*}
\Phi(t) = \nu_2 t + \lambda_2 \tF(t) =  \nu_2 t + \lambda_2 (F(t) - E(t)) \quad \Rightarrow \quad E(t) = F(t) - (\Phi(t)-\nu_2 t)/ \lambda_2 \, .
\end{eqnarray*}
Observe that as solution to  \eqref{eq:fixedPoint}, $\Psi$, and therefore $\Phi$, $F$, and $E$, are all continuously differentiable on 
$\mathbbm{R}^+$ except on a countable number of points $\mathcal{D}$.
Denoting by $f$ the right-continuous version of the generalized derivative $\mdd F/\mdd t$,
the blowup indicator function is specified as  $B(t) = \mathbbm{1}_{\{f(t)>1/\delta\} \cup \{ E(t)>0\}}$.

$(ii)$
By definition of the buffer mechanism, we have
\begin{eqnarray*}
f 
&=&  (1-B ) (K \circ z) + B (L \circ z) \, ,  
\end{eqnarray*}
with buffer mechanism functions
\begin{eqnarray*}
K(z) =   \frac{ \nu_2 z}{1 - \lambda_2 z} \, , \quad L(z) =  \frac{z}{\delta_2} \, ,  \quad \mathrm{and} \quad  z(t) = g(\Phi(t))  \, .
\end{eqnarray*}
In particular, we have
\begin{eqnarray*}
f > 1/\delta 
\quad \Leftrightarrow \quad
z>z_\delta = \frac{1}{\lambda_2+\nu_2 \delta} = \frac{\delta}{\delta_2}
\quad \Leftrightarrow \quad
g \circ \Phi >  \frac{\delta}{\delta_2} \, .
\end{eqnarray*}
Defining the time-changed excess function and blowup indicator function as $D(\sigma)=E(\Psi(\sigma)) \geq 0$
and $A(\sigma)=B(\Psi(\sigma))$, respectively, this implies that
\begin{eqnarray*}
A(\sigma)=B(\Psi(\sigma)) = \mathbbm{1}_{\{ f(\Psi(\sigma)) >1/\delta\} \cup \{ E(\Psi(\sigma)) > 0\}} = \mathbbm{1}_{\{ g(\sigma) >\delta_2/\delta\} \cup \{ D(\sigma) > 0\}} \, .
\end{eqnarray*}
This shows that \eqref{eq:TCindicator} holds.

$(iii)$
By \eqref{eq:TCdrift}, the drift function $\mu$ featured in the time-changed PDE problem reads
\begin{align*}
\mu 
&=  \left( \nu_1 - \frac{\lambda_1}{\lambda_2} \nu_2   \right)\Psi'  +  \frac{\lambda_1}{\lambda_2}  \, .
\end{align*}
As solution to  the fixed-point equation \eqref{eq:fixedPoint}, outside of $\mathcal{D}$, $\Psi$ satisfies
\begin{eqnarray*}
\Psi' = (1-A) (1-\lambda_2 g)/\nu_2 + A \delta_2 \, ,
\end{eqnarray*}
so that if $A=0$, we have
\begin{align*}
\mu 
=  \left( \nu_1 - \frac{\lambda_1}{\lambda_2} \nu_2   \right) \frac{1-\lambda_2 g}{\nu_2}  +  \frac{\lambda_1}{\lambda_2}  
= \left( \lambda_1-\frac{\nu_1}{\nu_2} \lambda_2 \right) g + \frac{\nu_1}{\nu_2} \, ,
\end{align*}
and if $A=1$, we have
\begin{align*}
 \mu 
=  \left( \nu_1 - \frac{\lambda_1}{\lambda_2} \nu_2   \right) \delta_2  +  \frac{\lambda_1}{\lambda_2} 
= \left( \lambda_1-\frac{\nu_1}{\nu_2} \lambda_2 \right) \frac{\delta_2}{\delta} + \frac{\nu_1}{\nu_2}  \, .
\end{align*}
Since $\mathcal{D}$ is countable, we can always choose to define $\Psi'$ as the right-continuous 
generalized derivative of $\Psi$. Taking this convention implies that on $\mathbbm{R}^+$, we have
\begin{eqnarray*}
 \mu = \left( \lambda_1-\frac{\nu_1}{\nu_2} \lambda_2 \right) \left((1-A)g + A\frac{\delta_2}{\delta} \right) + \frac{\nu_1}{\nu_2} \, ,
\end{eqnarray*}
where we recognize $\tg$, the $(\delta/\delta_2)$-buffered version of $g$ defined in \eqref{eq:TCbuffer}.
This shows that \eqref{eq:bufferTCPDE}, \eqref{eq:TCgeta1},  \eqref{eq:TCgeta2}, and \eqref{eq:TCbuffer} hold.

$(iv)$
By definition, we have
\begin{eqnarray*}
D(\sigma)
&=& E(\Psi(\sigma)) \, , \\
&=& \int_0^{\Psi(\sigma)} B(s) \big( f(s) - 1/\delta \big) \, \mdd s \, , \\
&=& \int_0^{\sigma}  B(\Psi(\tau)) \big( f(\Psi(\tau)) - 1/\delta \big) \, \mdd \Psi (\tau) \, , \\
&=& \int_0^{\sigma} A(\tau) \big(g(\tau) -  \Psi' (\tau)/\delta \big) \, \mdd \tau \, , \\
&=& \int_0^{\sigma} A(\tau) \big( g(\tau) -  \delta_2/\delta \big) \, \mdd \tau \, ,
\end{eqnarray*}
where we used the fact that  $\Psi' (\sigma) = \delta_2$, whenever $\sigma=\Phi(t)$
 is a blowup time, i.e., whenever $A(\sigma)=B(\Psi(\sigma))=1$.
This shows that \eqref{eq:TCexcess} holds.

$(v)$
Finally, by definition of the excess function $D$, we have
\begin{eqnarray*}
\Psi(\sigma) = \big(\sigma - \lambda_2 (G(\sigma) - D(\sigma))\big)/\nu_2 \, ,
\end{eqnarray*}
where using the definition of $\tg$ in \eqref{eq:TCbuffer}, we have
\begin{align*}
G(\sigma) - D(\sigma)
&=
\int_0^\sigma g(\tau) \, \mdd \tau - \int_0^{\sigma} A(\tau) \big( g(\tau) -  \delta_2/\delta \big) \, \mdd \tau \, , \\
&=
 \int_0^{\sigma}  \big( (1-A(\tau)) g(\tau)  +  A(\tau) \delta_2/\delta \big) \, \mdd \tau \, , \\
&=
 \int_0^{\sigma} \tg(\tau) \, \mdd \tau \, , \\
 &=
 \tG(\sigma)  \, .
\end{align*}
This shows that \eqref{eq:tG} holds.

\end{proof}

We are now able to exhibit how performing the time change $\sigma=\Phi(t) \Leftrightarrow t=\Psi(\sigma)$
regularizes the action of the buffer mechanism.
Indeed, Proposition~\ref{prop:bufferTCdPMF} reveals that the PDE problem governing buffered dPMF dynamics 
involves a constant diffusion coefficient with value $1/2$ in the time-changed picture (see \eqref{eq:bufferTCPDE}). 
As a result, the time-changed buffer mechanism degenerates in the sense that 
it corresponds to identical buffer functions $L(z)=K(z)=z/2$ within the analytical framework of Section~\ref{sec:bufferAnalysis}.
This implies that the output rate of the buffer mechanism, i.e., 
the time-changed firing rate $g$, must be continuous, even at the exit times of regularized blowups.
This is by contrast with the firing rate $f$ in the original time picture, 
which exhibits jump discontinuities at blowup exit times by Proposition~\ref{prop:buffReg}.


\subsection{Persistence of the buffer mechanism in the time-changed picture}\label{sec:buffPers}
Informally, the limiting process $\delta \downarrow 0$ yields physical solutions to the original dPMF problem in Definition~\ref{def:mainProb1}
by asymptotically setting the threshold value of the buffer mechanism to infinity in Definition~\ref{def:bufferPDE}.
Although it is not obvious in the original-time picture, it is interesting to note that even with infinite threshold, physical dPMF dynamics retained 
their buffered nature in the time-changed picture when $\delta \downarrow 0$.
This is simply because  in the time-changed picture, the buffer parameter $\delta_2/\delta$ featured in Proposition~\ref{prop:bufferTCdPMF} 
converges to the finite value $\delta_2/\delta \to 1/\lambda_2$ when $\delta \downarrow 0$.
Specifically, we have the following:

\begin{proposition}\label{def:lmTCdPMF}
\begin{subequations}\label{all:limTCDdPMF}
Given generic initial conditions $(q_0,g_0)$ in $\mathcal{M}(\mathbbm{R}^+) \times \mathcal{M}(\mathbbm{R}^-)$, the density function $(\sigma,x) \mapsto q(\sigma,x)$ of the time-changed dPMF dynamics solves the PDE problem
\begin{align}\label{eq:limTCPDE}
\partial_\sigma q = \left( \left( \lambda_1-\frac{\nu_1}{\nu_2} \lambda_2 \right) \tg(\sigma) + \frac{\nu_1}{\nu_2} \right) \partial_x q+\frac{1}{2} \partial_x^2 q+ \partial_\sigma G_\epsilon(\sigma)   \delta_\Lambda \, , \quad q(\sigma,0)=0
\end{align}
on $\mathbbm{R}^+ \times\mathbbm{R}^+$  with cumulative rate  $G$ and cumulative reset rate $G_\epsilon$ such that
\begin{eqnarray*}\label{eq:limTCgeta}
&\partial_\sigma G(\sigma) = g(\sigma) = \partial_x q(\sigma,0)/2 \, , \quad G(0)=G_0(0) \, ,& \\
& G_\epsilon(\sigma) = \int_0^\infty  G(\sigma-\eta) \, \mdd P_\epsilon (\Psi(\sigma)-\Psi(\sigma -\eta) ) \, , & \label{eq:limTCgeta2}
\end{eqnarray*}
where the inverse time change $\Psi=\Phi^{-1}$ is specified via 
\begin{eqnarray}\label{eq:limtG}
\Psi(\sigma) = (\sigma - \lambda_2 \tG(\sigma))/\nu_2 =  \left(\sigma - \lambda_2 \int_0^\sigma \tg(\tau) \, \mdd \tau \right)/\nu_2 \, .
\end{eqnarray}
and where $\tg$  is defined as the $(1/\lambda_2)$-buffered version of $g$.
\end{subequations}
\end{proposition}

\begin{proof}
We only have to show that given $\delta_n \downarrow 0$, the physical dPMF 
solution parametrized by $\Psi = \lim_{n \to \infty} \Psi_n$ obeys a $(1/\lambda_2)$-buffer
mechanism.
This is equivalent to showing that the physical dPMF solution admits an excess function 
$D$ specified as
\begin{eqnarray}
D(\sigma) = \int_{S_0(\sigma)}^\sigma ( g(\tau) - 1/\lambda_2 ) \, \mdd \tau \, ,
\end{eqnarray}
where as usual, $S_0(\sigma)$ is the last-blowup-trigger time.

To show this, we proceed in three steps:

$(i)$ On one hand, by the results of Section~\ref{sec:recov}, we know that  any converging sequence $\delta_n \downarrow 0$
specifies a  compactly convergent sequence of pairs  $(\Psi_n,G_n) \to (\Psi,G)$, where 
$G_n$ and $G$ denote the cumulative rate function associated to $\Psi_n$ and $\Psi$, respectively.
Therefore, by Proposition~\ref{prop:bufferTCdPMF}, for each $n \in \mathbbm{N}$, 
the excess function
\begin{align*}
\sigma \mapsto D_n(\sigma) =  G_n(\sigma)+ (\nu_2 \Psi_n(\sigma)-\sigma)/\lambda_2  \, , 
\end{align*}
also compactly converges toward the limit excess function
\begin{align}\label{eq:comp1}
\sigma \mapsto D(\sigma) =  G(\sigma) + (\nu_2 \Psi(\sigma) -\sigma )/\lambda_2  \, ,
\end{align}
which was originally  introduced in \eqref{eq:firstD}.
 
$(ii)$ On the other hand, by Proposition~\ref{prop:bufferTCdPMF}, the excess functions $D_n$ may be expressed in terms 
of the last-blowup-trigger time $S_{0,n}(\sigma)$ as
\begin{align*}
D_n(\sigma) 
&= \int_{S_{0,n}(\sigma)}^\sigma (g_n(\tau) - \delta_{2,n}/\delta_n) \, \mdd \tau \, , \\
&= G_n( \sigma) - G_n(S_{0,n}(\sigma)) - (\delta_{2,n}/\delta_n) (\sigma-S_{0,n}(\sigma)) \, . 
\end{align*}
Since we have $0 \leq S_{0,n}(\sigma) \leq \sigma$, let us consider an extraction $n_k$, $k \in \mathbbm{N}$,
such that $S_{0,n_k}$ converges toward some adherent point $L_0(\sigma) \in [0,\sigma]$.
By compact convergence of $G_n \to G$ and $D_n \to D$, we must have
\begin{align}\label{eq:comp2}
D(\sigma) = G(\sigma)-G(L_0(\sigma)) -   (\sigma-L_0(\sigma))/\lambda_2 \, ,
\end{align}
where we have used the fact that $\delta_{2,n}/\delta_n \to 1/\lambda_2$ as $\delta_n \to 0$.

$(iii)$ Equating \eqref{eq:comp1} and \eqref{eq:comp2}, we obtain that
\begin{align*}
\Psi(\sigma) = (L_0(\sigma)-\lambda_2 G(L_0(\sigma)))/ \nu_2 \, .
\end{align*}
As $\Psi$ and $G$ satisfy \eqref{eq:PsiFromG}, this implies that
\begin{align*}
\frac{1}{\nu_2} \big(L_0(\sigma)-\lambda_2 G(L_0(\sigma))\big) = \frac{1}{\nu_2} \sup_{0 \leq \tau \leq \sigma} \big( \tau-\lambda_2 G(\tau) \big) \, ,
\end{align*}
which shows that $L_0(\sigma)$ must be equal to the last-blowup-trigger time $S_0(\sigma)$.
Therefore
\begin{align*}
D(\sigma) = G(\sigma)-G(S_0(\sigma)) -   (\sigma-S_0(\sigma))/\lambda_2 =  \int_{S_0(\sigma)}^\sigma ( g(\tau) - 1/\lambda_2 ) \, \mdd \tau\, ,
\end{align*}
which concludes the proof.

\end{proof}


\subsection{Delays preclude eternal blowups}\label{sec:delayBlowup}
We conclude by demonstrating that in addition to the advantages discussed in Remark~\ref{rem:globSol}, 
considering refractory delays precludes the occurrence of nonphysical eternal blowups.
Specifically, we show in the following that given the delay lower bound $\epsilon>0$,
choosing $\delta<\epsilon$ guarantees that blowup durations are uniformly bounded,
thus precluding infinite-size (eternal) blowups:

\begin{proposition}
If the buffer parameter $\delta$ is such that $\delta <\epsilon$,  the time-changed, blowup durations satisfy $U_k-S_k \leq \lambda_2+\nu_2\delta$ for all $k \in \mathcal{K}$.
\end{proposition}

\begin{proof} We process in three steps:

$(i)$ In the proof of Proposition~\ref{th:globSol}, we establish that for all $t \geq 0$, $F(t+\epsilon)-F(t) \leq 1$.
To lift this result to the time-changed picture with $\sigma=\Phi(t) \Leftrightarrow t=\Psi(\sigma)$, observe that
\begin{align*}
F(t+\epsilon)-F(t)
&= G(\Phi(t+\epsilon)) - G(\Phi(t)) \, , \\
&= G(\Phi(\Psi(\sigma)+\epsilon)) - G(\sigma) \, , \\
&= G(\sigma+\gamma(\sigma)) - G(\sigma) \, ,
\end{align*}
where we have defined the forward delay function $\gamma(\sigma) = \Phi(\Psi(\sigma)+\epsilon) - \sigma  \geq \nu_2 \epsilon> 0$.
Thus, for all $\sigma \geq 0$, we have $G(\sigma+\gamma(\sigma)) - G(\sigma) \leq 1$.

$(ii)$ Consider a full-blowup interval $(S,U) \subset \mathcal{B}_\sigma$.
Let us show by contradiction that if $\delta <\epsilon$, we must have  $U - S< \epsilon/\delta_2$.
Suppose $U-S \geq \epsilon/\delta_2$. Then, $\Psi(U)-\Psi(S) = \delta_2 (U-S) \geq \epsilon$, so that 
$\gamma(S)=\epsilon/\delta_2$.
In particular, using $(i)$, we have $G(S+\gamma(S))-G(S) \leq 1$, so that
\begin{eqnarray*}
\frac{G(S+\gamma(S))-G(S)}{\gamma(S)} \leq \frac{1}{\gamma(S)} = \frac{\delta_2}{\epsilon} < \frac{\delta_2}{\delta} \, ,
\end{eqnarray*}
where the last strict inequality follows from our assumption that $\delta <\epsilon$.
Since $G$ is $C_1$ and $U$ is defined as
\begin{align}\label{def:U}
U  
&
=  \inf \left\{ \xi > \sigma  \, \bigg \vert \, \int_{S}^\xi (g(\tau) - \delta_2/\delta)\, \mdd \tau \leq 0 \right\}  \, , \nonumber\\
&
= \inf \left\{ \xi > S  \, \bigg \vert \, \frac{G(\xi) - G(S)}{\xi-S} \leq \frac{\delta_2}{\delta}  \right\} \, ,
\end{align}
this implies that $U<S+\gamma(S)=S+\epsilon/\delta_2$.
This contradicts our original assumption that $U -S\geq \epsilon/\delta_2$.

$(iii)$ By $(ii)$, assuming $\delta<\epsilon$ implies that $U-S < \epsilon/\delta_2$, so that 
$\Psi(U)-\Psi(S) = \delta_2 (U-S) \leq \epsilon$, and thus $U-S \leq \gamma(S)$.
Then, using $(i)$, we have $G(U)-G(S) \leq 1$, so that
\begin{eqnarray*}
\frac{1}{U-S} \geq \frac{G(U)-G(S)}{U-S} = \frac{\delta_2}{\delta} = \frac{1}{\lambda_2+\nu_2 \delta} \, ,
\end{eqnarray*}
where the second to last equality follows from the definition of $U$ in \eqref{def:U}.
This shows that if $\delta<\epsilon$, we have $U-S \leq \lambda_2+\nu_2 \delta$.

\end{proof}

\begin{remark}
Eternal blowups will be avoided for all physical limit dynamics as long as these are obtained via 
limiting processes that satisfy the criterion $\delta<\epsilon$.
For instance, eternal blowups do not occur in nonbuffered solutions $(\delta=0)$ with vanishing refractory period $(\epsilon=0)$ 
if these are defined via limiting processes of the form $0 \leq \delta_n<\epsilon_n \downarrow 0$.
\end{remark}

\section{Proofs}\label{sec:proofs}


\subsection{Specification of generic initial conditions}\label{sec:initproofs}

\begin{proof}[Proof of Proposition~\ref{prop:initLink}]

Observe first that \eqref{eq:initLink3} necessarily follows from \eqref{eq:initLink2} as $\sigma \mapsto G_0(\sigma)=-g((\sigma,0))$
is a nondecreasing, \cadlag function on $\mathbbm{R}^-$.
To see why, consider $U_k$, $k \in \mathcal{K}$ defined as in \eqref{eq:initLink2}.
If $U_k$ is a continuity point of $G_0$ \eqref{eq:initLink3} holds.
Suppose then that $U_k$ is discontinuity point.
Then $G_0$ must have a positive jump discontinuity in $U_k$ so that
\begin{eqnarray*}
\lim_{\sigma \downarrow U_k} G_0(\sigma) = G_0(U_k^-) < G_0(U_k) = \lim_{\sigma \uparrow U_k} G_0(\sigma) \, .
\end{eqnarray*}
But, by \eqref{eq:initLink2}, it must be that 
$$G_0(U_k^-)- G_0(S_k)  \geq (U_k - S_k ) / \lambda_2\, , $$
so that we also have 
$$\lim_{\sigma \uparrow U_k} G_0(\sigma)  = G_0(U_k) > G_0(U_k^-) \geq  G_0(S_k) + (U_k - S_k ) / \lambda_2\, , $$
which contradicts the definition of $U_k$ as a first-passage time.
Therefore, $U_k$ must always be a continuity point and \eqref{eq:initLink3} always holds.

Let us then show that given nonexplosive initial conditions $(p_0,f_0)$, the class of measure $(q_0,g_0)$
defined in Proposition~\ref{prop:initLink} specifies generic initial conditions.
This amounts to checking that \eqref{eq:TCinit} in Definition~\ref{def:TCinit} holds. 
Consider the increasing function $\Phi_0: \mathbbm{R}^- \to  \mathbbm{R}^-$, $t \mapsto \Phi_0(t) = \nu_2 t + \lambda_2 F_0(t)$,
and its nondecreasing, continuous inverse $\Psi_0=\Phi_0^{-1}$.
For all $\sigma \leq 0$, plugging $t=\Psi(\sigma)$ in the definition of $\Phi_0$ yields
\begin{eqnarray*}
\Phi_0(\Psi_0(\sigma)) = \nu_2 \Psi_0(\sigma)+\lambda_2 F_0(\Psi_0(\sigma)) \, ,
\end{eqnarray*}
so that by definition of $G_0$, if $\sigma \notin \cup_{k \in \mathcal{K}} [S_k,U_k)$, we have
\begin{eqnarray*}
\Psi_0(\sigma) 
= \frac{1}{\nu_2}  \big( \Phi_0(\Psi_0(\sigma)) - \lambda_2 F_0(\Psi_0(\sigma)) \big) 
= \frac{1}{\nu_2}  \big( \sigma-\lambda_2 F_0(\Psi_0(\sigma)) \big) 
= \frac{1}{\nu_2}  \big( \sigma-\lambda_2 G_0(\sigma)) \big) \, ,
\end{eqnarray*}
where we have used that $\Phi_0$ is one-to-one  $\mathbbm{R}^- \to \mathbbm{R}^- \setminus \cup_{k \in \mathcal{K}} [S_k,U_k)$
and the definition of $G_0$  on $\mathbbm{R}^- \setminus \cup_{k \in \mathcal{K}} [S_k,U_k)$ given in \eqref{eq:initLink1}.

Moreover, consider $l \in \mathcal{K}$.
Since the intervals $[S_k,U_k]$, $k \in \mathcal{K}$, are nonoverlapping,
for all $n>0$, there is $\sigma_n \notin  \cup_{k \in \mathcal{K}} [S_k,U_k]$ such that $0 < S_l-\sigma_n<1/n$.
In particular, by continuity of $\Psi_0$, we have
\begin{eqnarray*}
\Psi_0(S_l) = \lim_{n \to \infty} \Psi_0(\sigma_n) 
= \lim_{n \to \infty}  \frac{1}{\nu_2}  \big( \sigma_n -\lambda_2 F_0(\Psi_0(\sigma_n)) \big)
= \frac{1}{\nu_2}  \big( S_l -\lambda_2 G_0(S_l)) \big) \, ,
\end{eqnarray*}
where we have used the definition of $G_0$ on  $ \{ S_k \}_{k \in \mathcal{K}}$ given in \eqref{eq:initLink1}.

Then, by \eqref{eq:initLink2}, we also have that if there is $k \in \mathcal{K}$ such that $S_k \leq \sigma < U_k$,
we must have
\begin{eqnarray*}
\frac{1}{\nu_2} \big( \sigma - \lambda_2 G_0(\sigma) \big) \leq \frac{1}{\nu_2} \big( S_k - \lambda_2 G_0(S_k) \big) = \Psi_0(S_k) = \Psi_0(\sigma)\, ,
\end{eqnarray*}
where last equality follows from the fact that $\Psi_0$ is flat on $[S_k,U_k]$.
Thus $\Psi_0(\sigma) \geq (\tau - \lambda_2 G_0(\tau) )/ \nu_2$ for all $\sigma \leq 0$,
and since $\Psi_0$ is nondecreasing, this implies that
\begin{eqnarray*}
\Psi_0(\sigma) =\sup_{\tau \leq \sigma} \Psi_0(\tau) \geq \frac{1}{\nu_2}\sup_{\tau \leq \sigma} \big( \tau - \lambda_2 G_0(\tau) \big) \, .
\end{eqnarray*}
At the same time, defining the time
\begin{eqnarray*}
S_0(\sigma) = \sum_{k \in \mathcal{K}} S_k \mathbbm{1}_{\{S_k < \sigma <U_k\}} + \left(1- \sum_{k \in \mathcal{K}} \mathbbm{1}_{\{S_k < \sigma <U_k\}} \right) \sigma \, ,
\end{eqnarray*}
which satisfies $-\infty< S_0(\sigma) \leq \sigma$, we have for all $\sigma \leq 0$
\begin{eqnarray*}
\Psi_0(\sigma)
=
 \big( S_0(\sigma)-\lambda_2 G_0(S_0(\sigma)) \big) /\nu_2 \, .
\end{eqnarray*}


Therefore $g_0$ and $\Psi_0$ defined above satisfies \eqref{eq:TCinitd}  and. \eqref{eq:TCinita} for $G_0(0)=0$, which
directly follows from the fact that by construction $\Psi_0(0)=0$.
The boundedness condition \eqref{eq:TCinitb} directly follows from \eqref{eq:mainProb2b}.
To establish \eqref{eq:TCinitc}, first observe that by integration by parts for Stieltjes integrals, we have
\begin{align*}
\int_{(S_n,0]} &(1-Q_0(\sigma)) \, \mdd G_0(\sigma) \\
&=
\big[(1-Q_0(\sigma))G_0(\sigma) \big]^0_{S_n^-} + \int_{(S_n,0]}   G_0(\sigma) \, \mdd Q_0(\sigma)) \, , \\
&=
G_0(0) - (1-Q_0(S_n^-))G_0(S_n^-) + \int_{(S_n,0]}   G_0(\sigma) \, \mdd Q_0(\sigma)) \, , \\
&=
F_0(0) - (1-P_\epsilon(-T_n))F_0(T_n) + \int_{(S_n,0]}   G_0(\sigma) \, \mdd P_\epsilon(\Psi_0(\sigma)) \, , 
\end{align*}
Remember then that the nonincreasing function $\sigma \mapsto P_\epsilon(-\Psi_0(\sigma))$ is flat on
$\cup_{k \mathcal{K}} [S_k, U_k]$, so that the restriction of $G_0$ on $[S_k, U_k]$ does not contribute 
to the integral term in the last equality above.
In particular, we may write 
\begin{align*}
\int_{(S_n,0]}   G_0(\sigma) \, \mdd P_\epsilon(-\Psi_0(\sigma))
=
\int_{(S_n,0]}   F_0(\Psi_0(\sigma)) \, \mdd P_\epsilon(-\Psi_0(\sigma)) \, ,
\end{align*}
which corresponds to set $g_0$ to its maximum possible value on $[S_k, U_k]$ according to \eqref{eq:initLink2},
i.e., $g_0=1/\lambda_2$.
Since $\Psi_0$ is continuous, the substitution formula is licit~\cite{Falkner:2012} so that we have
\begin{align*}
\int_{(S_n,0]}   G_0(\sigma) \, \mdd P_\epsilon(-\Psi_0(\sigma))
=
\int_{(T_n,0]}   F_0(t) \, \mdd P_\epsilon(-t) \, .
\end{align*}
Moreover, remembering that $T_n=\Psi_0(S_n) \to- \infty$ when $n \to \infty$, we have
\begin{eqnarray*}
\lim_{n \to \infty} (1-P_\epsilon(-T_n))F_0(T_n) = 0 \, ,
\end{eqnarray*}
by the same argument as in the proof of Proposition~\ref{prop:inactProp}.
This shows that
\begin{align*}
\int_{(-\infty,0]} (1-Q_0(\sigma)) \, \mdd G_0(\sigma) 
&= 
\lim_{n \to \infty} \int_{(S_n,0]} (1-Q_0(\sigma)) \, \mdd G_0(\sigma) \, , \\
&= 
F_0(0) - \int_{[0, \infty)}   F_0(-t) \, \mdd P_\epsilon(t) \, , \\
& = 
\int_{[0, \infty)}  (1-P_\epsilon(t)) \, \mdd F_0(-t) \, ,
\end{align*}
so that \eqref{eq:TCinitb} follows from \eqref{eq:mainProb2c}.
This concludes the proof.

\end{proof}


\subsection{Buffer mechanism proofs}\label{sec:buffproofs}

\begin{proof}[Proof of Proposition~\ref{prop:buffExist}]
Pick $T>0$.
If $\{ t  \, \vert \, 0<t<T , \,  z(t)> z_\delta \} = \varnothing$, the $\delta$-buffered problem \eqref{all:bufferMech} clearly admits 
a unique solution on $[0,T)$ given by $E=B=0$ and $\theta = K \circ z$.
Suppose then that $\{ t  \, \vert \, 0< t<T , \,  z(t)> z_\delta \} \neq \varnothing$.
For all $T>0$, consider the decomposition of the open set $\{ t  \, \vert \, 0<t<T , \,  z(t)> z_\delta \} $ as a nonempty countable union of 
nonoverlapping open intervals, which can be thought of blowup seeds.

\emph{(a) Case of a finite number of blowup seeds.} 
If the blowup seeds are finite, a solution to the $\delta$-buffered problem \eqref{all:bufferMech} can
be constructed iteratively as follows. By assumption, there is $K$, $0<K<\infty$, such that
\begin{eqnarray*}
\{ t  \, \vert  \, 0< t<T , \, z(t)> z_\delta \} = \bigcup_{k=1}^K (a_k,b_k) \, ,
\end{eqnarray*}
with $0<a_1<b_1<a_2<b_2< \ldots \leq T$.
Set  $T_1=a_1$ and define the sequence of blowup trigger times $T_l$, $1 \leq l \leq K$, and exit times $V_l$, $1 \leq l \leq K$ ,
via
\begin{eqnarray*}
V_k =  \inf \left\{ t> T_k \, \bigg \vert \, \int_{T_k}^t \big(L(z(s))-1/\delta \big) \, \mdd s < 0 \right\} \quad \mathrm{and} \quad
T_{l+1} = \inf_{1 \leq k \leq K} \left\{ a_k > V_l \right\}  \, ,
\end{eqnarray*}
with the usual convention that $ \inf \varnothing = \infty$.
Next, for all $t$, $0<t<T$, define the last-blowup-trigger time $T_0(t)$ and the last blowup exit time $V_0(t)$ as
\begin{eqnarray*}
T_0(t) = \sup_{1 \leq l \leq K} \left\{ T_l \leq t \right\}  \quad \mathrm{and} \quad V_0(t) = \sup_{1 \leq l \leq K} \left\{ V_l \leq t \right\} \, ,
\end{eqnarray*}
with the  usual convention that $ \sup \varnothing = -\infty$
Specify the excess function as 
\begin{eqnarray*}
E(t)  = \mathbbm{1}_{\{ T_0(t)>V_0(t) \}} \int_{T_0(t)}^{t} (L(z(s)) - 1/\delta) \, \mdd s \, .
\end{eqnarray*}
and set $B$ and $\theta$ according to \eqref{eq:buff1}.
This implies  that for all $t$, $0<t<T$, we have
\begin{eqnarray*}
B(t)=\sum_{l=1}^K \mathbbm{1}_{(T_l \wedge T,V_l \wedge T)}(t) \geq \sum_{k=1}^K \mathbbm{1}_{(a_k,b_k)}(t) \, .
\end{eqnarray*}
Then, $(\theta,B,E)$ solves the $\delta$-buffered problem \eqref{all:bufferMech} on $(0,T]$.
Uniqueness follows from considering the solution sequentially on $[0,T_1]$ and on the 
consecutive nonempty intervals of the form $[T_l, V_l) \cap [0,T)$,  $[V_l, T_{l+1}) \cap [0,T)$, $1 \leq l \leq K$.

\emph{(b) Case of an infinite number of blowup seeds.} 
If the interval decomposition of $\{ t> 0 \, \vert  \,  t<T , \, z(t)> z_\delta \}$ is infinite, we may write it as
\begin{eqnarray*}
\{t  \, \vert  \, 0< t<T , \, z(t)> z_\delta \} = \bigcup_{i \in \mathbb{N}} (A_i,B_i) \, ,
\end{eqnarray*}
where the intervals are listed by decreasing size: $\Delta_0=B_0-A_0 \geq \Delta_1=B_1-A_1 \geq \ldots$.
Clearly, $\sum_{i \in \mathbbm{N}} \Delta_i \leq T$.
For all $n\in \mathbbm{N}$, there is $I_n>0$ such that $\sum_{i > I_n} \Delta_i \leq 1/n$.
Let us then specify $E_n$ as the excess function for the unique solution of the 
$\delta$-buffered problem \eqref{all:bufferMech} for the modified continuous input function
\begin{eqnarray*}
z_n(t) = 
\left\{
\begin{array}{ccc}
 z(t)  & \: \: \mathrm{if} \: \: & t \notin \cup_{i > I_n}(A_i,B_i) \, , \\
1/\delta  & \: \: \mathrm{if}  \: \: &  t \in \cup_{i > I_n}(A_i,B_i)  \, .  \\  
\end{array}
\right.
\end{eqnarray*}
This solution exists and is unique by $(a)$ as $z_n$ is defined as a continuous function with
finite  blowup seed decomposition: 
\begin{eqnarray*}
\{ t \, \vert  \, 0< t<T , \, z_n(t)> z_\delta \} = \bigcup_{i=1}^{I_n} (A_i,B_i) \, .
\end{eqnarray*}

\emph{(i) Existence.} 
By construction, for all $n \in \mathbbm{N}$, $z_{n+1} \geq z_n$.
Consider $E_{n+1}$ and $E_n$ the excess functions associated to the unique solutions
to the $\delta$-buffered problem \eqref{all:bufferMech} on $[0,T)$ for $z_{n+1}$ and $z_n$,
respectively.

Let us show that $E_{n+1} \geq E_{n}$.
Pick $t$, $0<t<T$, if $E_n(t)=0$, there is nothing to show.
If $E_n(t)>0$, consider $T_{n,0}(t)$, the last-blowup-trigger time for $E_n$.
By definition, for all $s$, $T_{n,0}(t) <s \leq t$, we have
\begin{equation*}
0<E_n(s) = \int_{T_{n,0}(s)}^s \big(L(z_n(u)) -1/\delta\big) \, \mdd u \leq  \int_{T_{n,0}(s)}^s \big(L(z_{n+1}(u)) -1/\delta \big) \, \mdd u \, .
\end{equation*}
The above inequality shows that $T_{n+1,0}(s) \leq T_{n,0}(s)$. This implies that
\begin{eqnarray*}
E_{n+1}(s) = E_{n+1}(T_{n,0}(t)) +  \int_{T_{n,0}(t)}^s \big(L(z_{n+1}(u)) -1/\delta\big) \, \mdd u \geq E_{n+1}(T_{n,0}(t)) +  E_n(s) \, ,
\end{eqnarray*}
so that in particular $E_{n+1}(t) \geq E_n(t)$, as $E_{n+1}(T_{n,0}(t)) \geq 0$.
Therefore, the functions $E_n$ specify an increasing sequence of functions on $[0,T)$.
Using $E_n$ to specify $(B_n,\theta_n)$ according to \eqref{eq:buff1}, this implies that 
for all $t$, $0<t<T$, $E_n(t)$, $B_n(t)$, and $\theta_n(t)$ are pointwise increasing, 
upper bounded, and thus converging, sequences toward (lower semi-continuous) limit functions $E$, $B$, $\theta$, respectively.
Moreover, for all $T>0$, by continuity of $z$ on $[0,T]$, $E_n$, $B_n$, and $\theta_n$ are uniformly upper bounded by 
integrable functions on $[0,T]$, whereas $z_n$ uniformly converges toward $z$ on $[0,T]$.
By the monotone convergence theorem, \eqref{eq:buff3} is satisfied by  $(B,\theta,E)$, and $E$ is necessarily continuous.
This shows that  $(B,\theta,E)$ solves the $\delta$-buffer problem on $[0,T)$, for arbitrary $T>0$, and thus on $\mathbbm{R}^+$.

\emph{(ii) Uniqueness.} 
Suppose that $\theta$ solves the $\delta$-buffer problem \eqref{all:bufferMech} for a continuous function
 $z$ with  blowup indicator function $B$ and excess function $E$.
Consider the modified input function $z_n$ as defined in $(i)$, whose $\delta$-buffer problem admits
a unique solution as specified in $(a)$, which we denote by $\theta_n$, with blowup indicator 
and excess function denoted by $B_n$ and $E_n$, respectively.
Clearly, we have $z \geq z_n$, which implies that $E \geq E_n$, where $E$ and $E_n$ are the excess functions
associated to $z$ and $z_n$. This follows from the exact same argument as in $(i)$, 
where we show that $z_{n+1} \geq z_n$ implies $E_{n+1} \geq E_n$, 

Next, for all $t$, $0<t<T$, consider 
\begin{eqnarray*}
T_0(t) = \sup \{ 0<s \leq t \, \vert \, E(s) = 0 \} \quad \mathrm{and} \quad T_{n,0}(t) = \sup \{ 0<s \leq t \, \vert \, E_{n,0}(s) = 0 \} \, .
\end{eqnarray*}
As $E \geq E_n$, we have $0=E(T_0(t)) \geq E_n(T_0(t)) \geq 0$ so that $E_n(T_0(t))=0$, and necessarily $T_0(t) \leq T_{n,0}(t)$.
Moreover, we have
\begin{eqnarray*}
E(t) - E_n(t)
&=&
\int_{T_0(t)}^t \big(L(z(s)) -1/\delta \big) \, \mdd s - \int_{T_{n,0}(t)}^t \big(L(z_n(s)) -1/\delta \big) \, \mdd s \, , \\
&=&
\int_{T_0(t)}^t \big(L(z(s)) -L(z_n(s)) \big) \, \mdd s + \int_{T_0(t)}^{T_{n,0}(t)} \big(L(z_n(s)) -1/\delta \big) \, \mdd s \, .
\end{eqnarray*}
Pick $\epsilon>0$.
As $L$ is continuous, there is $d_1>0$ such that $\vert z-z_\delta \vert <d_1$ implies $\vert L(z)-1/\delta \vert \leq \epsilon/(2T)$.
As $z$ is continuous, it is also uniformly continuous on $[0,T]$, so that there is $d_2>0$ such that
$\vert t-s \vert < d_2$ implies $\vert z(t) - z(s) \vert< d_1$.
But then by construction, for all $n>N_1 = 1/d_2$, we have for all $t$, $t<0<T$,
\begin{eqnarray*}
\vert z(t) - z_n(t) \vert \leq  \sup_{i \geq I_n} \left( \sup_{t,s \in (A_i,B_i)} \vert z(t) - z(s)  \vert \right) \leq  d_1 \, ,
\end{eqnarray*}
as $\sum_{i>I_n} \Delta_i \leq 1/n < d_2$ where $\Delta_i=B_i-A_i$.
This shows that  for all $n>N_1$
\begin{eqnarray*}
\int_{T_0(t)}^t \big(L(z(s)) -L(z_n(s)) \big) \, \mdd s \leq \epsilon\, .
\end{eqnarray*}

Then, observe that since $E_n(T_0(t))=0$, we have
\begin{eqnarray*}
\int_{T_0(t)}^{T_{n,0}(t)} B_n(s) \big(L(z_n(s)) -1/\delta \big) \, \mdd s
= E_n(T_{n,0}(t)) - E_n(T_0(t)) = 0 \, .
\end{eqnarray*}
Therefore, we can write
\begin{align*}
\int_{T_0(t)}^{T_{n,0}(t)}  & \big(L(z_n(s)) -1/\delta \big) \, \mdd s \\
&=
\int_{T_0(t)}^{T_{n,0}(t)} (1-B_n(s)) \big(L(z_n(s)) -1/\delta \big) \, \mdd s \, , \\
&\leq
\int_{T_0(t)}^{T_{n,0}(t)} \left(1- \sum_{i=1}^{I_n} \mathbbm{1}_{(A_i,B_i)}(s)\right) \big(L(z_n(s)) -1/\delta \big) \, \mdd s \, , \\
&\leq
\int_{T_0(t)}^{T_{n,0}(t)} \left( \sum_{i \geq 1} \mathbbm{1}_{(A_i,B_i)}(s) - \sum_{i=1}^{I_n} \mathbbm{1}_{(A_i,B_i)}(s))\right) \big(L(z_n(s)) -1/\delta \big) \, \mdd s \, , 
\end{align*}
where the first inequality follows from the fact that by construction, we have 
\begin{eqnarray*}
B_n = \sum_{i=1}^{I_n} \mathbbm{1}_{(T_i \wedge T,V_i \wedge T)} \geq \sum_{i=1}^{I_n} \mathbbm{1}_{(A_i,B_i)} \, ,
\end{eqnarray*}
and where the last inequality follows from the fact that
\begin{eqnarray*}
L(z_n(s)) -1/\delta &\leq& \mathbbm{1}_{\{L(z_n(s)) >1/\delta\}} (L(z_n(s)) -1/\delta) \, , \\
&=& \mathbbm{1}_{\{z_n>z_\delta\}} (L(z_n(s)) -1/\delta) \, ,  \\
&=&  \sum_{i \geq 1} \mathbbm{1}_{(A_i,B_i)}(s)  (L(z_n(s)) -1/\delta) \, .
\end{eqnarray*}
This implies that
\begin{eqnarray*}
\int_{T_0(t)}^{T_{n,0}(t)} \big(L(z_n(s)) -1/\delta \big) \, \mdd s
&\leq&
 \sum_{i > I_n} \int_{T_0(t)}^{T_{n,0}(t)}  \mathbbm{1}_{(A_i,B_i)}(s)  \big(L(z_n(s)) -1/\delta \big) \, \mdd s \, ,  \\
&\leq&  L(\Vert z \Vert_{[0,T]})  \sum_{i > I_n} \Delta_i < L(\Vert z \Vert_{[0,T]})/n \, .
\end{eqnarray*}
Altogether, setting $N_2=(\epsilon/2L(\Vert z \Vert_{[0,T]})$, for all $n>\max(N_1,N_2)$, we have
\begin{eqnarray*}
0 \leq E(t) - E_n(t)  \leq \epsilon \, ,
\end{eqnarray*}
showing that  $E =\lim_{n \to \infty} E_n$. This shows that there is a unique excess function $E$, and therefore
a unique solution $\theta$  to the $\delta$-buffer problem \eqref{all:bufferMech}.

\end{proof}


\begin{proof}[Proof of Proposition~\ref{prop:buffReg}]
We proceed in three steps:

$(i)$ \emph{Candidate set $\mathcal{D}$ for discontinuity times}.  
Observe that by definition, $L(z)=K(z)$ when $z=z_\delta$ so that the function 
\begin{eqnarray*}
(z,E) \in \mathbbm{R} \times \mathbbm{R} \mapsto \theta(z,E) = B(z,E) L(z) + (1-B(z,E)) K(z)
\end{eqnarray*}
with $B(z,E)=\mathbbm{1}_{\{z>z_\delta \} \cup \{E>0\}}$ is continuous everywhere except on $\{(z,E) \, \vert  \, E=0  , \,  z<z_\delta\}$.
Since $E$ is necessarily continuous, this shows that $\theta(t)=B(z(t),E(t))$ is continuous everywhere except possibly on the set
\begin{eqnarray*}
\mathcal{D} = \partial \{ t > 0 \, \vert \, E(t)=0 \} \cap \{ z(t) < z_\delta \}  = \partial \{ t > 0 \, \vert \, E(t)=0 \} \cap \{ \theta(t) < 1/\delta \} \, ,
\end{eqnarray*}
where $\partial \{ t > 0 \, \vert \, E(t)=0 \}$ is the boundary of the set $\{ t \geq 0 \, \vert \, E(t)=0 \}$ and where the equality follows from Lemma~\ref{lem:buffcond}.
Observe that by continuity of $E$, we must have  $\mathcal{D} \subset  \{ t > 0 \, \vert \, E(t)=0 , \, \theta(t) < 1/\delta \}$.

$(ii)$ \emph{Right-continuity of $\theta$.} Pick $t \in \mathcal{D}$.
As the set $\{ t \geq 0 \, \vert \, \theta(t) < 1/\delta \} = \{ t \geq 0 \, \vert \,  z(t)<z_\delta \} $ is open by continuity of $z$,  there is $t'>t$ such that $\theta< 1/\delta$ on $[t,t')$, which implies that $B(s)=\mathbbm{1}_{\{ E(s) > 0\}}$ for all $t \leq s <t'$.
But then, by \eqref{eq:buff3} with $E(t)=0$, we must have that $E=0$ on the whole interval $[t,t')$, so that $\theta=K \circ z$ on $[t,t')$.
In particular, $\theta$ is right-continuous in $t$. This implies that the set $\mathcal{D}$ is at most countable.
Furthermore, we have
\begin{eqnarray*}
\theta(t)= K(z(t)) < 1/\delta \quad \mathrm{with} \quad z(t)<z_\delta \, ,
\end{eqnarray*}
and because $t  \in \partial \{ t > 0 \, \vert \, E(t)=0 \}$, there exists a sequence $t_n \to t$ with $t_n<t$ such that $E(t_n)>0$.
Any such sequence satisfies
\begin{eqnarray*}
\theta(t_n)= L(z(t_n)) \xrightarrow[n \to \infty]{}  L(z(t)) > K(z(t)) \, .
\end{eqnarray*}

$(iii)$ \emph{Left limits of $\theta$.} To prove that $\theta$ is left continuous in all $t \in \mathcal{D}$, observe that either $(a)$ there is a left open interval $(t',t)$, $t'<t$, over which $E(t)>0$, $(b)$ there is a left open interval $(t',t)$, $t'<t$, over which $E(t)=0$ or $(c)$ there are sequences $t'_n \to t$ and $t_n \to t$, with $t_n'<t$, and such that $E(t'_n)=0$ and $E(t_n)>0$.
If $(a)$ holds, for all $t'<s<t$, we have $\theta= L \circ z$ so that left continuity follows from the continuity of $z$.
A similar result follows from $\theta= K \circ z$ if $(b)$ holds.
Suppose that $(c)$ holds, then $z(T_0(t_n))= z_\delta$ and $T_0(t_n)\to t$ therefore by continuity of  $z$, we have $z(t)=z_\delta$,
 but this contradicts the fact that $t\in \mathcal{D}$, so $\theta$ must be left continuous.

 \end{proof}



\subsection{Proofs about the running-supremum function}\label{sec:RSfunction}

\begin{proof}[Proof of Corollary \ref{cor:PsiGamma}]
Since $\tT\leq \Theta$, their inverse functions $\tGamma=\tT^{-1}$ and $\Gamma=\Theta^{-1}$ are such that $\Gamma \leq \tGamma$.
Moreover, we have
$\tT\ (v)-\tT\ (t) \leq (v-t)/\delta$,
so that for all $\sigma=\Phi (t)>0$, $\tau=\Phi (t)>0$, $\tGamma(\sigma) -\tGamma(\tau) \geq \delta (\sigma - \tau)$.
Therefore $\sigma \mapsto \tGamma (\sigma) - \delta \sigma$ is a nondecreasing function $\mathbbm{R}^+ \to \mathbbm{R}^+$.
This implies that for all $\sigma>0$, we have
\begin{eqnarray*}
\tGamma(\sigma) - \delta \sigma 
=
 \sup_{0 \leq \tau \leq \sigma} ( \tGamma (\tau) - \delta \tau) 
 \geq 
 \sup_{0 \leq \tau \leq \sigma} ( \Gamma(\tau) - \delta \tau) \, , 
\end{eqnarray*}
where the inequality follows from the fact that $\Gamma \leq \tGamma$.
To conclude, let us show the opposite equality.
If $\tGamma(\sigma) = \Gamma(\sigma)$, there is nothing to show.
Suppose then that $\tGamma(\sigma) > \Gamma(\sigma)$, which means that 
$E(\tGamma(\sigma))>0$, so that $t=\tGamma(\sigma)$ is in a full-blowup interval.
Consider the last-blowup-trigger time $T_0(t)$ and define $S_0(\sigma)=\tT(T_0(t))=\tT(T_0(\tGamma(\sigma)))$.
Then, \eqref{eq:ThetaT}, which always holds if $\Theta(0)=E(0)=0$, reads 
\begin{eqnarray*}
\Theta(\tGamma(\sigma)) - \Theta(\tGamma(S_0(\sigma)))   = E(\tGamma(\sigma)) + (\tGamma(\sigma)-\tGamma(S_0(\sigma)))/\delta \, .
\end{eqnarray*}
Remembering that $\Theta=\tT+E=\tGamma^{-1}+E$, this means that
\begin{eqnarray*}
\sigma - \Theta(\tGamma(S_0(\sigma)))   =  (\tGamma(\sigma)-\tGamma(S_0(\sigma)))/\delta \, .
\end{eqnarray*}
Using the fact that $\tGamma(S_0(\sigma))=\Gamma(S_0(\sigma))=\Theta^{-1}(S_0(\sigma))$, this further simplifies to
\begin{eqnarray*}
\sigma  - S_0(\sigma)   =  (\tGamma(\sigma)-\Gamma(S_0(\sigma)))/\delta \, .
\end{eqnarray*}
But then, since $S_0(\sigma) \leq \sigma$, we have
\begin{eqnarray*}
\tGamma(\sigma) - \delta \sigma 
=
\Gamma(S_0(\sigma)) - \delta S_0(\sigma)
\leq 
\sup_{0 \leq \tau \leq \sigma} ( \Gamma(\tau) - \delta \tau)  \, ,
\end{eqnarray*}
which is the desired inequality and concludes the proof.
\end{proof}

\begin{proof}[Proof of Lemma~\ref{lem:Sup}]
It is clear from \eqref{eq:supGamma} that $\sigma \mapsto \tGamma(\sigma) \geq \Gamma$ is a nondecreasing function
that inherits the continuity of $\Gamma$ via the running $\sup$ operator, which may be replaced by a $\max$ for continuous $\Gamma$
operator.

$(i)$ Let us show that  $\tGamma'(\sigma)=0$ for all $\sigma > 0$ such that $\Gamma'(\sigma)< 0$ or $\Gamma(\sigma)<\tGamma(\sigma)$.

If $\Gamma'(\sigma)<0$, there is $d>0$, such that $\Gamma$ is decreasing on $(\sigma-d,\sigma+d)$,
so that its maximum value over $[0,\sigma]$ must be attained for some $S\leq \sigma-d$.
This means that $\tGamma(\sigma)= \Gamma(S)$ is constant on $(\sigma-d,\sigma+d)$.
In particular $\tGamma'(\sigma)=0$.

Similarly, if $\tGamma(\sigma)>\Gamma(\sigma)$, there is $d>0$, such that $\Gamma<\tGamma$ on $(\sigma-d,\sigma+d)$,
so that the maximum value of $\tGamma$ over $[0,\sigma]$ must be attained for some $S\leq \sigma-d$.
This again implies that $\tGamma'(\sigma)=0$.

$(ii)$
Consider then $\sigma>0$ such that $\Gamma'(\sigma)=0$ and $\Gamma(\sigma)=\tGamma(\sigma)$.
Let us introduce for all $\tau>0$
\begin{align*}
S(\tau) = \sup \{ \xi \geq 0 \, \vert \, \Gamma(\xi) = \tGamma(\tau)  \}  = \sup \{ \xi \geq 0 \, \vert \, \Gamma(\xi) = \max_{0 \leq \xi \leq \tau}  \Gamma(\xi)  \} \leq \tau \, 
\end{align*}
which allows one to write for all $\tau \geq 0$, $\tGamma(\tau)=  \Gamma(S(\tau))$.
Observe  that $S(\sigma)=\sigma$ since we assume $\Gamma(\sigma)=\tGamma(\sigma)$.

If there is $\tau>\sigma$ such that $S(\sigma)=S(\tau)$, then for all $\xi$, $\sigma \leq \xi \leq \tau$, $\tGamma(\xi)=\tGamma(\sigma)$.
Otherwise, for all $\tau>\sigma$ we necessarily have $S(\sigma)<S(\tau)$  and 
\begin{eqnarray*}
  \Gamma(S(\tau))  = \tGamma(\tau) = \max_{0 \leq \xi \leq \tau} \Gamma(\xi)> \max_{0 \leq \xi \leq \sigma}  \Gamma(\xi) =  \tGamma(\sigma) = \tGamma(S(\sigma)) \, .
\end{eqnarray*}
Moreover, we have
\begin{eqnarray*}
\frac{\tGamma(\tau) - \tGamma(\sigma)}{\tau - \sigma}
&=&
\frac{\Gamma(S(\tau)) - \Gamma(S(\sigma))}{\tau - \sigma} \, , \\
&=&
 \frac{\Gamma(S(\tau)) - \Gamma(S(\sigma))}{S(\tau) - S(\sigma)}  \left(  \frac{S(\tau)-S(\sigma)}{\tau - \sigma} \right)\, . \\
&\leq&
 \frac{\Gamma(S(\tau)) - \Gamma(\sigma)}{S(\tau) - \sigma}  \, . 
\end{eqnarray*}
where the last inequality follows from the facts that $\Gamma(S(\tau)) \geq \Gamma(S(\sigma))$, $S(\sigma)=\sigma$ and $\sigma < S(\tau) \leq \tau$.
Similar arguments hold under the assumption that $\tau<\sigma$, allowing one to write
\begin{eqnarray*}
0 \leq w_{\tGamma} (\sigma,\tau) \leq \frac{\Gamma(S(\tau)) - \Gamma(\sigma)}{S(\tau) - \sigma} \xrightarrow[\tau \to \sigma]{} \Gamma'(\sigma) = 0\, .
\end{eqnarray*}
This shows that $\tGamma'(\sigma)=0$ if $\Gamma'(\sigma)=0$ and $\Gamma(\sigma)=\tGamma(\sigma)$.

$(iii)$
Let us finally assume that $\Gamma'(\sigma) > 0$ and $\Gamma(\sigma)=\tGamma(\sigma)$, so that there is $d_1>0$ such that $\Gamma$ is increasing on $(\sigma-d_1,\sigma+d_1)$.
Let us introduce
\begin{align*}
s(\tau) = \inf \{ \xi \geq 0 \, \vert \, \Gamma(\xi) = \tGamma(\tau)  \}  = \sup \{ \xi \geq 0 \, \vert \, \Gamma(\xi) = \max_{0 \leq \xi \leq \tau}  \Gamma(\xi)  \} \leq S(\tau) \leq \tau \, .
\end{align*}

Suppose $s(\sigma)=\sigma$.
Then $\Gamma(\sigma-d_1) \leq \tGamma(\sigma-d_1)<\tGamma(\sigma)=\Gamma(\sigma)$.
Thus, since $\Gamma$ is continuous increasing on $[\sigma-d_1,\sigma]$, there exists a unique $d_2$, $0<d_2 \leq d_1$ such that $\Gamma(s-d_2)=\tGamma(\sigma-d_1)$.
We necessarily have that $\tGamma=\Gamma$ on the interval $(\sigma-d_2,\sigma+d_1)$, so that in particular $\tGamma'(\sigma)=\Gamma'(\sigma)$.
This shows that a non-differentiable point $\sigma$ must be such that $\Gamma'(\sigma) > 0$, $\Gamma(\sigma)=\tGamma(\sigma)$ and $s(\sigma)<\sigma$.
But then $\tGamma$ is constant on $(s(\sigma),\sigma]$ and $\tGamma_-'(\sigma)=0$, whereas $\tGamma=\Gamma$ on the interval $[\sigma,\sigma+d_1)$ 
so that $\tGamma_+'(\sigma)=\Gamma'(\sigma)$.
There can only be a countable number of such points for being defined as endpoints of nonoverlapping, nonempty intervals of the form $(s(\sigma),\sigma)$.

This concludes the proof.
\end{proof}


\subsection{First-passage density bounds}\label{sec:boundh}

\begin{proposition}\label{prop:boundh}
\begin{subequations}\label{eq:allboundh}
For fixed $x>0$ and all $\Psi \in \mId$,  the density $\sigma \mapsto h(\sigma,0,x)$ of the the first-passage time $\sigma_{0,x}$, $x>0$, satisfies for all $\sigma, \tau$, $0 \leq \tau \leq \sigma$
\begin{align}\label{eq:boundh1}
h(\sigma,\tau,x)  &\leq  \left( \frac{x}{\sqrt{(\sigma-\tau)^3}} +  \frac{\mu_M-\mu_m}{\sqrt{\sigma-\tau}} \right)\exp\left(- \frac{\left(x-\int_\tau^\sigma \mu(\xi) \mdd \xi \right)^2}{2(\sigma-\tau)} \right) \, .
\end{align}
where $\mu_m =(\lambda_1/\lambda_2) \wedge (\nu_1/\nu_2)$ and  $\mu_M =(\lambda_1/\lambda_2) \vee (\nu_1/\nu_2)$.
Moreover, for fixed $x>0$, there is a constant $\sigma^\star(x)>0$ that is independent of $\Psi \in \mId$, such that for all $\sigma$, $0\leq\sigma \leq \sigma^\star(x)$, we have
\begin{align}\label{eq:boundh2}
\Vert h(\cdot ,\tau,x) \Vert_{\tau, \sigma}  
\leq  
\left( \frac{x}{\sqrt{(\sigma-\tau)^3}} +  \frac{\mu_M-\mu_m}{\sqrt{\sigma-\tau}} \right)e^{- \frac{\left(x-\mu_M \right)^2}{2(\sigma-\tau)} }  \xrightarrow[\sigma \downarrow 0]{}  0 \, , 
\end{align}
uniformly in $\Psi \in \mId$.
\end{subequations}
\end{proposition}

We prove the above results by following analytical methods developed in \cite{Peskir:2002aa}.

\begin{proof}
	It is enough to shoe the result for $\tau=0$.
Defining the cumulative function 
$$C(y)=\int_y^\infty e^{-x^2/2}/\sqrt{2 \pi} \, \mdd x \, ,$$
the density $\sigma \mapsto h(\sigma,0,x)$ of the the first-passage time $\sigma_{0,x}$, $x>0$
satisfies the integral equation 
\begin{eqnarray}\label{eq:master}
C\left( \frac{z}{\sqrt{\sigma}} \right)=\int_0^\sigma C\left( \frac{z-x+M(\tau)}{\sqrt{\sigma-\tau}}\right)h(\tau,0,x) \, \mdd \tau \, ,
\end{eqnarray}
where $M(\tau)=\int_0^\tau \mu(\xi) \, \mdd \xi$.

On one hand, differentiating with respect to $z>x-\int_0^\sigma\mu(\xi) \mdd \xi $ and letting $z\downarrow x-\int_0^\sigma\mu(\xi) \, \mdd\xi$, we obtain
\begin{equation}\label{vcgg}
\frac{1}{\sqrt{\sigma}}\exp\left(- \frac{\left(x-M(\sigma)  \right)^2}{2\sigma} \right)
=\int_0^\sigma \frac{1}{\sqrt{\sigma-\tau}}\exp\left(
-\frac{\left(M(\sigma)-M(\tau)\right)^2}{2(\sigma-\tau)}
\right) h(\tau,0,x) \, \mdd\tau \, .
\end{equation}

On the other hand, letting $z\downarrow x-\int_0^\sigma\mu(u)du$ in \eqref{eq:master} first and then differentiating with respect to $\sigma$ 
yields 
\begin{align*}
\frac{\mdd}{\mdd \sigma} \left[ C\left( \frac{x-M(\sigma)}{\sqrt{\sigma}} \right) \right]= \frac{1}{2} h(\sigma,0,x)+ \int_0^\sigma \frac{\mdd}{\mdd \sigma}  \left[ C\left( \frac{M(\sigma)-M(\tau)}{\sqrt{\sigma-\tau}}\right)\right] h(\tau,0,x)  \, \mdd \tau \, . 
\end{align*}
Evaluating the derivatives in the expression above and using \eqref{vcgg} to simplify the resulting expression (see~\cite{Peskir:2002aa}), we obtain
\begin{align}\label{vcgg2}
h(\sigma,0,x)=& \frac{x-M(\sigma)}{\sqrt{\sigma^3}}\exp\left(- \frac{\left(x-M(\sigma)\right)^2}{2\sigma} \right) \nonumber\\
& \hspace{20pt}+\int_0^\sigma \frac{M(\sigma)-M(\tau)}{\sqrt{(\sigma-\tau)^3}}\exp\left(
-\frac{\left(M(\sigma)-M(\tau)\right)^2}{2(\sigma-\tau)}
\right) h(\tau,0,x) \, \mdd \tau \, .
\end{align}
For all $\Psi \in \mId$, we have 
\begin{eqnarray*}
\mu_m=\frac{\lambda_1}{\lambda_2} \wedge \frac{\nu_1}{\nu_2} 
\leq
\frac{\int_\tau^\sigma \mu(\xi) \, \mdd \xi}{\sigma-s} = \frac{M(\sigma)-M(\tau)}{\sigma-\tau} 
\leq 
\frac{\lambda_1}{\lambda_2} \vee \frac{\nu_1}{\nu_2} = \mu_M \, .
\end{eqnarray*}

Therefore, we have 
\begin{align*}
\int_0^\sigma \frac{M(\sigma)-M(\tau)}{\sqrt{(\sigma-\tau)^3}} & \exp\left(
-\frac{\left(M(\sigma)-M(\tau) \right)^2}{2(\sigma-\tau)}
\right) h(\tau,0,x) \, \mdd \tau \, , \\
& \leq \mu_M \int_0^\sigma \frac{1}{\sqrt{\sigma-\tau}}\exp\left(
-\frac{\left(M(\sigma)-M(\tau)\right)^2}{2(\sigma-\tau)}
\right) h(\tau,0,x) \, \mdd \tau \, , \\
& = \frac{\mu_M }{\sqrt{\sigma}}\exp\left(- \frac{(x-M(\sigma))^2}{2\sigma} \right) \, ,
\end{align*}
where the last equality follows from  \eqref{vcgg}.
Thus, using \eqref{vcgg2} allows us to give the following upper bound for $h(\sigma,0,x)$:
\begin{align*}
h(\sigma,0,x)
&\leq \left( \frac{x-M(\sigma)}{\sqrt{\sigma^3}} +  \frac{\mu_M}{\sqrt{\sigma}} \right) \exp\left(- \frac{(x-M(\sigma))^2}{2\sigma} \right) \, , \\
&\leq  \left( \frac{x}{\sqrt{\sigma^3}} +  \frac{\mu_M-\mu_m}{\sqrt{\sigma}} \right)\exp\left(- \frac{\left(x-\int_0^\sigma \mu(\xi) \mdd \xi \right)^2}{2\sigma} \right) \, .
\end{align*}
This shows \eqref{eq:boundh1}.

Moreover, for $\sigma\leq x/\mu_M$ we have 
\begin{align*}
h(\sigma,0,x) \leq  h_M(\sigma,x) \, ,
\end{align*}
where
\begin{align*}
h_M(\sigma) &= \left( \frac{x}{\sqrt{\sigma^3}} +  \frac{\mu_M-\mu_m}{\sqrt{\sigma}} \right)\exp\left(- \frac{(x-\mu_M \sigma)^2}{2\sigma} \right) \,.
\end{align*}
One can check that for fixed $x>0$, $h_M$ is nonnegative, smooth, unimodal function on 
$\mathbbm{R}^+$ with $h_M(0)=0$ and $\lim_{\sigma \to \infty} h_M(\sigma) =0$.
In particular for all $x$, there is $\sigma_M^\star>0$, such that
\begin{eqnarray}
h_M(\sigma_M^\star) = \Vert h_M  \Vert_{0, \infty} \, .
\end{eqnarray}
Setting $\sigma^\star= x / \mu_M \wedge \sigma_M^\star>0$, this implies 
that for all $\sigma$, $0 \leq \sigma \leq \sigma^\star$, we have
\begin{eqnarray*}
\Vert h(\cdot, 0 , x) \Vert_{0,\sigma} \leq h_M(\sigma)  \xrightarrow[\sigma \to 0]{}  0 \, .
\end{eqnarray*}
This shows \eqref{eq:boundh2}.

\end{proof}


\subsection{Contraction argument proofs}\label{sec:buffcontract}

\begin{lemma}\label{ineqgrownwall}
    Consider a continuous function $f:[0,a]\rightarrow \mathbbm{R}^+$ satisfying the integral inequality
\begin{align*}
    f(x)\leq \frac{M}{\sqrt{x}}+C\int_0^x \frac{f(t)}{\sqrt{x-t}} \mdd t  \, , \quad  0 \leq x \leq a \, ,
\end{align*}
    for some nonnegative constants $M,C>0$. Then, the following inequality holds 
\begin{align*}
    f(x)\leq M(1/\sqrt{x}+\pi C)(2\pi C^2xe^{\pi C^2x}+1) \,  , \quad  0 \leq x \leq a \, .
\end{align*}
\end{lemma}

\begin{proof}
    Applying the integral inequality to $f(t)$ inside the integral we get
\begin{align*}
      f(x)&\leq \frac{M}{\sqrt{x}}+C\int_0^x\frac{1}{\sqrt{x-t}} \left(\frac{M}{\sqrt{t}}+C\int_0^t \frac{f(u)}{\sqrt{t-u}} \, \mdd u \right) \mdd t \, ,\\
      &=\frac{M}{\sqrt{x}}+CM\int_0^x\frac{1}{\sqrt{t}\sqrt{x-t}}dt+C^2\int_0^x\int_0^t \frac{f(u)}{\sqrt{x-t}\sqrt{t-u}} \mdd u\, \mdd t \, .
\end{align*}
Performing the change of variable $u=\sqrt{t/x}$ in the single integral yields 
   $$
   \int_0^x\frac{1}{\sqrt{t}\sqrt{x-t}} \, \mdd t=\int_0^1\frac{2}{\sqrt{1-u^2}} \, \mdd u=\pi \, .
   $$
Changing the order of integration in the double integral yields
\begin{align*}
    \int_0^x\int_0^t \frac{f(u)}{\sqrt{x-t}\sqrt{t-u}} \mdd u\, \mdd t 
    &=
\int_0^x f(u)\int_u^x \frac{1}{\sqrt{x-t}\sqrt{t-u}} \mdd t\,\mdd u \, ,\\
&=\int_0^x f(u)\int_0^{x-u} \frac{1}{\sqrt{x-u-y}\sqrt{y}} \, \mdd y\, \mdd u \, , \\
&=\pi\int_0^x f(u)\, \mdd u
\end{align*}
Therefore, we get
\begin{eqnarray}\label{vcghd}
    f(x)\leq\frac{M}{\sqrt{x}}+\pi CM + \pi C^2 \int_0^x f(t) \, \mdd t \, .
\end{eqnarray}
Then, applying the differential form of Gronwall's inequality to $x \mapsto \int_0^x f(t)dt$, we have
\begin{align*}
    \int_0^x f(t)dt&\leq e^{\pi C^2x}\int_0^x \left(\frac{M}{\sqrt{t}}+\pi CM\right) \, \mdd t \, ,
    \\
    &=e^{\pi C^2x}(2M\sqrt{x}+\pi CMx) \, , \\
    &\leq2xe^{\pi C^2x}(M/\sqrt{x}+\pi CM) \, .
\end{align*}
Therefore, substituting in \eqref{vcghd}, we  obtain the desired inequality:
\begin{align*}
    f(x)&\leq (M/\sqrt{x}+\pi CM)(2\pi C^2xe^{\pi C^2x}+1) \, .
\end{align*}

\end{proof}
\begin{proposition}
Consider the cumulative distribution $H(\sigma,t,x)=P[\sigma_{x,t}\leq \sigma]$ of the first passage-time 
$\sigma_{\tau,x}=\inf \left\{ \sigma >\tau \, \big \vert \, W_\tau=x, \, W_\sigma <  \int_\tau^\sigma \mu(\xi) \, \mdd \xi \right\} $, 
 where $W$ denotes the standard Brownian motion and where $\mu$ is the drift term defined in \eqref{eq:TCdrift}.
 Then:
 
 $(i)$ if $\sigma-t<1/(8\mu_M^2)$  with  $\mu_M =(\lambda_1/\lambda_2) \vee (\nu_1/\nu_2)$, for all $0\leq t<\sigma$ and $x,y>0$, we have 
$$
|H(t,\sigma,x)-H(t,\sigma,y)|\leq\frac{4|x-y|}{\sqrt{\sigma-t}} \, .
$$

$(ii)$ given $c >0$, if $y>x>2\sqrt{c}(\sigma-t)^{1/4}+\mu_M(\sigma-t)$, then 
 $$|H(t,\sigma,x)-H(t,\sigma,y)|\leq \frac{|x-y|}{c}\, . $$
\end{proposition}

\begin{proof}
 Denoting the cumulative distribution of $\sigma_{x,t}$ by $s \mapsto H(s,t,x)$ and 
setting $z=x-\int_t^\sigma \mu(u)\, \mdd u$ in \eqref{eq:master}, we get 
\begin{eqnarray*}
C\left( \frac{x-\int_t^\sigma \mu(u) \, \mdd u}{\sqrt{\sigma-t}} \right)=\int_t^\sigma C\left( \frac{-\int_s^\sigma\mu(u) \, \mdd u}{\sqrt{\sigma-s}}\right)  H( \mdd s,t,x) \, .
\end{eqnarray*}
Since $H(t,t,x)=0$ and $C(0)=1/2$, integration by parts  yields
\begin{eqnarray*}
\frac{1}{2}H(t,\sigma,x)=
C\left( \frac{x-\int_t^\sigma \mu(u)\, \mdd u}{\sqrt{\sigma-t}} \right)+\int_t^\sigma H(t,s,x) \, \partial_s\left[C\left( \frac{-\int_s^\sigma\mu(u)\, \mdd u}{\sqrt{\sigma-s}}\right) \right] \, \mdd s \, .
\end{eqnarray*}
We notice that the differential term in the right-hand-side integral does not depend on $x$. 
Therefore, introducing $y>0$, we have
\begin{align}\label{mastercol}
    \frac{1}{2}|H(t,\sigma,x)-H(t,\sigma,y)|&\leq  \left|C\left( \frac{x-\int_t^\sigma \mu(u)\, \mdd u}{\sqrt{\sigma-t}} \right)-C\left( \frac{y-\int_t^\sigma \mu(u)\, \mdd u}{\sqrt{\sigma-t}} \right)\right| \nonumber \\
    &+\int_t^\sigma|H(t,s,x)-H(t,s,y)|\left| \partial_s \left[C\left( \frac{-\int_s^\sigma\mu(u)du}{\sqrt{\sigma-s}}\right) \right]\right| \, \mdd s \, .
\end{align}
We want to bound the two right-hand side terms. 
We use the mean value theorem to bound first term as
\begin{align*}
    \left|C\left( \frac{x-\int_t^\sigma \mu(u) \, \mdd u}{\sqrt{\sigma-t}} \right)-C\left( \frac{y-\int_t^\sigma \mu(u)\, \mdd u}{\sqrt{\sigma-t}} \right)\right|\leq \frac{|x-y|}{\sqrt{\sigma-t}}\sup_{s\in I}\left\{ 
    \frac{e^{-s^2/2}}{\sqrt{2\pi}}\right\} \, ,
\end{align*}
where we have defined the interval
$$I=\left(\frac{x-\int_t^\sigma \mu(u)du}{\sqrt{\sigma-t}},\frac{y-\int_t^\sigma \mu(u)du}{\sqrt{\sigma-t}}\right).$$
To bound the second teerm, we evaluate
\begin{align*}
\left| \partial_s \left[C\left( \frac{-\int_s^\sigma\mu(u)du}{\sqrt{\sigma-s}}\right) \right]\right| 
=
    &\left| \mu(s)\sqrt{\sigma-s}-\frac{\int_s^\sigma\mu(u) \, \mdd u}{2\sqrt{\sigma-s}}\right|\frac{e^{-\left( \frac{\int_s^\sigma\mu(u)du}{\sqrt{\sigma-s}}\right)^2}}{\sqrt{2\pi}(\sigma-s)} \leq \frac{\mu_M}{\sqrt{2\pi}\sqrt{\sigma-s}} \, ,
\end{align*} 
  where we have used that $\mu(s)\leq \mu_M$. 
   
In order to apply Lemma~\ref{ineqgrownwall} on the interval $[0,a]$, let us introduce the following bound 
$$
\frac{M_a}{2} = \sup_{\sigma\leq a}\sup_{s\in I}\left\{ 
    \frac{e^{-s^2/2}}{\sqrt{2\pi}}\right\}=\frac{1}{\sqrt{2\pi}}\exp\left[-
    \left(0 \vee  \left(\frac{x\wedge y-\int_t^a \mu(u)du}{\sqrt{a-t}}\right)
    \right)^2\right] \, ,
$$
which clearly satisfies $M_a<\sqrt{2}/\sqrt{\pi}$.
Plugging the various upper bounds obtained above in \eqref{mastercol},  for all $\sigma \leq a$, we obtain the following inequality
$$
|H(t,\sigma,x)-H(t,\sigma,y)|\leq \frac{|x-y|}{\sqrt{\sigma-t}}M_a+\frac{\sqrt{2}\mu_M}{\sqrt{\pi}}\int_t^\sigma \frac{|H(t,s,x)-H(t,s,y)|}{\sqrt{\sigma-s}} \, \mdd s\, ,
$$
which allows on to apply Lemma~\ref{ineqgrownwall} for  the function $f:[0,a-t]\rightarrow [0,\infty)$ with 
$$
f(z)=|H(t,t+z,x)-H(t,t+z,y)| \, ,
$$
Setting $z=a-t$, this leads to
$$
|H(t,a,x)-H(t,a,y)|\leq |x-y|M_a\left(\frac{1}{\sqrt{a-t}}+ \sqrt{2\pi}\mu_M\right)(4\mu_M^2(a-t)e^{2\mu_M^2(a-t)}+1) \, ,
$$
so for $a-t\leq 1/(8\mu_M^2)$, we get

\begin{equation}\label{ineq:fdfds}
   \big |H(t,a,x)-H(t,a,y) \big| \leq \frac{2(1/2e^{1/4}+1)}{\sqrt{a-t}}M_a|x-y|\leq \frac{4M_a}{\sqrt{a-t}}|x-y|
\end{equation}
This establishes $(i)$. 

To prove $(ii)$, notice that $\int_t^a \mu(u) \, \mdd u\leq \mu_M(a-t)$ so that
\begin{align*}
    x-\int_t^a \mu(u) \, \mdd u \geq x-\mu_M(a-t)\geq 2\sqrt{c}(a-t)^{1/4} \, ,
\end{align*}
and therefore, we have
\begin{align*}
    M_a&=\sqrt{\frac{2}{\pi}}e^{-
    \left(\frac{x-\int_t^a \mu(u) \, \mdd u}{\sqrt{a-t}}
    \right)^2} \leq e^{-\frac{4c}{\sqrt{a-t}}}\leq \frac{\sqrt{a-t}}{4c} \, .
\end{align*}
Plugging the above inequality into \eqref{ineq:fdfds} gives the desired result.

\end{proof}

\begin{proof}[Proof of Lemma \ref{lem:Hprop}]
Let us first notice that given the definition of the drift $\mu$ given in \eqref{eq:TCdrift} we have
\begin{align*}
\int_\tau^\sigma \mu_b(\xi) \, \mdd \xi 
&= -\left( \nu_1 - \frac{\lambda_1}{\lambda_2} \nu_2\right) (\Psi_b(\sigma)-\Psi_b(\tau)) +  \frac{\lambda_1}{\lambda_2} (\sigma - \tau) \, , \\
&= -\left( \nu_1 - \frac{\lambda_1}{\lambda_2} \nu_2\right) (\Psi_b(\sigma)-\Psi_a(\sigma) + \Psi_a(\tau)-\Psi_b(\tau)) +  \int_\tau^\sigma \mu_a(\xi) \, \mdd \xi  \, .
\end{align*}
Therefore, setting $C=2 \vert \nu_1 - \nu_2 \lambda_1/\lambda_2 \vert$, the condition $\Vert \Psi_a-\Psi_b \Vert_{0,\sigma} \leq \beta$ implies
that for all $0\leq \tau\leq \sigma$
\begin{align*}
 (x-C \beta) -\int_\tau^\sigma \mu_a(\xi) \mdd \xi \leq x-\int_\tau^\sigma \mu_b(\xi) \mdd \xi \leq (x+C \beta) -\int_\tau^\sigma \mu_a(\xi) \, \mdd \xi \, .
\end{align*}
The above chain of inequalities  indicates that given the same driving Wiener process $W$,
the first passage times  
\begin{align*}
\sigma_{b,\tau,x} 
&= \inf \left\{ \sigma \leq \tau \, \Big \vert \, W_\tau=x , \, , W_\sigma < \int_\tau^\sigma \mu_b(\xi) \, \mdd \xi \right\} \, , \\
\sigma_{a,\tau,x-C\beta} 
&= \inf \left\{ \sigma \leq \tau \, \Big \vert \, W_\tau=x-C\beta , \, , W_\sigma < \int_\tau^\sigma \mu_a(\xi) \, \mdd \xi \right\} \, , \\
\sigma_{a,\tau,x+C\beta} 
&= \inf \left\{ \sigma \leq \tau \, \Big \vert \, W_\tau=x+C\beta , \, , W_\sigma < \int_\tau^\sigma \mu_a(\xi) \, \mdd \xi \right\} \, , 
\end{align*}
are path-wise ordered in the sense that $\sigma_{a,\tau,x-C\beta} \leq \sigma_{b,\tau,x} \leq \sigma_{a,\tau,x+C\beta}$.
As $H_b(\sigma, \tau , x)= \Prob{\sigma_{b,\tau,x}}$ by definition, this directly implies that 
\begin{align*}
\Prob{\sigma_{a,\tau,x+C\beta}} = H_a(\sigma, \tau , x+C\beta) \leq   H_b(\sigma, \tau , x) \leq H_a(\sigma, \tau , x- C\beta)= \Prob{\sigma_{a,\tau,x-C\beta} }\, .
\end{align*}
This proves \eqref{eq:lemHprop1}.
Observe that the above inequalities remain valid when $x-C\beta<0$ with the convention that $H(\sigma,\tau,x)=1$ if $x<0$.

To justify \eqref{eq:lemHprop2}, let us assume with no loss of generality that $\tau_a<\tau_b$ and introduce $\Psi_c(\sigma)=\Psi_a(\sigma+\tau_a-\tau_b)$.
Then by the definition, we have $H_a(\sigma,\tau_a,x)=H_c(\sigma+\tau_b-\tau_a,\tau_b,x)$  so that
\begin{align*}
H_a(\sigma,\tau_a,x)-H_a(\sigma,\tau_b,x)
&=H_c(\sigma,\tau_b,x)-H_a(\sigma,\tau_b,x) \, + \\
& \hspace{40pt} (H_c(\sigma+\tau_b-\tau_a,\tau_b,x)- H_c(\sigma,\tau_b,x)) \, , \\
&=H_c(\sigma,\tau_b,x)-H_a(\sigma,\tau_b,x)+\int_\sigma^{\sigma+\tau_b-\tau_a}h_c(\xi,\tau_b,x)\, \mdd \xi \, , \\
&=H_c(\sigma,\tau_b,x)-H_a(\sigma,\tau_b,x)+\int_{\sigma-\tau_b+\tau_a}^{\sigma}h_a(\xi,\tau_a,x)\, \mdd \xi \, .
\end{align*}
Then, \eqref{eq:lemHprop2} follows from applying \eqref{eq:lemHprop1} to the difference
$H_c(\sigma,\tau_b,x)-H_a(\sigma,\tau_b,x)$ and using the fact that 
$\Vert \Psi_c-\Psi_a\Vert \leq  \vert \tau_2-\tau_1\vert /\nu_2$ as $\Psi_a \in \mId$.

\end{proof}


\subsection{Extraction argument proofs}\label{sec:buffextract}


\begin{proof}[Proof of Lemma \ref{lem:Hconv}]
Consider $\xi>0$ and set $\beta_n=\Vert \Psi_n-\Psi \Vert_{0,\xi}$ for convenience.
By Lemma \ref{lem:Hprop}, \eqref{eq:lemHprop1}, for all $x > 0$, we have
\begin{align*}
\vert H_n(\sigma,\tau,x)-H(\sigma,\tau,x) \vert
\leq 
&  \big[ H_n(\sigma,\tau;x)-H(\sigma,\tau,x+C\beta_n) \big] \\& \vee  \big[ H_n(\sigma,\tau,x-C\beta_n) - H(\sigma,\tau,x)  \big] \nonumber \, ,
\end{align*}
where $C=2|\nu_1-\lambda_1\nu_2/\lambda_1|$. 
Given $x>0$, we can always assume $n$ to be large enough so that $x/2>C\beta_n$. 
Then, there exists $C_0$ such that for $\sigma-\tau\leq C_0$, we have $x/2\geq (\sigma-\tau)^{1/3}+\mu_M(\sigma-\tau)$. 
As a result, the hypothesis for Lemma \ref{lem:FP},  \eqref{eq:lemFP2}, to hold are satisfied if $\sigma-\tau\leq C_1=C_0\wedge1/(8\mu_M^2)$, 
and we have:
$$
\vert H_n(\sigma,\tau,x-C\beta_n) - H(\sigma,\tau,x) \vert\leq C\sqrt{\sigma-\tau}\beta_n\leq C\sqrt{C_1}\beta_n
$$
Let us then examine the case  $\sigma-\tau> C_1$.
In this case, following the proof of Lemma~\ref{lem:FP}, we have the bound: 
\begin{align*}
\vert H_n(\sigma,\tau,x-C\beta_n) &- H(\sigma,\tau,x) \vert \\
&\leq C\beta_n\left(\frac{1}{\sqrt{\sigma-\tau}}+\sqrt{2\pi}\mu_M\right)\left(4\mu^2_M(\sigma-\tau)e^{2\mu^2_M(\sigma-\tau)}+1\right)\\
&\leq C\left(1/\sqrt{C_1}+\sqrt{2\pi}\mu_M \right)\left(4\mu^2_M\xi e^{2\mu^2_M\xi}+1 \right)\beta_n
\end{align*}
In any case, we see that for every $\sigma,\tau,$ with $0\leq\tau\leq\sigma\leq\xi$ 
$$
\vert H_n(\sigma,\tau,x-C\beta_n) - H(\sigma,\tau,x) \vert\leq C_2\beta_n
$$
where $C_2=C\sqrt{C_1}\vee C\left(1/\sqrt{C_1}+\sqrt{2\pi}\mu_M \right)\left(4\mu^2_M\xi e^{2\mu^2_M\xi}+1 \right)$.
This establishes the desired result.
\end{proof}

\begin{proof}[Proof of Lemma \ref{lem:HLipsch}]
    Consider $\tau_1<\tau_2$ we have by the second part of Lemma~\ref{lem:Hprop} that
    \begin{align}\label{eq:HLip1}
        \vert H(\sigma,\tau_1,x)-H(\sigma,\tau_2,x) \vert \leq  
        & \big(  \big[ H(\sigma,\tau_2,x)-H(\sigma,\tau_2, x+C\vert \tau_2-\tau_1 \vert /\nu_2 )  \big] \\
        & \vee   \big[ H(\sigma,\tau_2, x-C \vert \tau_2-\tau_1 \vert /\nu_2 ) - H(\sigma,\tau_2,x)  \big] \big) \nonumber\\
        &+ \Vert h(\cdot,\tau_1,x)\Vert_{\tau_1, \sigma} \vert \tau_2-\tau_1 \vert  \, , \nonumber
    \end{align}
    where $C=2|\nu_1-\lambda_1\nu_2/\lambda_1|$. We can assume $|\tau_1-\tau_2|\leq \nu_2 x/(2C)$. Then, there exists $C_0$ such that for $\sigma-\tau\leq C_0$, we have $x/2\geq (\sigma-\tau)^{1/3}+\mu_M(\sigma-\tau)$. 
As a result, the hypothesis for Lemma \ref{lem:FP},  \eqref{eq:lemFP2} are satisfied if $\sigma-\tau\leq C_1=C_0\wedge1/(8\mu_M^2)$, 
and we have:
$$
\vert H(\sigma,\tau,x-C \vert \tau_2-\tau_1 \vert /\nu_2 ) - H(\sigma,\tau,x) \vert\leq \sqrt{\sigma-\tau}C \vert \tau_2-\tau_1 \vert /\nu_2 \leq \sqrt{C_1}C \vert \tau_2-\tau_1 \vert /\nu_2 
$$
Let us then examine the case  $\sigma-\tau> C_1$.
In this case, following the proof of Lemma~\ref{lem:FP}, we have the bound: 
\begin{align*}
\vert H(\sigma,\tau,x-&C \vert \tau_2-\tau_1 \vert /\nu_2 ) - H(\sigma,\tau,x) \vert \\
&\leq C \vert \tau_2-\tau_1 \vert /\nu_2 \left(\frac{1}{\sqrt{\sigma-\tau}}+\sqrt{2\pi}\mu_M\right)\left(4\mu^2_M(\sigma-\tau)e^{2\mu^2_M(\sigma-\tau)}+1\right) \, ,\\
&\leq C/\nu_2\left(1/\sqrt{C_1}+\sqrt{2\pi}\mu_M \right)\left(4\mu^2_M\xi e^{2\mu^2_M\xi}+1 \right)\vert \tau_2-\tau_1 \vert \, .
\end{align*}
This shows that the term in parenthesis in the right-hand side of \eqref{eq:HLip1}
admits a Lipschitz bound that does not depend on $\Psi \in \mIz$.
The final result follows from Proposition \ref{prop:boundh}, which shows that $\Vert h(\cdot,\tau_1,x)\Vert_{\tau_1, \sigma} \vert$
is uniformly bounded with respect to $\Psi \in \mIz$.

\end{proof}

\bibliographystyle{imsart-number}

\end{document}